\documentclass[11pt,a4paper]{article}

\usepackage{geometry}
\usepackage{cite}

\usepackage[utf8]{inputenc}
\usepackage[T1]{fontenc}
\usepackage{lmodern}

\usepackage{amsmath,amsthm,amstext,amscd,amssymb,euscript,mathrsfs}
\usepackage{calrsfs}
\usepackage{epsfig}
\usepackage{mathtools}
\usepackage[draft]{todonotes}
\usepackage{hyperref} 
\usepackage{csquotes}
\usepackage[english]{babel}
\usepackage{verbatim} 
\usepackage{bbm}
\usepackage{color}
	\definecolor{dgreen}{RGB}{0,190,0}
\usepackage{wrapfig} 

\usepackage[normalem]{ulem} 

\newcommand{\SortNoop}[1]{}

\usepackage{enumitem}
\setenumerate{label={\normalfont(\alph*)},topsep=4pt,itemsep=0pt} 

\theoremstyle{plain}
\newtheorem{theorem}{Theorem}[section]
\newtheorem{lemma}[theorem]{Lemma}
\newtheorem{proposition}[theorem]{Proposition}

\theoremstyle{definition}
\newtheorem{definition}[theorem]{Definition}

\newtheorem{remark}[theorem]{Remark}

\newtheorem{assumption}[theorem]{Assumption}

\newcommand{\downq}{\includegraphics[height=8pt]{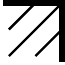}}
\newcommand{\upq}{\includegraphics[height=8pt]{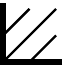}}

\newcommand{\Z}{\mathbb Z}
\newcommand{\R}{\mathbb R}
\newcommand{\N}{\mathbb N}
\newcommand{\K}{\mathbb K}

\renewcommand{\phi}{\varphi}

\def\1{{\mathchoice {\rm 1\mskip-4mu l} {\rm 1\mskip-4mu l}
{\rm 1\mskip-4.5mu l} {\rm 1\mskip-5mu l}}}

\newcommand{\cA}{\mathcal{A}}

\newcommand{\cC}{\mathcal{C}}

\newcommand{\cG}{\mathcal{G}}

\newcommand{\cI}{\mathcal{I}}
\newcommand{\cJ}{\mathcal{J}}

\newcommand{\cL}{\mathcal{L}}

\newcommand{\fL}{\mathfrak{L}}

\newcommand{\fS}{\mathfrak{S}}

\newcommand{\bK}{\mathbb{K}}

\newcommand{\bP}{\mathbb{P}}

\newcommand{\bR}{\mathbb{R}}

\newcommand{\bfE}{\mathbf{E}}

\newcommand{\limn}{\lim_{n\rightarrow \infty}}

\newcommand{\liminfn}{\liminf_{n\rightarrow \infty}}

\newcommand{\dd}{ \mathrm{d}}
\newcommand{\bONE}{\mathbbm{1}}

\DeclareMathOperator*{\argmin}{argmin}

\newcommand{\ip}[2]{\langle #1,#2\rangle}

\DeclareMathOperator{\arctanh}{arctanh}

\newenvironment{colorwil}
    {\color[rgb]{0,0.5,0}    }

\newcommand{\bwil}{\begin{colorwil}}
\newcommand{\ewil}{\end{colorwil}}
\newcommand{\towil}{\todo[color=green!50]}

\newenvironment{calculations}
    {\color{orange}
    }
    { 
    }

\usepackage{environ}
\RenewEnviron{calculations}{}  

\begin{document}

\renewcommand{\thefootnote}{\Roman{footnote}}

\title{A Hamilton-Jacobi point of view on mean-field Gibbs-non-Gibbs transitions}
\author{
\renewcommand{\thefootnote}{\Roman{footnote}}
Richard C. Kraaij
\footnotemark[1]
\\
\renewcommand{\thefootnote}{\Roman{footnote}}
Frank Redig
\footnotemark[2]
\\
\renewcommand{\thefootnote}{\Roman{footnote}}
Willem B. van Zuijlen
\footnotemark[3]
}

\footnotetext[1]{
Delft Institute of Applied Mathematics,
Technische Universiteit Delft, 
Delft, the Netherlands,  
r.c.kraaij@tudelft.nl. 
}
\footnotetext[2]{
Delft Institute of Applied Mathematics,
Technische Universiteit Delft, 
Delft, the Netherlands,  
f.h.j.redig@tudelft.nl. 
}
\footnotetext[3]{
Weierstrass Institute, Berlin, Germany, 
vanzuijlen@wias-berlin.de. 
}

%

\maketitle

\renewcommand{\thefootnote}{\arabic{footnote}} 

\begin{abstract}

We study the loss, recovery, and preservation of differentiability of time-dependent large deviation rate functions. This study is motivated by mean-field Gibbs-non-Gibbs transitions. 
	
The gradient of the rate-function evolves according to a Hamiltonian flow. This Hamiltonian flow is used to analyze the regularity of the  time-dependent rate function, both for Glauber dynamics for the Curie-Weiss model and Brownian  dynamics in a potential. We extend the variational approach to this problem of time-dependent regularity in order to include Hamiltonian trajectories with a finite lifetime in closed domains with a boundary. 
This leads to new phenomena, such a recovery of smoothness.

We hereby create a new and unifying approach for the study of mean-field Gibbs-non-Gibbs transitions, based on Hamiltonian dynamics and viscosity solutions of Hamilton-Jacobi equations.

\end{abstract}





\section{Introduction}

	The large deviation approach to dynamical Gibbs-non-Gibbs transitions, initiated in van Enter, Fern\'andez, den Hollander and Redig \cite{EFHR10}, characterizes the emergence of `bad configurations' via the non-uniqueness of optimal starting configurations corresponding to a given arrival configuration. `Bad configurations' have to be interpreted as points of essential discontinuity of conditional probabilities and `optimal' has to be interpreted here in the sense of minimizing a large deviation cost,
	\begin{equation} \label{eqn:intro_rate_function}
	I(\gamma) = I_0(\gamma(0)) + \int_0^\infty \cL(\gamma(s),\dot{\gamma}(s)) \, \dd s,
	\end{equation}
	 which is the sum of an initial cost  $I_0$  corresponding to the starting measure and a path-space cost  in the form of a Lagrangian action. 
	In the mean-field context one considers trajectories of the magnetization and the dynamical Gibbs-non-Gibbs transitions are to be interpreted in the sense of Gibbsianness for mean-field models, a notion introduced in K\"ulske and le Ny \cite{KN07} and studied in  \cite{EK10,FHM13,HRZ15,KiKu19,KuMe20b}. Inspired from the literature on dynamical Gibbs-non-Gibbs transitions it is natural to conjecture that for a large variety of models the following three statements are equivalent: 		
	\begin{enumerate}
		\item Mean-field Gibbsianness  at time $t$;
		\item Unique optimal trajectories: 
		for all arrival points $b$ at time $t$ the optimal trajectory arriving at $b$ is unique, i.e., 
		\begin{equation*}
		\argmin_{\gamma, \gamma(t) = b} I_0(\gamma(0)) + \int_0^t \cL(\gamma(s),\dot{\gamma}(s)) \dd s
		\end{equation*}
		is a singleton;
		\item Differentiability of the rate function  at time $t$, i.e., $I_t$ given by 
		\begin{equation} \label{eqn:intro_rate_It}
		I_t(b) := \inf_{\gamma, \gamma(t) = b} I_0(\gamma(0)) + \int_0^t \cL(\gamma(s),\dot{\gamma}(s)) \dd s,
		\end{equation}
		 is differentiable as a function of $b$. 
	\end{enumerate} 
	
	In this paper, we prove, for a broad class of models in the one-dimensional setting, the equivalence of (b) and (c) and introduce  new methods within the framework of calculus of variations to  investigate the differentiability of the rate function at time $t$. The proofs are based on techniques from the theory of calculus of variations.  Under certain conditions the equivalence of (b) and (c) has been proved, however, our setting is fundamentally different as will be explained below.

By using this general approach, we do not use  specific  information  of the considered models  and therefore our methods are applicable in a large variety of models.  This is in contrast to \cite{KN07,EK10,FHM13,HRZ15,KiKu19,KuMe20b}, whose approaches rely on explicit information about the specific models,  for which they treat the relation between (a) and (b) .

	 Let $H(x,p) = \sup_v pv - \cL(x,v)$ be the Hamiltonian corresponding to the Lagrangian in \eqref{eqn:intro_rate_function}. Following classical mechanics, if the characteristics of the Hamilton-Jacobi equation
	\begin{equation} \label{eqn:intro_HJ}
	\partial_t u(t,x) + H(x,\partial_x u(t,x)) = 0, \qquad  u(0,x) = I_0(x),
	\end{equation}
	do not intersect, then  $u(t,x) := I_t(x)$, with $I_t$ defined in \eqref{eqn:intro_rate_It},  is continuously differentiable  and a classical solution of \eqref{eqn:intro_HJ}. Following the theory of calculus of variations,  even  if the characteristics intersect, then  $u(t,x)=I_t(x)$ still  solves \eqref{eqn:intro_HJ} as a \textit{viscosity solution}. In addition, one can show that $u$ is \textit{locally semi-concave}.
	
	The characteristics  of \eqref{eqn:intro_HJ} are exactly the Hamilton trajectories, i.e., they  solve the Hamilton equations
	\begin{equation}\label{eqn:intro_hamilton_equations}
	\begin{bmatrix}
	\dot{X}(s) \\
	\dot{P}(s)
	\end{bmatrix}
	=
	\begin{bmatrix}
	\partial_p H(X(s),P(s)) \\
	-\partial_x H(X(s),P(s))
	\end{bmatrix}.
	\end{equation}
	 Moreover,  every optimal trajectory as in (b) above has an associated characteristic.  A rigorous analysis shows that these observations  can be turned into a proof that (b) is equivalent to (c). In addition, a study of solutions to \eqref{eqn:intro_hamilton_equations} can be used to prove or disprove that $I_t$ is differentiable.

	\smallskip
	
	We proceed with giving three techniques based on the analysis of Hamiltonian flows that either guarantee differentiability or non-differentiability of $I_t$: 
	order preservingness, linearization and rotation.  
	We illustrate these methods on the natural examples of Glauber dynamics for the Curie-Weiss model and that of mean-field interacting Brownian particles in a single- or double-well potential. Phenomena such as short-time conservation, loss and recovery of Gibbsianness (differentiability) for both high- and low- temperature dynamics, previously obtained in specific examples only, are now obtained in a broad context.

	\smallskip
	
\textbf{Comparison with existing literature on calculus of variations.}	
One of the earlier works dealing with the equivalence between (b) and (c) is \cite{Fl69}. The theory of viscosity solutions, their regularity and Hamiltonian trajectories is a well established theory, we refer to the works \cite{Li82}, \cite{FlSo89}, \cite{LiYo95}, \cite{BaCD97}, \cite{CaSi04} and \cite{FlSo06}. 
	However, these works do not apply in our context. Indeed, in the references mentioned above the spaces are open and all Hamilton trajectories have an infinite life time (they stay at all times in the open set) and the Lagrangian is finite and continuously differentiable on the closure of the space. 
	A major difficulty that we have to overcome is that the mean-field models which we consider have Hamiltonians for which the solutions to \eqref{eqn:intro_hamilton_equations} have a finite time of existence, after which they arrive at `points at infinity'. Also the Lagrangian is not finite on the boundary of the space. 
	These problems are not merely technical: they also lead to new phenomena, such as `recovery of differentiability' which in the Gibbs-non-Gibbs literature corresponds to `recovery of Gibbsianness' (which for measures on the lattice has been shown in \cite{EFHR02} and for infinite-temperature mean-field dynamics in \cite{FHM13}).

	\smallskip

\textbf{Overview.}
The rest of our paper is organized as follows. 

In Section \ref{section:preliminaries_theoretical_main_results} we introduce Gibbs-non-Gibbs transitions, path-space large deviations, examples of models giving rise to Hamiltonians that fall within our framework and some definitions and preleminaries from the theory of calculus of variations. 
Then we give our main result on the equivalence between (b) and (c), the relation between uniqueness of optimisers and the regularity of the rate function and finally we present a relation between this regularity and the push-forward of the graph of the derivative of the initial rate function.

In Section \ref{section:applications_of_analyzing_Hamiltonian_flow} we prescribe conditions under which the regularity is preserved or broken and apply this to obtain different scenarios for the models introduced in Section \ref{section:preliminaries_theoretical_main_results}. 
One of the scenarios treated here it the one of recovery, as mentioned before. 

Techniques and proofs for theorems of Section \ref{section:preliminaries_theoretical_main_results} can be found in 
Section \ref{section:calculus_of_variations},  \ref{section:proofs_regularity_rate_function} and \ref{section:topological_properties_of_Gt}. 

Techniques and proofs for theorems of Section \ref{section:applications_of_analyzing_Hamiltonian_flow} can be found in 
Section 
\ref{section:abstract_convexity_preserving}, 
\ref{section:creation_of_overhangs_proofs}
 and  \ref{section:verification_explicit_example}. 





\textbf{Acknowledgements.} 
The authors are grateful to C. Külske and A.C.M. van Rooij for discussions.

RK is supported by The Netherlands Organisation for Scientific Research (NWO), grant number 600.065.130.12N109 and the Deutsche Forschungsgemeinschaft (DFG) via RTG 2131 High-dimensional Phenomena in Probability – Fluctuations and Discontinuity. 

WvZ is supported by the German Science Foundation (DFG) via the Forschergruppe FOR2402 "Rough paths, stochastic partial differential equations and related topics".

\section{Preliminaries and theoretical main results} \label{section:preliminaries_theoretical_main_results}

\subsection{ Mean-field Gibbs measures }\label{section:gibbs_non_gibbs}

In this paper,  as the initial model at time $t=0$ we consider two mean-field models.  In the next two sections we will describe the dynamics to which these initial models will be subjected. Namely, we consider one with spins that attain values in $\R$, which we refer to as the $\R$-space-model, and one with spins that attain values in $\{-1,1\}$, which we refer to as the $\pm 1$-space-model. We will write $\K$  for the space in which empirical averages will take their values, in particular we have that $\K$  equals $\R$  for the $\R$-space-model, and $[-1,1]$ for the $\pm 1$-space-model. 

We start in both cases from an initial measure $\mu_{N,0}$ of the form
\begin{align}
\label{eqn:curie_weiss_measure}
\mu_{N,0}( \dd \sigma_1, \cdots, \dd \sigma_N) = \frac{e^{-NV(m_N(\sigma_1,\dots, \sigma_N)) }}{Z_N} \lambda^N ( \dd \sigma_1, \cdots, \dd \sigma_N),
\end{align}
where
\begin{align}
m_N(\sigma_1,\dots,\sigma_N) = \frac1N \sum_{i=1}^N \sigma_i,
\end{align}
$\lambda^N$ is the  $N$-fold  product of $\lambda$, $Z_N$ the normalizing constant, and
\begin{enumerate}[label=(\roman*)]
\item for the $\R$-space-model, $V: \R \rightarrow [0,\infty)$ is continuous and $\lambda$ is a standard normal distribution on $\R$.
\item for the $\pm 1$-space-model, $V: [-1,1]\rightarrow \R$ is continuous and $\lambda$ is the uniform measure on $\{-1,1\}$.
\end{enumerate}
The ``potential'' $V$ determines in both cases uniquely the rate function for the large deviation principle of the magnetization $m_N$ under $\mu_{N,0}$, which is the function 
\begin{align}
\label{eqn:rate_function_initial_CW}
x \mapsto V(x)+ i(x)- \inf_{x\in \K } (V(x)+i(x)),
\end{align}
where
\begin{enumerate}[label=(\roman*)]
\item for the $\R$-space-model, $i(x) = \frac12 x^2$,
\item for the $\pm 1$-space-model, $i(x)= \frac{1-x}{2}\log(1-x) +\frac{1+x}{2}\log(1+x)$. 
\end{enumerate}
We consider the spins to evolve according to the following dynamics
\begin{enumerate}[label=(\roman*)]
\item for the $\R$-space-model; interacting diffusions as described in Section \ref{subsection:interacting_diffusions}.
\item for the $\pm 1$-space-model; Glauber dynamics as described in Section \ref{subsection:glauber}.
\end{enumerate}

The initial measure $\mu_{N,0}$ is transformed by the dynamics to the measure $\mu_{N,t}$ at time $t>0$.

\begin{definition}
\label{def:seq_Gibbs}
Let $t\ge 0$.
$\alpha \in \K$ is called a
\emph{good magnetization for} $(\mu_{N,t})_{N\in\N}$ if there exists a probability measure $\gamma_t(\cdot|\alpha)$ such that \footnote{In case $\K =\R$, the conditional measure on the left-hand-side of \eqref{eqn:seq_gibbs_conv} has to be understood in terms of weakly continuous regular conditional probabilities as is done in \cite{HRZ15}.}
\begin{align}
\mu_{N,t}( \dd \sigma_1 | \sigma_2^N, \dots, \sigma_N^N ) \xrightarrow{weakly} \gamma_t ( \dd \sigma_1 | \alpha),
\label{eqn:seq_gibbs_conv}
\end{align}
for all $\sigma_2^N,\dots, \sigma_N^N$ such that $m_{N-1}(\sigma_2^N, \dots, \sigma_N^N) \rightarrow \alpha$. If $\alpha$ is not a good magnetization, it is called a \emph{bad magnetization}.

The sequence $(\mu_{N,t})_{N\in\N}$ is called \emph{sequentially Gibbs} if
$\alpha$ is a \emph{good magnetisation}
for all $\alpha \in \K$.

If $(\mu_{n,t})_{N\in\N}$ is sequentially Gibbs, then $\alpha \mapsto \gamma_t(\cdot | \alpha)$ is weakly continuous (see \cite[Theorem 3.A.1]{vZ16} or \cite[Lemma 1.3]{HRZ15}).
\end{definition}

\begin{remark}
The definition of sequentially Gibbs follows those in 
\cite{KN07},
\cite{EK10},
\cite{FHM13},
\cite{HRZ15}. We refer to \cite{KN07} for the explanation of the definition with regard to Gibbs measures on the lattice. 
\end{remark}

%

\subsection{Glauber dynamics}
\label{subsection:glauber}

In this section, we describe the dynamics for the $\pm 1$-space-model. For each $N$, we consider a continuous-time Markov process $(X_1(t),\cdots, X_N(t)) \in \{-1,1\}^N$ of mean-field interacting spins with mean-field jump rates $c_N$. The law of $(X_1(0),\cdots, X_N(0))$ is $\mu_{N,0}$ and the Markov generators of these spin-flip systems are of the form
\begin{equation*}
\cA_N f(\sigma_1,\dots,\sigma_N) := \sum_{i=1}^N c_N(\sigma_i,m_N(\sigma)) \left[f(\sigma^i) - f(\sigma) \right],
\end{equation*}
where $c_N \ge 0$
and where the configuration $\sigma^i \in \{-1,1\}^N$ is given by
\begin{equation*}
\sigma^i_j = \begin{cases}
- \sigma_j & \text{if } i = j, \\
\sigma_j & \text{if } i \neq j.
\end{cases}
\end{equation*}

We denote by $M_N(t) := \frac1N \sum_{i =1}^N X_i(t)$ the \emph{empirical magnetization at time $t$}. Due to the mean-field character of this dynamics, also the dynamics of the empirical magnetization is Markovian, and an elementary computation shows that the generator of the process $(M_N(t))_{t\ge 0}$ on $m_N(\{-1,1\}^N) \subseteq [-1,1]$ is given by
\begin{align} \label{generator_of_dynamics}
A_Nf(x) =
\notag
& \ \  N \frac{1-x}{2}
 c_N(-1,x) 
\left[f(x + 2N^{-1}) - f(x) \right] \\
& + N \frac{1+x}{2}
 c_N(+1,x) 
\left[f(x - 2N^{-1}) - f(x) \right],
\end{align}
 as it satisfies $\cA_N (f \circ m_N) = (A_N f) \circ m_N$. 
For later purposes, we assume the following.

\begin{assumption} \label{assumption:jump_rates}
There exist functions 
$v_+,v_- : [-1,1] \rightarrow [0,\infty)$ such that 
\begin{equation}
\label{eqn:convergence_c_N}
\lim_{N \rightarrow \infty} \sup_{x \in m_N(\{-1,1\}^N)}\left|\frac{1-x}{2}
c_N(-1,x)
 - v_+(x) \right| + \left|\frac{1+x}{2}
 c_N(+1,x)
  - v_-(x) \right| =0,
\end{equation}
for which the following properties hold:
\begin{enumerate}
\item $v_-(-1) = 0$, $v_-(x) > 0$ for $x \neq 1$, and $v_+(1) = 0$ and $v_+(x) > 0$ for $x \neq 1$,
\item 
$v_-,v_+$ have an extensions to an open set $V\subseteq \R$ that contains $[-1,1]$ and these extensions are twice continuously differentiable,
\item $v_+'(1) < 0$ and $v_-'(-1) > 0$. 
\end{enumerate}
\end{assumption}

In concrete examples, we consider $v_-,v_+$ of the form 
\begin{equation}
\label{eqn:common_v_-_and_v_+}
v_-(x) = 
\frac{1+x}{2} e^{-\beta x- h}, \qquad v_+(x) = 
\frac{1-x}{2} e^{\beta x+ h},
\end{equation} 
which  correspond to  the rates obtained from Glauber spin-flip dynamics  reversible with respect to the Curie-Weiss measure in \eqref{eqn:curie_weiss_measure}  at inverse temperature $\beta \geq 0$ and external magnetic field $h \in \bR$, i.e., for $V(x) = - \beta x^2 - hx$ .

\subsection{Interacting diffusion processes}
\label{subsection:interacting_diffusions}

In this section, we describe the dynamics for the $\R$-space-model. For each $N$, we consider $N$ mean-field interacting diffusions $(X_1(t),\dots,X_N(t)) \in \bR^N$ in a potential landscape $W_N : \bR \rightarrow \bR$, where $W_N$ is continuously differentiable. We assume that $-W_N'$ is one-sided Lipschitz: there is some $M \geq 0$ such that for all $x>y$ 
\begin{equation*}
-(W_N'(x) - W_N'(y)) \leq M (x-y).
\end{equation*}
The law of  $(X_1(0),\dots,X_N(0))$ is given by $\mu_{N,0}$ and the dynamics are given by
\begin{equation*}
	\dd X_i(t) = -W'_N(M_N(t)) \dd t + \dd B_i(t)
\end{equation*}
where $M_N(t) := \frac1N \sum_{i =1}^N X_i(t)$ is the {empirical magnetization at time $t$} as above and where $B_1, \dots, B_N$ are independent standard Brownian motions. Note that there exists a unique solution to this stochastic differential equation by \cite[Proposition 3.38]{PaRa14} and the one-sided Lipschitz property of $-W_N'$. The empirical magnetization is also Markovian and satisfies
\begin{equation*}
	\dd M_N(t) = -W'_N(M_N(t)) \dd t + \tfrac{1}{\sqrt{N}} \dd B_1(t).
\end{equation*}
Again by \cite[Proposition 3.38]{PaRa14} this equation has a unique solution, and additionally, its generator $A_N$ with domain $C_b^2(\bR)$ is given by 
\begin{equation*}
A_N f(x) = - W_N'(x)f'(x) + \tfrac{1}{2N} f''(x).
\end{equation*}

Also in this case we have the following assumption. 

\begin{assumption} \label{assumption:diffusion_drift}
We assume that there is some three times continuously differentiable function $W: \bR \rightarrow \bR$ for which  $-W'$ is one-sided Lipschitz and such that for every compact set $K \subseteq \bR$ 
\begin{equation}
\label{eqn:convergence_W_N}
	\lim_{N \rightarrow \infty} \sup_{x \in K}\left|W_N'(x) - W'(x)\right| =0.
\end{equation}
\end{assumption}

$W(x) = \sum_{i=1}^{2k} a_i x^i$ with $a_i \in \R$ and $a_{2k} > 0$ is an example of such a function for which $- W'$ is one-sided Lipschitz. 
In the examples that we will consider, we will use $W(x) = \frac{1}{4}x^4 - \frac{1}{2}dx^2$ with $d \in \bR$.
 This function is strictly convex for $d\le 0$ (`high temperature') and has the shape of a double well for $d>0$ (`low temperature').

\subsection{Path-space large deviations}

In various works, see e.g. \cite{Co89,FW98,DPdH96,Le95,FK06,Kr16b,CoKr17}, it has been shown that if the initial magnetization $M_N(0)$ satisfies a large deviation principle with rate function $I_0$, then the Markov process $t \mapsto M_N(t)$ satisfies the large deviation principle on\footnote{$D_{\K}([0,\infty))$ is the Skorohod space of càdlàg paths $[0,\infty) \rightarrow  \K$, see  also \cite[Section 3.5]{EK86}.} 
 $D_{\K}([0,\infty))$, i.e., 
\begin{equation*}
	\bP\left[ (M_N(t)) _{t \geq 0} \approx \gamma \right] \approx e^{-N \cI(\gamma)},
\end{equation*}
with rate function
\begin{equation}
\label{eqn:rate_function_paths}
	\cI(\gamma) = \begin{cases}
		I_0(\gamma(0)) + \int_0^\infty \cL(\gamma(s),\dot{\gamma}(s)) \dd s & \text{if } \gamma \in \cA\cC, \\
		\infty & \text{otherwise}.
	\end{cases}
\end{equation}
Here $\cA\cC$ is the space of absolutely continuous trajectories $\gamma : [0,\infty) \rightarrow \K$ and $\cL : \K\times \bR \rightarrow [0,\infty]$ is the Lagrangian
obtained by taking the Legendre transform
\begin{align*}
	\cL(x,v) := \sup_{p\in\R} \left( pv - H(x,p) \right),
\end{align*}
of the Hamiltonian $H : \K \times \bR \rightarrow \bR$ given 
\begin{enumerate}[label=(\roman*)]
	\item for the $\R$-space-model, by 
\begin{align}
\label{eqn:Hamiltonian_Rmodel}
H(x,p) = \frac{1}{2} p^2 - p W'(x).
\end{align}	
	\item for the $\pm 1$-space-model, by 
	\begin{equation}
	\label{eqn:Hamiltonian_CWmodel}
	H(x,p) = v_+(x) \left[e^{2p} - 1\right] + v_-(x) \left[e^{-2p} - 1\right].
	\end{equation}
\end{enumerate}
This Hamiltonian in turn is for example obtained by an operator approximation procedure introduced by Feng and Kurtz \cite{FK06}. This procedure is explained informally in Redig and Wang \cite{ReWa12} 
and rigorously for $\bR^d$-valued processes in Kraaij \cite{Kr16b} and Collet and Kraaij \cite{CoKr17}. 
$H$ is derived from an operator $\mathcal{H}$ by the relation $ H(x,f'(x))= \mathcal{H}f(x)$, where $\mathcal{H}$ satisfies $\mathcal{H}f = \lim_{N\rightarrow \infty} \mathcal{H}_N f $ where   $\mathcal{H}_N$ is the operator defined by $\mathcal{H}_N f = N^{-1}e^{-Nf} A_N e^{Nf}$, with $A_N$ as defined in \eqref{generator_of_dynamics}.


For any two points $a,b \in \bK$ and time $t$, denote by
\begin{equation}\label{miniaction}
S_t(a,b)= \inf_{\gamma \in \cA\cC: \gamma(0)=a, \gamma(t)=b} \int_0^t \cL(\gamma(s),\dot{\gamma}(s)) \dd s.
\end{equation}
 $S_t(a,b)$ is the minimal Lagrangian action of a trajectory starting at $a$ and arriving at time $t$ at $b$.  

By the contraction principle, the rate function for the large deviation principle for the magnetization $M_N(t)$ at time $t>0$ is given by (for $\cI$ see \eqref{eqn:rate_function_paths})
\begin{equation}
\label{borom}
I_t(b) = \inf_{x\in \K } (I_0(x) + S_t(x,b)) = 
\inf_{\gamma \in \cA\cC: \gamma(t) = b} \cI(\gamma).
\end{equation}


\begin{definition}
	\label{definition:optimal_paths_and_points}
	We call $\gamma\in \cA\cC$ an \emph{optimal trajectory for} $S_t(a,b)$ (see \eqref{miniaction}) if $\gamma(0)=a$, $\gamma(t)=b$ and  $\int_0^t \cL(\gamma(s),\dot{\gamma}(s)) \dd s = S_t(a,b)$. Analogously, $\gamma$ is called an \emph{optimal trajectory for} $I_t(b)$ (see  \eqref{borom}) if $\gamma(t) = b$ and $I_t(b) = I_0(\gamma(0)) + \int_0^\infty \cL(\gamma(s),\dot{\gamma}(s)) \dd s $.  Finally,  $x\in \K$ is called a \emph{optimal starting point for} $I_t(b)$ if $I_t(b)= I_0(x) + S_t(x,b)$.
\end{definition}

%
%
%
%
%
%


In the following definition we define another way to say that (c) of the conjecture  in the introduction does not hold.

\begin{definition}
	We will say that $\alpha \in \K^\circ$ is a point of \textit{non-differentiability} of $I_t$ when $I_t$  is not differentiable at $\alpha$. 
\end{definition}

%
%

	\subsection{Preliminaries from the theory of calculus of variations}

	We follow the route of studying the optimal trajectories and non-differentiabili\-ties in the rate function by introducing  techniques from the theory of calculus of variations.  The first observation from classical mechanics is that optimal trajectories are known to solve the second-order \textit{Euler-Lagrange} equation, which is the point of view taken in \cite{EK10}. On equal footing, it is known that dual variables satisfy the first-order Hamilton equations.

	\begin{definition}
		Let $t>0$. Let $A$ be any one of the intervals $[0,t], [0,t),(0,t]$ or $(0,t)$.
		
		Let $\gamma : A \rightarrow \K^\circ$ be absolutely continuous. 
		If $(\gamma(s),\dot\gamma(s))$ is in the domain where $\cL$ is $C^1$ for all $s\in A$, then the trajectory $\eta$  defined by 
		\begin{align}
		\label{eqn:dual_path_to_gamma}
		\eta(s) = \partial_v \cL(\gamma(s),\dot \gamma(s)),
		\end{align}
		is called the \emph{dual trajectory to} $\gamma$. 
		
		
		Let $\gamma \in C^1(A, \K )$ and $\eta \in C^1(A)$. 
		We say that $(\gamma,\eta)$ satisfies the Hamilton equations, if they solve
		\begin{align}
		\begin{bmatrix}
		\dot{\gamma}(s) \\
		\dot{\eta}(s)
		\end{bmatrix}
		=
		\begin{bmatrix}
		\partial_p H(\gamma(s),\eta(s)) \\
		-\partial_x H(\gamma(s),\eta(s))
		\end{bmatrix}.
		\label{eqn:hamilton_equations}
		\end{align}
		If $(\gamma,\eta)$ satisfies the Hamilton equations  and $\gamma \in C^1(A,\bK)$, then $\eta$ is the dual trajectory to $\gamma$ (see \cite[Corollary A.2.7, equation (A.28)]{CaSi04}). Moreover, there exists a $c\in \R$ such that $ H(\gamma(s),\eta(s))=c$ for all $s\in A$. 
	\end{definition}

In addition, we use the following definitions of Cannarsa and Sinestrari \cite{CaSi04}. Let $d\in\N$ and $A\subseteq \R^d$ be open. Let $v : A \rightarrow \R$.
\begin{description}
	\item[Superdifferential] \cite[Definition 3.1.1]{CaSi04} The \textit{superdifferential} of $v$ at $x \in A$ is defined as $D^+ v (x) := \{ p \in \R^d : \limsup_{y\rightarrow x} \frac{v(y) -v(x) - \ip{p}{y-x}}{|y-x|} \le 0\}$.  Similarly, we define a \textit{subgradient} $D^- v(x)$. If $v$ is differentiable at $x \in A$, we write $Dv(x)$ for the derivative of $v$ at $x$. Note that in that case $D^+ v(x) = D^- v(x) = \{Dv(x)\}$.
	\item[Reachable gradient] \cite[Definition 3.1.10]{CaSi04} Let $v$ be locally Lipschitz. A $p\in \R^d$ is called a \emph{reachable gradient} of $v$ at $x$ if there exists a sequence $(x_k)_{k\in\N}$ in $A\setminus\{x\}$ such that $v$ is differentiable at $x_k$ for all $k$ and $x_k \rightarrow x$, 
	$Dv(x_k) \rightarrow p$. We write $D^* v(x)$ for the set of all reachable gradients of $v$ at $x$.
	\item[Viscosity solutions] \cite[Definition 5.2.1]{CaSi04} Let $F \in C(A \times \bR \times \bR^d)$. Consider the equation 
	\begin{equation} \label{eqn:abstract_diff_equation}
	F(x,v,Dv) = 0.
	\end{equation}
	$v \in C(A)$ is called a \textit{viscosity subsolution} to \eqref{eqn:abstract_diff_equation} if for all $x \in B$, we have
	\begin{equation*}
	F(x,v(x),p) \leq 0, \qquad \mbox{ for all } \, p \in D^+ v(x),
	\end{equation*}
	$v \in C(A)$ is called a \textit{viscosity supersolution} to \eqref{eqn:abstract_diff_equation} if for all $x \in B$, we have
	\begin{equation*}
	F(x,v(x),p) \geq 0, \qquad \mbox{ for all } \, p \in D^- v(x).
	\end{equation*}
	$v$ is called a \textit{viscosity solution} to \eqref{eqn:abstract_diff_equation} it it is both a sub- and a supersolution.
	\item[ Local  semi-concavity] \cite[Definition 1.1.1 and Proposition 1.1.3]{CaSi04} Let $B$ be a subset of $\bR^d$. Let $K \subseteq B$ be compact. We say that $v$ is \textit{semi-concave} on $K$ if there is some $C \ge 0$ such that
	\begin{equation}
	\label{eqn:semi-concave}
	\lambda v(x) + (1-\lambda) v(y) - v\left(\lambda x + (1-\lambda)y\right) \leq C \frac{\lambda(1-\lambda)}{2} |x-y|^2
	\end{equation}
	for all $x,y \in K$ such that the line from $x$ to $y$ is contained in $K$ and for all $\lambda \in [0,1]$. We call $v$ \textit{locally semi-concave on $B$} if it is semi-concave on each compact set $K \subseteq B$. Note that in \cite{CaSi04} these functions are called semi-concave with linear modulus, to distinguish them from a broader class of semi-concave functions. We do not need this generality  here,  and therefore will call ``semi-concave functions with linear modulus'' simply ``semi-concave functions''. 
\end{description}

\begin{remark}
With $\Phi_2 V$ the second difference quotient of $V$ (see \cite[Section 1.2]{vRSc82} for a definition),  
	$\Phi_2$ describes the convexity and concavity properties of a function in the sense that $V$ is convex if and only if $\Phi_2 V \ge 0$ and $V$ is concave if and only if $\Phi_2 V \le 0$. 
	But it also relates to semi-concavity, as one has that if $V$ is continuous, then $V$ is semi-concave with constant $C>0$ if and only if $\Phi_2 V \le C$. In  \cite{HRZ15} the $\Phi_2 V$ is used to describe the Gibbsianity of the $\R$-space-model  with the dynamics of independent Brownian motions, see also Remark \ref{remark:comparing_with_HRZ15}.
\end{remark}

\subsection{Regularity of the rate-function} \label{section:regularity}

In this section we establish the announced equivalence between the uniqueness of optimizing trajectories and differentiability of the rate function (Theorem \ref{theorem:equivalences_differentiability}). The main issue, which distinguishes our problems from the ones considered in \cite{CaSi04}, is that the maximal time of existence  of solutions of the Hamilton equations, contrary to \cite{CaSi04}, may be finite. This causes certain divergence of the momentum at the boundary of $\K$. 
To extend the techniques to our setting, we will work under the assumptions in Assumption \ref{assumption:Hamiltonian_for_domain_extension}.

In our setting it is natural  to start with rate functions whose superdifferential is close to $-\infty$ near the left boundary and close to $\infty$ near the right boundary of $\bK$ as this property is preserved for $I_t$ (see Theorem \ref{theorem:preserviation_of_semi_concavity} and  \eqref{eqn:infinity_boundary_conditions}). 
Moreover, we assume our initial rate function to be $C^1$. The space of functions that combines these two properties is called $C^{1,\partial}$ (see Definition \ref{definition:C_1_with_boundary}). 

The main result (Theorem \ref{theorem:equivalences_differentiability}) then shows that such a rate function under the time evolution is again in $C^{1,\partial}(\K)$ if and only if there is a unique optimizing trajectory. In Proposition \ref{proposition:solution_HJ_equation} we state that the function $(t,x)\mapsto I_t(x)$ is the viscosity solution of the Hamilton-Jacobi equation.

\begin{definition}
	\label{definition:C_1_with_boundary}
	For $\K = \R$ we write $\partial_- = -\infty$ and $\partial_+ = \infty$, while for $\K = [-1,1]$ we write $\partial_- = -1$ and $\partial_+ = 1$. 
	
	We write $C^{k,\partial}(\K)$ for the set of functions $g: \K \rightarrow  \R$ such that $g$ is $k$ times continuously differentiable on $\K^\circ$, continuous on $\K$ and 
	\begin{align}
	\label{eqn:infinity_boundary_conditions}
	\lim_{x \rightarrow \partial_+} g' (x) = \infty, \qquad \lim_{x \rightarrow \partial_-} g'(x) = -\infty.
	\end{align}
\end{definition}	
Note that for $V\in C^1[-1,1]$ the function \eqref{eqn:rate_function_initial_CW} is an element of $C^{1,\partial}[-1,1]$. Moreover, note that $I_0 \in C^{1,\partial}(\K)$ implies that $I_0$ is bounded from below and has compact sublevel sets. 

\begin{assumption}
	\label{assumption:general_ones_on_R_and_CW}
	We will assume that $I_0 \in C^{1,\partial}(\K)$ and for  $H: \K \times \R \rightarrow \R$ we assume
	\begin{itemize}
		\item considering the $\R$-space-model that (a) or (b) is satisfied:
		\begin{enumerate}
			\item $H$ is of the form \eqref{eqn:Hamiltonian_Rmodel}, where $W$ satisfies Assumption \ref{assumption:diffusion_drift}.
			\item $H$ satisfies  Assumptions \ref{assumption:Hamiltonian_for_domain_extension} and \ref{assumption:compact_level_sets_R} below.
		\end{enumerate}
		\item 
		considering the $\pm 1$-space-model that (a) or (b) is satisfied:
		\begin{enumerate}
			\item $H$ is of the form \eqref{eqn:Hamiltonian_CWmodel}, where $v_+, v_-$ satisfy Assumption \ref{assumption:jump_rates}.
			\item $H$ satisfies Assumptions \ref{assumption:Hamiltonian_for_domain_extension} and \ref{assumption:trajectories_in_interior} below. 
		\end{enumerate}
	\end{itemize}
\end{assumption}
The examples that we consider satisfy condition (a), however the proofs of the theory are based  on the more general condition (b). In Appendix \ref{appendix:assumptions_for_models} we show that (a) indeed implies (b).

\begin{theorem}[Lemma \ref{lemma:classically_semi_concave} and \ref{lemma:exploding_derivatives_at_boundary}] \label{theorem:preserviation_of_semi_concavity}

Assume Assumption \ref{assumption:general_ones_on_R_and_CW}. 
Then $I_t$ is locally semi-concave on $\K^\circ$ for all $t \geq 0$. Moreover, $(t,x) \mapsto I_t(x)$ is  locally semi-concave on $(0,\infty) \times \K^\circ$ and
	\begin{align*}
	\sup_{a \in  (0,\partial_+)} \inf_{b \ge a} \inf D^+ I_t(b) = \infty,
	\qquad
	\inf_{ a \in  (\partial_-,0)} \sup_{b \le a}  \sup D^+ I_t(b) = - \infty.
	\end{align*}
\end{theorem}

\begin{lemma}
\label{lemma:differentiability_and_subgraphs_linked_for_loc_semiconcave}
 \cite[Theorem 3.3.4 and 3.3.6]{CaSi04}
Let $d\in \N$ and $A\subseteq \R^d$ be open. 
Let $v : A \rightarrow \R$ be locally semiconcave and $x\in A$. Then $D^+ v(x) \neq \emptyset$ and
\begin{enumerate}
\item $D^* v(x) \subseteq \partial D^+v(x)$.
\item $D^+ v(x) = \text{co}D^*v(x)$, 
where $\text{co}(S)$ denotes the closed convex hull of a set $S \subseteq \bR^d$. 
\item $D^+v(x)$ is a singleton if and only if $v$ is differentiable at $x$.
\item If $D^+ v(y)$ is a singleton for every $y\in A$, then $v\in C^1(A)$.
\end{enumerate}
\end{lemma}

\begin{theorem}
\label{theorem:equivalences_differentiability}
Assume Assumption \ref{assumption:general_ones_on_R_and_CW}. Let $u : [0,\infty) \times \K \rightarrow [0,\infty]$ be given by 
\begin{align*}
u(t,x) = I_t(x).
\end{align*}
Let $b \in \K^\circ $, $t>0$. 
The following are equivalent.
\begin{enumerate}
\item
\label{item:unique_optimal_point}
$x\mapsto I_0(x) + S_t(x,b)$ has a unique optimizer ($I_t(b)$ has a unique optimal point).
\item
\label{item:unique_optimal_path}
$\gamma \mapsto I_0(\gamma(0)) + \int_0^\infty \cL(\gamma(s),\dot{\gamma}(s)) \dd s$ has an unique optimal trajectory with $\gamma(t) = b$.
\item 
\label{item:I_t_diffb_at_b}
$I_t$ is differentiable at $b$.
\item 
\label{item:u_diffb_at_t,b}
$u$ is differentiable at $(t,b)$, 
\item 
\label{item:reachable_grad_singleton}
$D^*u(t,b)$ is a singleton,
\end{enumerate}
Moreover, the following are equivalent.
\begin{enumerate}[label={\normalfont(\Alph*)}]
\item 
\label{item:unique_optimal_point_all_ends}
$x\mapsto I_0(x) + S_t(x,b)$ has a unique optimizer for all $b \in \K^\circ $.
\item 
\label{item:unique_optimal_path_all_ends}
$\gamma \mapsto I_0(\gamma(0)) + \int_0^t \cL(\gamma(s),\dot{\gamma}(s)) \dd s$ has an unique optimal trajectory with $\gamma(t) = b$ for all $b \in \K^\circ $.  
\item 
\label{item:I_t_diffb}
$I_t$ is differentiable on $\K^\circ $.
\item 
\label{item:I_t_in_C_1_partial}
$I_t \in C^{1,\partial}(\K)$.
\end{enumerate}
\end{theorem}

\begin{proof}
\ref{item:unique_optimal_path} $\Rightarrow$ \ref{item:unique_optimal_point}.
If $\gamma$ is an optimal trajectory, then $\gamma(0)$ is an optimal point for $I_t(b)$.

\ref{item:unique_optimal_point} $\Rightarrow$ \ref{item:unique_optimal_path}. Suppose $x$ is a unique optimal starting point for $I_t(b)$, and $\gamma_1,\gamma_2$ are  optimal trajectories for $I_t(b)$ that start in $x$. To both trajectories, we can associate dual trajectories $\eta_1,\eta_2$ so that the pairs $(\gamma_1,\eta_1)$ and $(\gamma_2,\eta_2)$ solve the Hamilton equations. Because their starting points are the same, the starting condition implies that $\eta_i(0) = I_0'(x)$. This, however, implies that the Hamilton equations are initialized with the same starting data, implying that $\gamma_1 = \gamma_2$. We conclude there is a unique optimal trajectory.

By Proposition \ref{proposition:push_forward_graph_contains_reachable_grads} 
the optimal trajectories for $I_t(b)$ are in one to one correspondence with the elements of $D^* u(t,b)$.
Hence \ref{item:unique_optimal_path} $\iff$ \ref{item:reachable_grad_singleton}. 

Theorem \ref{theorem:preserviation_of_semi_concavity} implies that $u$ and $I_t$ are locally semi-concave. 
By Lemma \ref{lemma:differentiability_and_subgraphs_linked_for_loc_semiconcave} we have
\ref{item:u_diffb_at_t,b} $\iff$ \ref{item:reachable_grad_singleton}. 

\ref{item:u_diffb_at_t,b} $\Rightarrow$ \ref{item:I_t_diffb_at_b} as differentiability of $u$ at $(t,b)$ implies differentiability of $I_t$ at $b$.

On the other hand, if there exist two distinct optimal trajectories for $I_t(b)$, then the corresponding end momenta are different. By Lemma \ref{lemma:differentiability_and_subgraphs_linked_for_loc_semiconcave} and \cite[Theorems 6.4.8]{CaSi04} this implies that $D^+ I_t(b)$ consists of at least two elements, i.e., $I_t$ is not differentiable at $b$ by Lemma \ref{lemma:differentiability_and_subgraphs_linked_for_loc_semiconcave}. Hence \ref{item:I_t_diffb_at_b} $\Rightarrow$ \ref{item:unique_optimal_path}. 

\ref{item:unique_optimal_point_all_ends} $\iff$ \ref{item:unique_optimal_path_all_ends} $\iff$ \ref{item:I_t_diffb} follow from \ref{item:unique_optimal_point} $\iff$ \ref{item:unique_optimal_path} $\iff$ \ref{item:I_t_diffb_at_b}. 

As $I_t$ is locally semi-concave, $I_t$ is differentiable if and only if $I_t \in C^1$ on $\K^\circ$ (see Lemma \ref{lemma:differentiability_and_subgraphs_linked_for_loc_semiconcave}). Hence \ref{item:I_t_diffb} $\iff$ \ref{item:I_t_in_C_1_partial} by Theorem \ref{theorem:preserviation_of_semi_concavity}. 
\end{proof}

\begin{proposition} \label{proposition:solution_HJ_equation}
Assume assumption \ref{assumption:general_ones_on_R_and_CW}.  Then $u$ (as in Theorem \ref{theorem:equivalences_differentiability}) is a viscosity solution to the Hamilton-Jacobi equation
	\begin{equation}
	\label{eqn:hamilton_jacobi}
	\partial_t u(t,x) + H\left(x, \partial_x u(t,x)  \right)=0
	\end{equation} 
	on $\K^\circ \times (0,\infty)$.
\end{proposition}

\subsection{Regularity via the push-forward of the graph of \texorpdfstring{$I_0'$}{} } \label{section:regularity_via_pushforward}

Our next step is to relate optimal trajectories to the solutions of the Hamilton equations. In the following definition we present the push-forward under the Hamiltonian flow of the graph of $I_0'$. In Proposition \ref{proposition:graph_of_derivative_pushforward}, we show that if this push-forward at time $t$ is a graph then $I_t$ is differentiable. Additionally, we show that overhangs in the push-forward are indications for existence of points of non-differentiability.

\begin{definition}
Assume Assumption \ref{assumption:general_ones_on_R_and_CW}.
\begin{enumerate}
\item We write
\begin{align*}
\cG:=\{(x, I_0'(x)) : x\in \K^\circ\},
\end{align*}
for the graph of the derivative of $I_0$.
\item For all $(x,p)\in \K \times \bR$ let 
 $(X^{x,p}_t, P^{x,p}_t)$  be the solution of the Hamilton equations \eqref{eqn:hamilton_equations} with initial conditions  $(X_0^{x,p},P_0^{x,p})=(x,p)$  and up to the maximal time of existence $t_{x,p}$. 
 (See \cite[Theorem 2.4.1]{Pe01}, which by Assumption \ref{assumption:Hamiltonian_for_domain_extension}(a) can also be applied in the case that $\K$ equals $[-1,1]$.)
\item For all $t>0$ we define the \emph{push-forward of} $\cG$ to be the set 
\begin{equation*}
\cG_t := \left\{(X^{x,p}_t, P^{x,p}_t) : (x,p) \in \cG, \, t < t_{x,p} \right\}.
\end{equation*}
\item 		Fix $t > 0$. We say that $\cG_t$ has an \textit{overhang} at $x \in \K^\circ$, if there exist $y_1, y_2 \in \bR$ with $y_1 \neq y_2$ such that $(x,y_1),(x,y_2) \in \cG_t$. 
Hence if $\cG_t$ has an overhang, then it is not a graph (of a function). 
\end{enumerate}
\end{definition}

\begin{proposition} \label{proposition:graph_of_derivative_pushforward}
Assume Assumption \ref{assumption:general_ones_on_R_and_CW}. 
Fix $t>0$. 
Then 
\begin{equation}
\left\{(x,p) : x \in \K^\circ, \, p \in D^* I_t(x)\right\} \subseteq \cG_t.
\label{eqn:graph_inclusion}
\end{equation}
Consequently, 
if $\cG_t$ has no overhang at $x$, then $I_t$ is differentiable at $x$. 

On the other hand, if there are $x_1, x_2 \in \K^\circ$, $x_1 < x_2$ such that $\cG_t$ has no overhang at $x_1$ and $x_2$, then 
	\begin{enumerate}
		\item 
		\label{item:I_t_diffb_implies_no_overhang}
		$I_t$ is differentiable on $[x_1,x_2]$ $\Longrightarrow$ $\cG_t$ has no overhang at $x$ for all $x \in (x_1,x_2)$.
		\item 
		\label{item:overhang_implies_non_diffb_I_t}		
		$\cG_t$ has an overhang at some $x \in (x_1,x_2)$  $\Longrightarrow$ $I_t$ is not differentiable on $[x_1,x_2]$  
	\end{enumerate}
\end{proposition}
\begin{proof}
\eqref{eqn:graph_inclusion} is proved in 
Proposition \ref{proposition:push_forward_graph_contains_reachable_grads}(c). 
If $\cG_t$ has no overhang at $x$, then $D^*I_t(x)$ is a singleton  and so $I_t$ is differentiable at $x$ by Theorem \ref{theorem:equivalences_differentiability}.  

We prove \ref{item:I_t_diffb_implies_no_overhang} as \ref{item:overhang_implies_non_diffb_I_t} is equivalent to \ref{item:I_t_diffb_implies_no_overhang}. 
The interval $(a,b)$ as in Proposition \ref{proposition:hamiltonian_paths_and_graph_structures}\ref{item:interval_connecting_boundaries} is such that $\Phi_t(a,b)$ contains $(x_i,I_t'(x_i))$ for $i\in \{1,2\}$. As $\Phi_t(a,b)$ is connected, even $\{(x,I_t'(x)): x\in [x_1,x_2] \} \subseteq \Phi_t(a,b)$. 
By continuity and injectivity of $\Phi_t$, see Proposition \ref{proposition:hamiltonian_paths_and_graph_structures}\ref{item:E_t_is_open}, 
and the fact that $\{y :  (x_i ,y) \in \cG_t\}$ are singletons for $i \in \{1,2\}$, 
 it follows  for all $x \in (x_1,x_2)$ that $\{y \in \bR : (x,y) \in \cG_t\} = \{y \in \bR : (x,y) \in \Phi_t(a,b)\} $ and that this set cannot contain multiple elements.
\end{proof}

%
%

\section{Applications of analyzing the Hamiltonian flow} \label{section:applications_of_analyzing_Hamiltonian_flow}

In this section we study the differentiability of $I_t$ by analysing the push-forward $\cG_t$. 

We will give a description of the results of this section in terms of the informal -but more familiar- notions \textit{high}- and \textit{low-temperature}. 
We say that a rate function $I_0$ is \emph{high-temperature} if it is strictly convex, whereas we say that $I_0$ is \emph{low-temperature} if it has at least two strict local minima. 
We say that the dynamics is \emph{high}- and \emph{low-temperature} if there is a high-  and  low-temperature $I_0$ such that $H(x,I_0'(x)) = 0$, respectively.  This means that $I_0$ is the rate function of the stationary distributions of the dynamics with Hamiltonian $H$.  

\smallskip

In Section \ref{section:preserve_lose_differentiability} we give two general results on the preservation of differentiability, and two general results on the creation of overhangs:
\begin{description}
	\item  [Preservation at high temperature] 
	We show that for certain types of high-temperature dynamics, combined with  high-temperature starting rate functions, we have differentiability of $I_t$ for all $t \geq 0$.
	\item  [Short-time preservation] 
	We show that `order preserving' behaviour of the dynamics close to the boundary, combined with a starting rate-function $I_0$ that is strictly convex close to the boundary, implies short-time differentiability of $I_t$. 

	\item  [Large-time loss, heating] 
	We consider linearizations of the Hamiltonian flow around stationary points,  which  in combination with low-temperature rate-functions, create overhangs for sufficiently large times.
	\item  [Large-time loss, cooling] 
	We considers areas in phase-space where the energy is negative and where the Hamiltonian flow `rotates'. If the graph of the gradient of the rate function $I_0$ crosses this `rotating' region, an overhang is created for sufficiently large times. Rotating regions occur for low-temperature Hamiltonians.  
\end{description}

We proceed in Section \ref{section:explicit_results_for_examples} by applying these results to two sets of well known examples: Glauber dynamics for the Curie-Weiss model and interacting diffusions in a potential. 

\smallskip

Before starting with the various applications, we introduce some notation.

\begin{definition} \label{def:quadrants_CW}
	For $(y,q) \in \K^\circ \times \bR$, we write 
		\begin{equation*}
		\upq_{y,q} := \{(x,p) \in \K \times \R : x \ge y, p \ge q\} , \qquad \downq_{y,q} := \{(x,p) \in \K \times \R : x \le y, p \le q\},
		\end{equation*}
				and $\upq_{y,q}^\circ, \downq_{y,q}^\circ$ for the interiors of $\upq_{y,q}, \downq_{y,q}$ in $\K \times \bR$. 
				
E.g., for $\K = [-1,1]$, $\upq_{y,q} = [y,1] \times [q, \infty)$ and $\upq^\circ_{y,q} = (y,1] \times (q,\infty)$. 
\end{definition}

\begin{definition} \label{definitions:H_preservation_convexity_I_convexity}
Assume Assumption \ref{assumption:general_ones_on_R_and_CW}.
\begin{enumerate}
	\item Let $A \subseteq \bK \times \R$. We say that $H$ \textit{preserves} $A$, or $A$ is \textit{preserved} under $H$, if all trajectories $(\gamma,\eta)$ satisfying the Hamilton equations with $(\gamma(0),\eta(0)) \in A$ stay in $A$ during their life-time, i.e., $(\gamma(s),\eta(s)) \in A$ for all $s < t_{\gamma(0),\eta(0)}$ if $(\gamma(0),\eta(0)) \in A$.
	\item If $(x_0,p_0) \in \K \times \R$ is such that $H$ preserves the set $\{(x_0,p_0)\}$, then we call $(x_0,p_0)$ \emph{stationary} under $H$. Mostly we  consider stationary points of the form $(x_0,0)$, in such case we will also say that $x_0$ is \emph{stationary}. 
	\item 
	\label{item:preserves_order}
	We say that $H : \K \times \bR \rightarrow \bR$ \emph{preserves order} on a subset $A \subseteq \K \times \bR$ if $A$ is preserved under $H$ and
	if for $(x_1,p_1), (x_2,p_2) \in \K \times \R$ and $t < t_{x_1,p_1} \wedge t_{x_2,p_2}$, 
	\begin{align*}
	x_1 <  x_2, \ p_1  <  p_2 \quad \Longrightarrow \quad 
	X_t^{x_1,p_1}  <  X_t^{x_2,p_2} , \ 	P_t^{x_1,p_1}  <  P_t^{x_2,p_2}. 
	\end{align*}
	\item 
	We say that $H$ \textit{preserves order at infinity} if there are $(y_-,q_-), (y_+,q_+) \in \K^\circ \times \bR$ such that $H$ preserves order on $\downq_{y_-,q_-}$ and $\upq_{y_+,q_+}$. 
	\item We say that $I_0$ is \textit{strictly convex at infinity} if there is some compact set $K \subseteq \bK^\circ$ such that $I_0'$ is strictly increasing on the complement of $K$. 
\end{enumerate}
\end{definition}

%
%

Two effective methods to determine that $\cG_t$ has overhangs are via rotating areas in the Hamiltonian flow and via linearizations of the flow,  of which the definition follows.


\begin{definition}
	\label{definition:linearization}
	Assume Assumption \ref{assumption:general_ones_on_R_and_CW}. 
	Let $x_0\in  \K^\circ$ be a stationary point for the Hamiltonian flow, i.e., $\partial_p H(x_0,0) = 0$. We say that the Hamiltonian flow admits a \emph{$C^1$ linearization} at $(x_0,0)$ if there exist open neighbourhoods $U,V\subseteq \R^2$ of $(x_0,0)$, and a $C^1$-diffeomorphism $\Psi : U \rightarrow V$ such that $\Psi(x_0,0)=(x_0,0)$ and $D\Psi(x_0,0) = \bONE$ (the identity matrix) and  such  that a trajectory $(\gamma,\eta)$ with values in $U$ solves the Hamilton equations \eqref{eqn:hamilton_equations} if and only if $(\xi, \zeta) =\Psi(\gamma,\eta)$ solves
	\begin{align}
	\label{eqn:linearised_system}
	\begin{bmatrix}
	\dot{\xi}(s) \\
	\dot{\zeta}(s)
	\end{bmatrix}
	=
	\nabla^2 H(x_0,0) 
	\begin{bmatrix}
	\xi(s) - x_0 \\
	\zeta(s)
	\end{bmatrix},
	\end{align}
	where $\nabla^2 H(x_0,0)$ denotes the Hessian of $H$ at $(x_0,0)$, so that with
	\begin{align}
	\label{eqn:definition_m_and_c}
	m:= \partial_x \partial_p H(x_0,0), \qquad c:= \partial_p^2 H (x_0,0),
	\end{align}
	it has the form
	\begin{align*}
	\nabla^2 H(x_0,0) 
	= 
	\begin{bmatrix}
	\partial_x \partial_p H(x_0,0) & \partial_p^2 H(x_0,0) \\ -\partial_x^2 H(x_0,0) & - \partial_x \partial_p H(x_0,0)
	\end{bmatrix} 
	= 
	\begin{bmatrix}
	m & c \\ 0 & -m
	\end{bmatrix},
	\end{align*}
	note that $- \partial_x^2 H(x_0,0) = 0$ because $H(x,0) = 0$ for all $x$ and that $c>0$. 
\end{definition}

	\begin{remark} \label{remark:value_of_Jacobi_linearization}
		Note that we include in the definition that $D \Psi(x_0,0) = \bONE$ which is not a standard assumption. However, we need this assumption to connect the dynamics of the graph $\cG$ under the Hamiltonian flow, to that of the dynamics of the tangent of the push-forward of $\cG$.
	\end{remark}

\subsection{Preservation of differentiability and the creation of overhangs} \label{section:intro_short_time_differentiability} \label{section:preserve_lose_differentiability}

The following theorem relates {preservation of order under $H$ and strict convexity of $I_0$ to differentiability of $I_t$. The proof can be found in Section \ref{section:abstract_convexity_preserving}.

\begin{theorem} \label{theorem:convexity_preserving_implies_diffb_preserv}
Assume Assumption \ref{assumption:general_ones_on_R_and_CW}. 
\begin{enumerate}
\item
\label{item:conv_implies_diffb}
Suppose that $I_0$ is strictly convex and that $\cG$ is contained in a set on which $H$ preserves order. 
	Then $I_t \in C^{1,\partial}(\K)$ for all $t \geq 0$. 
\item 
\label{item:conv_at_infinity_implies_short_time}
Assume that $I_0$ is $C^2$ on $\K^\circ$,  strictly convex at infinity 
and that $H$ preserves order at infinity. 
Then there is a $t_0 > 0$ such that $I_t \in C^{1,\partial}(\bK)$ for all $0 \leq t \leq t_0$.
\end{enumerate}
\end{theorem}

The next theorem gives conditions under which overhangs are created. 
The proofs of (a) and (b) can be found in Section \ref{section:creation_of_overhangs_proofs}.

\begin{theorem} \label{theorem:loss_of_differentiability}
	Assume Assumption \ref{assumption:general_ones_on_R_and_CW}. 
	\begin{enumerate}
	
		\item 
		\label{item:loss_linearization} 		
		Suppose $x_0 \in \bK^\circ$ is such that $I_0'(x_0) = 0$ and that $(x_0,0)$ is a stationary point. 
Let $m,c$ be as in \eqref{eqn:definition_m_and_c}. 
Suppose $I_0$ is $C^2$ in a neighbourhood of $x_0$ with $I_0''(x_0) < \min\{- \frac{2m}{c}, 0\}$.
		In addition, assume that the Hamiltonian flow admits a $C^1$ linearization at $(x_0,0)$ (a sufficient condition for this is that $H$ is $C^\infty$ and $m\ne 0$\footnote{See Theorem \ref{theorem:existence_linearization}.}).
		Let  
			\begin{equation} \label{eqn:time_for_vertical_in_linearized_sytem}
			t_0 := \begin{cases}
			- \frac{1}{2m} \log \left(1+\left(\frac{c}{2m}I_0''(x_0)\right)^{-1} \right) & \text{if } m \neq 0, \\
			-\frac{1}{c I_0''(x_0)} & \text{if } m = 0.
			\end{cases}
			\end{equation}		
		Then $\cG_t$ contains an overhang at $x_0$ for all $t > t_0$.
		
				\item 
		\label{item:loss_rotating}
			Suppose that $H$ is $C^3$. In addition suppose that $m_1, m_2 \in \bK^\circ$ are two points such that $m_1 < m_2$ and
			\begin{enumerate}[label=\textnormal{(\roman*)}]
				\item 
				\label{item:partial_p_H_in_m_1_m_2_and_inbetween}
				$\partial_p H(m_1,0) = 0 = \partial_p H(m_2,0)$ and $\partial_p H(x,0) \neq 0$ for all $x \in (m_1,m_2)$,
				\item 
				\label{item:partial_x_partial_p_H_in_m_1_and_m_2}				
				$\partial_x \partial_p H(m_1,0) \neq 0$ and $\partial_x \partial_p H(m_2,0) \neq 0$.
%
			\end{enumerate}
			Suppose that $\cG \cap \{(x,p) \in (m_1,m_2) \times \bR  : H(x,p) < 0 \} \neq \emptyset$. 
			Then there is a time $t_0 > 0$ such that there is a $x_0 \in (m_1,m_2)$ such that $\cG_t$ contains an overhang at $x_0$ for all $t \geq t_0$.
	\end{enumerate}
\end{theorem}

\begin{remark}
	In \ref{item:loss_linearization} $t_0$ is the time that the line $x \mapsto x (1,I_0''(x_0))$ is transformed into a vertical line under the linearized flow.
	
	Regarding \ref{item:loss_rotating}, in Section \ref{section:creation_of_overhangs_proofs} we show that if a set $A$ satisfies certain properties, then the Hamiltonian flow rotates over the boundary $\partial A$ of $A$, which means that every Hamilton path started on $\partial A$ will move along $\partial A$ and return to its initial point in finite time. 
	
	We prove that the conditions of \ref{item:loss_rotating} imply the existence of such a set in Lemma \ref{lemma:regular_rotating_region}.
\end{remark}

\subsection{Explicit results for two main classes of examples} \label{section:explicit_results_for_examples}

In the present section, we   analyze particular canonical instances of the $\pm 1$-space-model and the $\R$-space-model: 

\begin{enumerate}[label=(\roman*)]
\item for the $\R$-space-model we consider $H$ and $I$ of the form
\begin{equation} \label{eqn:H_and_I_0_for_R}
H(x,p) = \frac{1}{2}p^2 - pW_b'(x), \qquad	I_0(x) = W_a(x),
\end{equation}
for $a,b\in \R$, where $W_d(x) = \frac{1}{4}x^4 - \frac{1}{2} d x^2$ for $d\in \R$. 
\item for the $\pm 1$-space-model we consider $H$ and $I$ of the form 
\begin{align}
\label{eqn:H_for_CW}
H(x,p) & = \frac{1-x}{2} e^{\beta x+ h} \left[ e^{2p} -1 \right] + \frac{1+x}{2} e^{-\beta x - h} \left[ e^{-2p} -1 \right], \\
\label{eqn:I_0_for_CW}
I_0(x) & = \frac{1-x}{2} \log (1-x) + \frac{1+x}{2} \log (1+x) - \frac{1}{2}\alpha x^2 - \theta x + C, 
\end{align} 
where $\alpha,\beta \ge 0$, $\theta,h\in \R$ and $C\in\R$ is such that $\inf_{x\in [-1,1]} I_0(x) =0$. 
I.e., we consider Curie-Weiss dynamics with inverse temperature $\beta$ and a starting rate function that corresponds to an inverse temperature $\alpha$. 
\end{enumerate}

In Proposition \ref{proposition:convex_at_infty_explict_functions} we consider conditions sufficient for $H$ to preserve order. In Theorem \ref{theorem:application_explicit} we  apply the results of Theorems \ref{theorem:convexity_preserving_implies_diffb_preserv} and \ref{theorem:loss_of_differentiability} for various  choices of $a$ and $b$ or of $\alpha,\beta, h$ and $\theta$. In Theorem \ref{theorem:loss_recovery_differentiability}, we show that for specific choices of $\alpha$ and $\theta$ one obtains recovery of differentiability. We restrict our analysis mostly to $\beta = h =0$, i.e., infinite-temperature dynamics and zero dynamical external magnetic field.  
For other parameters
we expect that 
one can obtain 
similar results 
at the expense of more involved computations, 
which go beyond the scope of this text. 

In the proof of Theorem \ref{theorem:application_explicit} and Theorem \ref{theorem:loss_recovery_differentiability}, we use the creation of overhangs proved in Theorem \ref{theorem:loss_of_differentiability} to show non-differentiability with the help of Proposition \ref{proposition:graph_of_derivative_pushforward}.


\begin{remark}
\label{remark:derivatives_of_H_and_I_0}
One can rewrite $H$ as in \eqref{eqn:H_for_CW} 
and compute the following derivatives, which will be used later in proofs. We spare the reader the computations. 
\begin{align}
\label{eqn:H_in_hyperbolic_functions}
	H(x,p) & = 2 \sinh(p) \left[ \sinh(\beta x +h + p) - x \cosh(\beta x +h + p) \right], \\
	\label{eqn:partial_p_H_in_hyperbolic_functions}
	\partial_p H(x,p) 
	& =    2 \sinh( \beta x + h + 2p) -   2 x \cosh(\beta x + h + 2p) 
	\\
\label{eqn:partial_x_H_in_hyperbolic_functions}
	\partial_x H(x,p) &=
- 2 \sinh(p)
\left[ (1- \beta) \cosh(\beta x + h + p) + \beta x \sinh(\beta x +h+  p) \right], \\
\label{eqn:partial_p_partial_x_H_in_hyperbolic_functions}
\partial_p \partial_x H(x,p) 
& =     2( \beta - 1)  \cosh( \beta x + h + 2p) -  2\beta x \sinh(\beta x + h + 2p), \\
\label{eqn:partial_x_sq_H_in_hyperbolic_functions}
	\partial_x^2 H(x,p) &= 
	 - 2 \beta \sinh(p) \left[  (2-\beta) \sinh(\beta x + h + p) +  \beta x  \cosh(\beta x + h +  p) \right], \\
	 \label{eqn:partial_p_sq_H_in_hyperbolic_functions}
	\partial_p^2 H(x,p) &=   4 \cosh( \beta x + h + 2p) -  4 x \sinh(\beta x + h + 2p). 
\end{align}

Note that $H(x,p) =0$ if and only if either $p=0$ or $h + p = \arctanh(x)- \beta x$ and that $\partial_p H(x,p)=0$ if and only if $2p = \arctanh(x)- \beta x - h$. 

For $I_0$ as in \eqref{eqn:I_0_for_CW} we have 
\begin{align}
\label{eqn:derivative_I_0}
I_0'(x) & 
=  \arctanh (x) - \alpha x - \theta, \\
\label{eqn:second_derivative_I_0}
I_0''(x) & = \frac{1}{1-x^2} - \alpha. 
\end{align}
\end{remark}

\begin{calculations}
\begin{align*}
& 2H(x,p) = 
(1-x) e^{\beta x + h} (e^{2p}-1) + (1+x) e^{-\beta x - h} (e^{-2p}-1) \\
& = e^{\beta x + h} (e^{2p}-1) +e^{-\beta x - h} (e^{-2p}-1) 
-x \left( e^{\beta x + h} (e^{2p}-1) - e^{-\beta x - h} (e^{-2p}-1) \right) \\
& = e^{\beta x + h+p} (e^{p}-e^{-p}) +e^{-\beta x - h-p} (e^{-p}-e^{p}) 
-x \left( e^{\beta x + h+p} (e^{p}-e^{-p}) - e^{-\beta x - h-p} (e^{-p}-e^{p}) \right) \\
& = (e^{p}-e^{-p}) \left(  e^{\beta x + h+p}  - e^{-\beta x - h-p}  \right)
-x (e^{p}-e^{-p}) \left( e^{\beta x +h+p}  + e^{-\beta x - h-p} \right) \\
& = 4 \sinh(p) \left( \sinh(\beta x +h + p ) - x \cosh(\beta x +h + p) \right) 
\end{align*}

\begin{align*}
& 2 \partial_x \left[\sinh(\beta x + h+ p ) - x \cosh(\beta x + h + p)\right]\\
&= \partial_x \left[  e^{\beta x + h+p}  - e^{-\beta x - h-p} - x ( e^{\beta x + h+p}  + e^{-\beta x - h-p} ) \right]\\
&=  \left[ \beta e^{\beta x + h+p}  + \beta e^{-\beta x - h-p} -  ( e^{\beta x + h+p}  + e^{-\beta x - h-p} ) - x ( \beta e^{\beta x + h+p}  - \beta e^{-\beta x - h-p} )\right]\\
& = 2\beta \cosh(\beta x + h+ p) - 2\cosh( \beta x + h + p ) -2 \beta x  \sinh(\beta x + h + p)  \\
& =-2 \left( (1- \beta) \cosh(\beta x + h+ p) + \beta x \sinh(\beta x + h+ p) \right) \\
&  \partial_x^2 \left[\sinh(\beta x + h + p ) - x \cosh(\beta x + h+ p)\right]\\
& = - \left( \beta (1- \beta) \sinh(\beta x + h + p) + \beta  \sinh(\beta x + h + p)  + \beta^2 x \cosh(\beta x + h + p) \right)
\end{align*}

\begin{align*}
\cosh(p) \cosh(q) + \sinh(p) \sinh(q) 
=  \cosh(p+q) \\
\cosh(p) \sinh(q) + \sinh(p) \cosh(q) = \sinh(p+q)
\end{align*}
Use this to get
\begin{align*}
\partial_p H(x,p) 
& = 2\cosh(p) \left[ \sinh(\beta x + h + p ) - x \cosh(\beta x + h+ p) \right]  \\
& + 2 \sinh(p) 
\left[ \cosh(\beta x + h + p ) - x \sinh(\beta x + h + p) \right] \\
& = 2 \sinh( \beta x + h + 2p) - 2x \cosh(\beta x + h + 2p) 
\end{align*}
and 
\begin{align*}
 \partial_p \partial_x H(x,p) 
& = 
-  2\cosh(p)
\left[ (1- \beta) \cosh(\beta x + h + p) + \beta x \sinh(\beta x +h+  p) \right] \\
&\quad  - 
2\sinh(p) 
\left[ (1- \beta) \sinh(\beta x + h + p) + \beta x \cosh(\beta x +h+  p) \right] \\
& = - 2(1- \beta)  \cosh( \beta x + h + 2p) - 2\beta x \sinh(\beta x + h + 2p). 
\end{align*}

\end{calculations}

	\begin{proposition}\label{proposition:convex_at_infty_explict_functions}
		\begin{enumerate}
			\item 
			\label{item:order_preserving_for_CW}
			Assume Assumption \ref{assumption:jump_rates} for the $\pm 1$-space-model, with $H$ as in \eqref{eqn:Hamiltonian_CWmodel}. 
			In addition, suppose $v_+''(1) < 0$ and  $v_-''(-1) < 0$ on a neighbourhood of $-1$. Then $H$  preserves order at infinity.

Let $H$ as in \eqref{eqn:H_for_CW}. Then $H$ preserves order at infinity for all $\beta \ge 0$ and $h\in \R$; for $h=0$ and  $\beta \in [0,2]$, $H$ preserves order on $\downq_{0,0}^\circ \cup \{0\} \cup \upq_{0,0}^\circ$; and, for $\beta = h =0$, $H$ preserves order on the entire space $[-1,1]\times \R$.

			
			\item 
			\label{item:order_preserving_for_R}
			Assume Assumption \ref{assumption:diffusion_drift} for the $\R$-space-model, with $H$ as in \eqref{eqn:Hamiltonian_Rmodel}. 
			Let $y,z\in \R$.  
		If $W'''(x) \le 0$ for all $x \le y$, then $H$ preserves order on $\downq_{y,0}^\circ$. 
		If $W'''(x) \ge 0$ for all $x \ge z$, then $H$ preserves on $\upq_{z,0}^\circ$. 

If such $y$ and $z$ as above exist, then $H$ preserves order at infinity. 
			Such $y$ and $z$ exist, e.g., for  $W$ of the form  $W(x) = \sum_{i=1}^{2k} a_i x^i$ with $a_{2k} > 0$ and thus for $W$ as in \eqref{eqn:H_and_I_0_for_R}.
		\end{enumerate}
		
	\end{proposition}

\begin{figure}
	\begin{center}
		\includegraphics[width=0.9\textwidth]{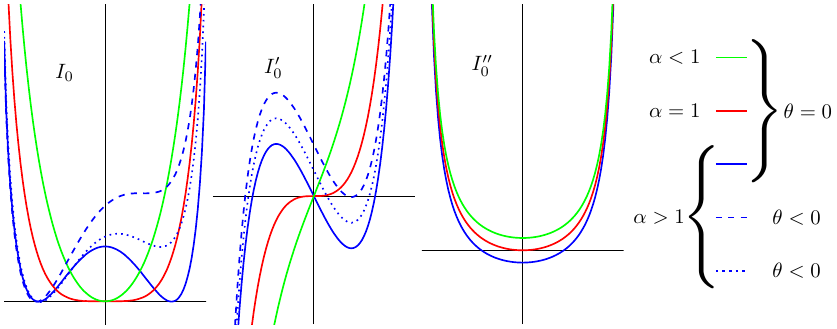}  
	\caption{$I_0$ of \eqref{eqn:I_0_for_CW}, $I_0'$ and $I_0''$ for different values of $\alpha$ and $\theta$.}
		\end{center}
	\label{fig:I_0_derivative_different_alphas}
\end{figure}

	\begin{calculations}
	\begin{align*}
	\partial_x H(x,p)= v_+'(x) [e^{2p} -1] + v_-'(x) [e^{-2p}-1], \\
	\partial_x^2 H(x,p)= v_+''(x) [e^{2p} -1] + v_-''(x) [e^{-2p}-1], 
	\end{align*}
	For $v_+(x) = \frac{1-x}{2} e^{\beta x +h} $ we have 
	\begin{align*}
	v_+'(x) &= \tfrac12 [ \beta (1-x) - 1 ] e^{\beta x + h}\\
	v_+''(x) 
	&= \tfrac{1}{2} [ - \beta + \beta (\beta (1-x) - 1) ] e^{\beta x + h}\\
	&= \tfrac{\beta}{2} [ \beta (1-x) - 2 ] e^{\beta x + h}
	\end{align*}
	And for $v_-(x) = \frac{1+x}{2} e^{-\beta x - h}$ we have 
	\begin{align*}
	v_-'(x) &= \tfrac12 [ 1-\beta (1+x)] e^{-\beta x - h}\\
	v_-''(x) & = \tfrac12 [ -\beta - \beta (1-\beta (1+x) )] e^{-\beta x - h}\\
	& = \tfrac{\beta}{2} [ \beta (1+x)  -2] e^{-\beta x - h}
	\end{align*}
	\end{calculations}

	\begin{proof}

\ref{item:order_preserving_for_CW}
By Proposition \ref{proposition:ordering_is_preserved} and 
because Assumption \ref{assumption:Hamiltonian_for_domain_extension}\ref{item:assumptions_convergent_quadrants} holds (see Lemma \ref{lemma:assumptions_jump_imply_assumptions_general1}), 
it is sufficient to show that $\partial_x^2 H  < 0$ on quadrants $\upq_{x,p}$ and $\downq_{x,p}$ for $x,p$ close to $(1,\infty)$ and $(-1,-\infty)$, respectively. 
Those $x,p$ can be found as $\partial_x^2 H(x,p)= v_+''(x) [e^{2p} -1] + v_-''(x) [e^{-2p}-1]$ and $v''$ is continuous. 

For $H$ as in \eqref{eqn:H_for_CW}, i.e., $H$ as in \eqref{eqn:Hamiltonian_CWmodel} with $v_-$ and $v_+$ as in \eqref{eqn:common_v_-_and_v_+} we have $v_+''(1)= - \beta e^{\beta +h}$ and $v_-''(-1) = - \beta e^{\beta - h}$ which are both $<0$ for all $\beta >0$ and $h\in \R$. 

For $\beta =0$ and $h\in \R$ it follows that $(x,p) \mapsto \partial_x H(x,p)$ is a decreasing function at infinity, whence $H$ preserves at infinity by Proposition \ref{proposition:ordering_is_preserved}. 

For $\beta\in (0,2]$ and $h = 0$, by \eqref{eqn:partial_x_sq_H_in_hyperbolic_functions} we see that $ \partial_x^2 H(x,p)<0 $ on $\downq_{0,0}^\circ \cup \upq_{0,0}^\circ$. 
As  $\partial_p H(0,p)= 2 \sinh(2p)$, 
and $H(x,p)=0$ if and only if either $p=0$ or $p = \arctanh(x) - \beta x$, a Hamilton trajectory in $\downq_{0,0}^\circ$ could only leave this set via the point $(x,0)$ with $H(x,0)=0$, which is a stable point (see Remark \ref{remark:derivatives_of_H_and_I_0}) and therefore cannot be reached. 
Hence $H$ preserves $\downq_{0,0}^\circ$ and similarly $\upq_{0,0}^\circ $. 
Proposition \ref{proposition:ordering_is_preserved} implies that $H$ preserves order on both $\downq_{0,0}^\circ$ and $\upq_{0,0}^\circ $, and therefore on the union of those sets with $\{0\}$. 

For $\beta =h=0$, by \eqref{eqn:partial_x_H_in_hyperbolic_functions} $\partial_x H(x,p) = - 2 \sinh(p) \cosh( p)$ which is decreasing as a function of $(x,p)$, so that $H$ preserves order on the whole space $[-1,1]\times \R$. 

\ref{item:order_preserving_for_R}
Note that $ \partial_x^2 H(x,p) = - p W'''(x)$, which immediately establishes the claim.
\end{proof}

In addition to applying the abstract results of Section \ref{section:preserve_lose_differentiability}, we use Proposition \ref{proposition:graph_of_derivative_pushforward} to show that in particular settings overhangs do induce non-differentiabilities of the rate function. The proofs of the following results are given in Section \ref{section:verification_explicit_example}.

\begin{theorem} \label{theorem:application_explicit}
For the $\R$-space-model and the $\pm 1$-space-model, under the above assumptions, we have the following scenarios for the following $a,b$ and $\alpha, \beta$: 
	\begin{enumerate}
		\item 
		\label{item:all_times_I_t_equals_I_0}
		\textnormal{[$a=b$], [$\alpha = \beta$,  $\theta = h $],  \textbf{equilibrium.}}
		We have $I_0 = I_t$ for all $t \geq 0$.
		\item 
		\label{item:all_a,b_alpha_beta_allways_short_time_gibbs}
		\textnormal{[$a,b\in \R$], [$\alpha,\beta \ge 0$, \ $\theta, h\in \R$],  \textbf{short-time differentiability.}}

		There is a $t_0 > 0$ such that for $t \leq t_0$ we have $I_t \in C^{1,\partial}(\bK)$.
		\item 
		\label{item:high_temp_starting}
		\textnormal{[$a \leq 0$], 
		[$\alpha \le 1$,  $0<\beta \leq 2$, $\theta \in \R$, $h=0$],
		\textbf{high-temperature starting profile.}} 
		For all $t \ge 0$ we have  $I_t \in C^{1,\partial}(\bK)$.
		\item
		\label{item:heating_up_low_temp}
		\textnormal{[$a > b \vee 0$, $b \neq 0$], [$\alpha>\beta \vee 1$, $\beta \neq 1$, $\theta = h =0$],
		\textbf{heating up a low-temperature starting profile.}}
		There is an overhang at $x = 0$ for all $t \geq t_1$ where 
		\begin{enumerate}[label=(\roman*)]
			\item for the $\R$-space-model
			$t_1 = -\frac{1}{b} \log \left(\frac{a-b}{a}\right)$. 
			\item for the $\pm 1$-space-model
			$t_1 = -  \frac{1}{4(1-\beta)} \log \left(\frac{\beta - \alpha}{1-\alpha}\right)$. 
		\end{enumerate}		
		\item
		\label{item:cooling_down_low_temp}
		\textnormal{[$0 <a < b$], [$1<\alpha<\beta$, $\theta = h =0$], 
		\textbf{cooling down a low-temperature starting profile.}}
		There is some $t_2$ such that for all $t \geq t_2$ there exist at least two points $x_{-},x_{+}$ of non-differentiabili\-ty of $I_t$ with $m_{-} < x_{-} < 0 < x_{+} < m_{+}$, where  
		\begin{enumerate}[label=(\roman*)]
		\item for the $\R$-space-model $m_{\pm} = \pm \sqrt{b}$.
		\item for the $\pm 1$-space-model $m_{\pm}$ are the solutions to $\arctanh(x) = \beta x$.\footnote{As $\beta>1$ such $m_-$ and $m_+$ exist.}
		\end{enumerate}		
	\end{enumerate}
\end{theorem}

\begin{theorem}[Loss and recovery] \label{theorem:loss_recovery_differentiability}
 	Let $\K = [-1,1]$ and $H$ and $I_0$ be as in \eqref{eqn:H_for_CW} and \eqref{eqn:I_0_for_CW} with $\beta = h = 0$. 
\begin{enumerate}
\item 
\label{item:unique_zero_derivative_then_later_diff}
For $\alpha \geq 0$ and $\theta \in \R$ such that $I_0'=0$ has a unique solution, there is some time $t^*$ such that  $I_t \in C^{1,\partial}[-1,1]$ for $t \geq t^*$. 
\item 
\label{item:loss_recovery_differentiability}
For all $\alpha>1$ there exists a $\kappa>0$ such that for all $\theta \in \R$ with $|\theta|>\kappa$, there are times $t_0 < t_1 \leq t_2$ such that 
	\begin{enumerate}[label=(\roman*)]
		\item 
		\label{item:small_times_differentiable}
		$I_{t} \in C^{1,\partial}[-1,1]$ for $t<  t_0$, 
		\item 
		\label{item:intermediate_times_overhang_and_non_diff}
		$\cG_t$ contains an overhang for $t \in (t_0,t_2)$ and $I_t$ is non-differentiable for $t \in (t_0,t_1)$,
		\item 
		\label{item:large_times_differentiable}
		$I_t \in C^{1,\partial}[-1,1]$ for $t > t_2$. 
	\end{enumerate}
\end{enumerate}
\end{theorem}

\subsection{Remarks on the results and comparison with the literature}

\begin{remark} 
Our method to verify non-differentiability for Theorem \ref{theorem:application_explicit} \ref{item:heating_up_low_temp} is based on Proposition \ref{proposition:graph_of_derivative_pushforward} up to the time at which the push-forward of $\cG$ falls apart in three separate curves, cf. Proposition \ref{proposition:hamiltonian_paths_and_graph_structures}. 

Similarly, in our proof of Theorem \ref{theorem:loss_recovery_differentiability}\ref{item:loss_recovery_differentiability}, we actually have $t_1< t_2$ (see also Remark \ref{remark:t_1_strict_less_t_2}), i.e., for $t\in [t_1,t_2)$ there is an overhang but again Proposition \ref{proposition:graph_of_derivative_pushforward} cannot be used to conclude that $I_t$ is non-differentiable. 
\end{remark}

\begin{remark}
\label{remark:comparing_with_HRZ15}
In \cite{HRZ15} the $\R$-space-model with $H(x,p) = \frac12 p^2$ and $I_0(x) = \frac12 x^2 +V(x) $, where $V \in C^1(\R,[0,\infty))$  is considered.
We show that the existence of an overhang as in Theorem \ref{theorem:loss_of_differentiability}\ref{item:loss_linearization} agrees with the non-differentiability claimed in \cite{HRZ15}. 
Note that the Hamiltonian flow admits a $C^1$ linearization by the identity map as the flow itself is linear. 
In \cite[Corollary 1.12]{HRZ15} it is shown that $I_t(b)$ has a unique global minimiser for all $b\in \R$ if and only if $\Phi_2 V > - \frac{1+t}{2t}$. 
Hence $I_t$ is not differentiable if $\Phi_2 V \not > - \frac{1+t}{2t}$, which by \cite[Lemma 5.9]{HRZ15} is the case when there exists an $x_0$ for which $V''(x_0) + \frac12 < - \frac{1}{t}$, which is the same as $t>t_0$ for $t_0$ as in \eqref{eqn:time_for_vertical_in_linearized_sytem}. 

Theorem \ref{theorem:convexity_preserving_implies_diffb_preserv} 
also agrees with \cite[Corollary 1.12]{HRZ15} in case $I_0(x) = \frac12 x^2 + V(x) $ is strictly convex (at infinity). 
However, the setting in \cite{HRZ15} allows $I_0$ not to be in $C^{1,\partial}$.
This is the case for, e.g., $V(x) = 1+ \cos(x^2)$ as in \cite[Example 1.16]{HRZ15}. 
Here one has immediate loss of uniqueness of minimisers and therefore an immediate loss of differentiability, which can be proved by \cite[Corollary 6.4.4 and Theorem 6.4.8]{CaSi04} in the same way as in the proof of Theorem \ref{theorem:equivalences_differentiability}. 
\end{remark}

	\begin{remark} 
		Consider a stationary point $x_0$. The condition that $\partial_p \partial_x H(x_0,0) \neq 0$ in both Theorem \ref{theorem:loss_of_differentiability} \ref{item:loss_linearization} and \ref{item:loss_rotating} is reflected by the exclusion of having 	$b =0$ and $\beta = 1$, i.e., $\partial_p \partial_x H(0,0) = 0$, in Theorem \ref{theorem:application_explicit} \ref{item:heating_up_low_temp}. 
		
		In the literature on dynamical systems, the failure of $\partial_p \partial_x H(x_0,0) \neq 0$ implies that $(x_0,0)$ is a non-hyperbolic fixed point of the Hamiltonian flow. This can be considered to be critical behaviour: for a non-hyperbolic fixed point, the first-order approximation does not describe the global behaviour of the flow around this point. 
		
 		This is similar to the statement that $\alpha = 1$ is critical for the Curie-Weiss model: the first-order approximation  $I_0'(0)$, with $I_0$ as in \eqref{eqn:I_0_for_CW}, of the rate function $I_0$ at the point $0$ vanishes for $\alpha = 0$, indicating a transition from a convex to a non-convex rate-function. 
		
	\end{remark}

\begin{remark}
Both the idea of linearization and rotation already appeared for the Lagrangian flow in \cite{EK10}. 
The idea of considering the Hamiltonian flow instead of the Lagrangian flow already appeared in \cite[Chapter 5]{KrPhD}. 
As our methods do not depend on a specific model, we recover part of the results of \cite{EK10}, however some of our results are slightly weaker: 
		
		Using explicit calculations, \cite[Theorem 1.3]{EK10} obtains the result in Theorem \ref{theorem:application_explicit}\ref{item:high_temp_starting} for all $\beta$ instead of $\beta \leq 2$. 		
		
		Our result in \ref{item:heating_up_low_temp} for $1 \vee \beta < \alpha$ is sub-optimal. \cite{KN07,EK10} show that points of non-differentiability occur before the linearized system assures that we have a point of non-differentiability at $0$. In this setting, there is a different mode, apart from the rotation around $0$ that creates the overhang. This mode can easily be identified by using pictorial analysis based on the Hamiltonian flow. 
\end{remark}

\section{A study of the Hamiltonian dynamics, optimizers, and the time-evolved rate function}
\label{section:calculus_of_variations}

The extension of calculus of variations to a setting where the Hamiltonian trajectories may have finite maximal times of existence needs a treatment of the behaviour of the Hamiltonian flow close to the boundary. This analysis will be carried out under general conditions that are introduced in Sections \ref{section:cond_ham_boundary} and \ref{section:cond_ham_boundary_case_specific}. We will show that these conditions imply that Hamiltonian trajectories are pushed away from the boundary, and can only arrive at the boundary with infinite momentum at their maximal time of existence.  

In Section \ref{section:relating_optimizers_to_hamflow}, we show that optimizers of $I_t(b)$ together with their dual trajectories,  are solutions of the Hamilton equations. In Section \ref{section:proofs_regularity_rate_function} we establish the regularity properties of the rate function that we introduced in Sections \ref{section:regularity} and \ref{section:regularity_via_pushforward}.

\subsection{Conditions on Hamiltonian and properties of the Hamiltonian flow} \label{section:cond_ham_boundary}

Below, we introduce the main assumptions of our paper. The assumptions (a)-(e) fall apart in two natural groups. (a) and (b) are there to study the Hamiltonian flow in open subsets of $\bK\times \bR$. In particular, these conditions suffice to apply the methods described in \cite{CaSi04}, compare with Conditions (L1)-(L4) on $\cL$ in \cite{CaSi04}, as long as Hamiltonian trajectories remain inside such open sets.

The assumptions (c)-(e) are made to study behaviour of Hamiltonian trajectories at the boundary, or to show that Hamiltonian trajectories stay away from the boundary.

\begin{assumption} \label{assumption:Hamiltonian_for_domain_extension}
	The Hamiltonian $H : \K \times \bR \rightarrow \bR$ satisfies $H(x,0) = 0$ for all $x$ and
	\begin{enumerate}
		\item \label{item:assumption_extended_regularity}
 		$H$ is $C^2$ and $\partial_p^2 H(x,p) > 0$ for all $(x,p) \in \K \times \bR$.
		
		If $\K = [-1,1]$, then we additionally assume that there exist an $\epsilon>0$ and a twice continuously differentiable function $\tilde H: (-1-\epsilon,1+\epsilon) \times \bR \rightarrow \bR$ such that $H$ equals $ \tilde H$ on $[-1,1]\times \bR$. 
		\item \label{item:assumptions_theta}
		For every compact set $K \subseteq  \K^\circ$, there exists a function $\theta_K : [0,\infty) \rightarrow [0,\infty)$, with the properties 
		\begin{enumerate}[label=(\roman*)]
			\item $\lim_{r \rightarrow \infty} r^{-1}\theta_K(r) = \infty$,
			\item for every  $M \geq 0$, there is a  $k_M \geq 0$ such that 	
			$\theta_K(r+m) \leq k_M\left(1+\theta_K(r)\right)$ for all $m \in [0,M]$ and $r \geq 0$, 
			\item there exists $c_K$ such that $\cL(x,v) \geq \theta_K(|v|) - c_K$ for all $x \in K$ and $v \in \bR$,
			\item there exists $C_K$ such that $|\partial_x \cL(x,v)| + |\partial_v \cL(x,v)| \leq C_K \theta_K(|v|)$ for all $x \in K$ and $v \in \bR$.
		\end{enumerate}
		\item \label{item:assumptions_quotient_H} For each compact $K \subseteq \K^\circ$, we have $\lim_{|p| \rightarrow \infty} \inf_{x\in K} \frac{H(x,p)}{|p|} = \infty$.
		
		If $\K = [-1,1]$, then we additionally assume
		\begin{align*}
		\lim_{p \rightarrow \infty}   \frac{H(-1,p)}{p} = \infty, 
		\qquad \lim_{p \rightarrow - \infty}  \frac{H(1,p)}{-p} = \infty.
		\end{align*}		
		
			
			\item \label{item:assumptions_drift_boundary}  We have 
			\begin{equation*}
			\lim_{x \rightarrow \partial_-} \argmin_{p\in\R} H(x,p) = -\infty, \qquad \lim_{x \rightarrow \partial_+} \argmin_{p\in\R} H(x,p) = \infty
			\end{equation*}
		

		\item \label{item:assumptions_convergent_quadrants}
		There exists a sequence $((y_n^+, q_n^+))_{n\in\N}$ in $\K^\circ \times (0,\infty)$  with $\limn (y_n^+, q_n^+) = (\partial_+,\infty)$ and
		\begin{align}
		\partial_p H(y_n^+,q) \geq 0 & \mbox{ for } q \ge q_n^+, \\
		-\partial_x H(y,q_n^+) \geq 0 & \mbox{ for } y \ge y_n^+,
		\end{align}
		and there exists a sequence $((y_n^-, q_n^-))_{n\in\N}$ in $\K^\circ \times (-\infty,0)$  with \\ $\limn (y_n^-, q_n^-) = (\partial_-,-\infty)$ and
		\begin{align}
		\partial_p H(y_n^-,q) \leq 0 & \mbox{ for } q \le q_n^-, \\
		-\partial_x H(y,q_n^-) \leq 0 & \mbox{ for } y \le y_n^-,
		\end{align}
	\end{enumerate}
\end{assumption}

\begin{remark} \label{remark:assumptions_implies_preserved_quadrants}
\begin{itemize}
\item 	By Assumption \ref{assumption:Hamiltonian_for_domain_extension} \ref{item:assumptions_convergent_quadrants}  $\upq_{y_n^+,q_n^+}$ and $\downq_{y_n^-,q_n^-}$ are preserved under $H$. 
\item 
Assumption \ref{assumption:Hamiltonian_for_domain_extension}\ref{item:assumptions_quotient_H} implies that for every compact set $K \subseteq \K^\circ$ and $c\in \R$ the set 
$(K\times \R) \cap H^{-1}(-\infty,c]$ is compact. 
\end{itemize}
\end{remark}

\begin{lemma}
\label{lemma:properties_obtained_for_CW_from_the_assumptions}
Let $\K = [-1,1]$. Then 
Assumption \ref{assumption:Hamiltonian_for_domain_extension}\ref{item:assumption_extended_regularity} and \ref{item:assumptions_drift_boundary} imply
		\begin{align}
\label{eqn:limits_partial_p_H_at_pm_1}
			v_-:= \lim_{p \rightarrow - \infty} \partial_p H(-1,p) &  \geq 0, &
			\qquad 
			v_+:=\lim_{p \rightarrow \infty} \partial_p H(1,p) & \leq 0, \\
\label{eqn:limit_of_quotient_of_H_at_boundary_points}
				\lim_{p \rightarrow - \infty} \frac{H(-1,p)}{p} =  v_-& \geq 0, &
				\qquad 
				\lim_{p \rightarrow \infty} \frac{H(1,p)}{p}  = v_+ & \leq 0, 
		\end{align}
and that the Hamiltonian vector field on the boundary points inwards:
		\begin{equation}
		\label{eqn:push_away_from_boundary}
		\partial_p H(-1,p) > 0, \qquad \partial_p H(1,p) < 0. 
		\end{equation}
Moreover, these assumptions together with Assumption \ref{assumption:Hamiltonian_for_domain_extension}\ref{item:assumptions_quotient_H} imply that for all $c\in \R$, for every $(a,q) \in (-1,1) \times (0,\infty)$ and $(b,r) \in (-1,1) \times (-\infty,0)$ the set (see Figure \ref{fig:push_boundary_and_complement_quadrants}(a)) 
\begin{equation}
\label{eqn:compact_set_of_complement_quadrants}
\Big( \upq_{a,q}^\circ \cup \downq_{b,r}^\circ \Big)^c \cap H^{-1}(-\infty,c]
\end{equation}
is compact. 
\end{lemma}
\begin{proof}
Because \ref{item:assumption_extended_regularity} we have that $p \mapsto H(x,p)$ is strictly convex and thus that 
$q=  \argmin_{p\in\R} H(x,p) $ if and only if $\partial_p H(x,p) =0$. 
Moreover, if $q >\argmin_{p\in\R} H(x,p) $ then $\partial_p H(x,p)>0$. 
Therefore \ref{item:assumptions_drift_boundary} implies that $\partial_p H(-1,q) = \lim_{x\rightarrow -1} \partial_p H(x,p) \ge 0$. This in turn implies \eqref{eqn:limits_partial_p_H_at_pm_1}. 
By the strict convexity of $p \mapsto  H(x,p) $ we have 
\begin{align}
\label{eqn:property_of_strict_convexity}
q< p 
\quad \Longrightarrow 
\quad \partial_p H(x,q) < \frac{H(x,q) - H(x,p)}{q-p} < \partial_p H(x,p). 
\end{align}
Hence \eqref{eqn:limits_partial_p_H_at_pm_1} implies \eqref{eqn:limit_of_quotient_of_H_at_boundary_points} and \eqref{eqn:push_away_from_boundary}. 

Note that $\Big( \upq_{a,q}^\circ \cup \downq_{b,r}^\circ )^c$ is a subset of the union of the compact set  $[-1,1]\times [r,q]$ with $[-1,u] \times [0,\infty)$, $[v,1] \times (-\infty,0]$ and $[u,v] \times \R$ for  any $u <a$ and $v>b$. 
Hence to show that \eqref{eqn:compact_set_of_complement_quadrants} is compact, it is by Remark \ref{remark:assumptions_implies_preserved_quadrants} sufficient to show that there exist $u$ and $v$ in $(-1,1)$ such that $\big ( [-1,u] \times [0,\infty) \big) \cap H^{-1}(-\infty,c]$ and $\big( [v,1] \times (-\infty,0] \big) \cap H^{-1}(-\infty,c]$ are compact. 

By \ref{item:assumptions_drift_boundary} there exists an $u$ such that $\argmin_{p\in \R} H(x,p)<0$ for all $x\le u$. 
Therefore there exists an $\epsilon>0$ such that $\partial_p H(x,p) > \epsilon$ for all $x \in [-1, u]$ and all $p \ge 0$ (see also \eqref{eqn:property_of_strict_convexity}). 
As $H(x,0)=0$ for all $x$, this implies that there exists an $M>0$ such that for $p \ge M$ and $x\in [-1,u]$ we have  $H(x,p) > c$. This proves that $\big ( [-1,u] \times [0,\infty) \big) \cap H^{-1}(-\infty,c]$ is compact. The existence of a $v$ can be proved in the same way. 
\end{proof}

We proceed with exploring the behaviour of solutions to the Hamilton equations `close' to the boundary. This analysis will crucially depend on Assumptions \ref{assumption:Hamiltonian_for_domain_extension} (c), (d) and (e). We start with a result that captures the idea that the Hamiltonian flow points away from the `boundary' points $\partial_-$ and $\partial_+$, unless one   starts with  very low and very high momentum, respectively (see also Figure \ref{fig:push_boundary_and_complement_quadrants}(b)). 

\begin{figure}[h]
	\begin{center}
\begin{tabular}{cc}
		\includegraphics[width=0.2\textwidth]{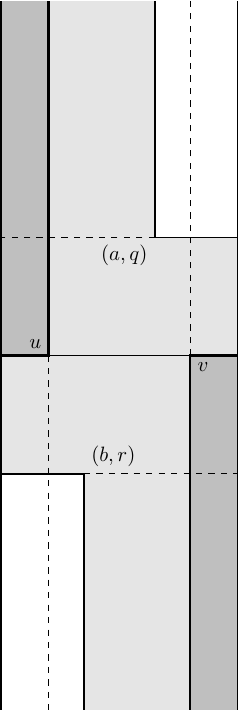}
		& \quad \quad \quad \quad 
		\includegraphics[width=0.2\textwidth]{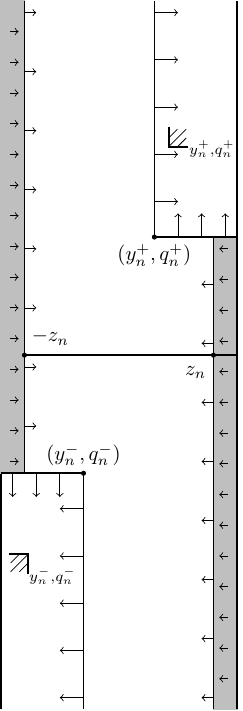}   \\
		(a) & (b)
\end{tabular}
\caption{(a) complement of quadrants, (b) push from boundary.}
	\end{center}
	\label{fig:push_boundary_and_complement_quadrants}
\end{figure}

	\begin{lemma}
		\label{lemma:delta_n}
		Let $H$ satisfy Assumption \ref{assumption:Hamiltonian_for_domain_extension}. 
		There exists a sequence  $(z_n )_{n\in\N}$ in $(0,\partial_+)$
		such that 
		\begin{align}
		\label{eqn:bound_delta_n_and_q_n}
		-z_n < y_n^-, &  \qquad   z_n > y_n^+, \\
		\label{eqn:forbidden_regions_left}
		\partial_p H(x,p) >0 &  \qquad \mbox{ if } & \partial_- \le  & \ x  \le -z_n, & p & \ge q_n^-, \\
		\label{eqn:forbidden_regions_right}
		\partial_p H(x,p) <0 &  \qquad \mbox{ if } & z_n \le & \ x  \le \partial_+,  & p & \le q_n^+. 
		\end{align}
		\eqref{eqn:forbidden_regions_left} and \eqref{eqn:forbidden_regions_right} 
		(together with Assumption \ref{assumption:Hamiltonian_for_domain_extension}\ref{item:assumptions_convergent_quadrants})
		imply that the complements of the sets 
		$\{x: x \le z_n\} \times [q_n^-,\infty)$ and $\{x: x \ge z_n\}\times (-\infty,q_n^+]$ 
		are preserved. 
	
See Figure \ref{fig:push_boundary_and_complement_quadrants}(b) for a picture. 
	\end{lemma}

	\begin{proof}
		Fix $n$. By  
			Assumption \ref{assumption:Hamiltonian_for_domain_extension}\ref{item:assumptions_drift_boundary}
		we can choose  $-z_n$ and $z_n$ in $\K^\circ$ such that
		\begin{equation*}
		\inf_{x \leq -z_n} \argmin_{p\in \R} H(x,p) \leq q_n^-, \qquad \inf_{x \geq z_n} \argmin_{p\in \R} H(x,p) \geq q_n^+.
		\end{equation*}
		By the assumed strict convexity of $p \mapsto H(x,p)$ we have $\partial_p H(x,p_0)=0 $ for $p_0 =\argmin_{p\in \R} H(x ,p)$ and $\partial_p H(x,p)<0$ for $p< p_0$ and $\partial_p H(x,p)>0$ for $p > p_0$. Therefore we conclude \eqref{eqn:forbidden_regions_left} and \eqref{eqn:forbidden_regions_right}. 
		We can choose $z_n$  large enough  such that  \eqref{eqn:bound_delta_n_and_q_n} is satisfied as well. 
	\end{proof}

\begin{remark}
	\label{remark:limit_Hamiltonian_path_boundaries}
	Note that Lemma \ref{lemma:delta_n} implies that for $x \in \K$ and $p\in \R$, one has the following implications  for $t>0$
	\begin{align}
	\label{eqn:diverging_momentum_minus}
	\lim_{s\uparrow t} X_t^{x,p} = -1 & \qquad \Longrightarrow \qquad \lim_{s\uparrow t} P_t^{x,p} = -\infty, \\
	\label{eqn:diverging_momentum_plus}
	\lim_{s\uparrow t} X_t^{x,p} = 1 & \qquad \Longrightarrow \qquad \lim_{s\uparrow t} P_t^{x,p} = \infty.
	\end{align}
	Note that in both cases $t= t_{x,p}$. 
\end{remark}

	\begin{lemma} \label{lemma:after_lifetime_in_boundary}
		Suppose that $(x,p) \in \K \times \bR$ 
		are such that $t_{x,p} < \infty$. Then
		\begin{equation*}
		\lim_{t \uparrow t_{x,p}} (X_t^{x,p},P_t^{x,p})
		\end{equation*}
		takes its values in $\{(\partial_-,-\infty),(\partial_+,\infty)\}$.
		\end{lemma}

	\begin{proof}
		By \cite[Theorem 2.4.3]{Pe01} (in case $\K = [-1,1]$, then applied to the Hamilton equations for the extended Hamiltonian as in Assumption \ref{assumption:Hamiltonian_for_domain_extension}(a)), for any compact set $K \subseteq \K \times \bR$ there exists a time $\beta < t_{x,p}$ such that $(X^{x,p}_\beta,P^{x,p}_\beta) \notin K$.  We will use this fact in two separate ways, depending on the setting, to prove the result.
		
		First suppose $\K = [-1,1]$. Let $c= H(x,p)$  and $n\in\N$,  $(a,q) = (y_n^+,q_n^+)$ and $(b,r) = (y_n^-,q_n^-)$ (see Assumption \ref{assumption:Hamiltonian_for_domain_extension}(e)). By Lemma \ref{lemma:properties_obtained_for_CW_from_the_assumptions} the set \eqref{eqn:compact_set_of_complement_quadrants} is compact.	Hence there exists a $\beta_0< t_{x,p}$ such that $(X^{x,p}_{\beta_0},P^{x,p}_{\beta_0})$ is either in $\upq_{y_n^+,q_n^+}$ or $\downq_{y_n^-,q_n^-}$. Assume it is in $\upq_{y_n^+,q_n^+}$. Then $(X^{x,p}_{\beta_0},P^{x,p}_{\beta_0})$ is in $\upq_{y_n^+,q_n^+}$  for all $\beta \in (\beta_0,t_{x,p})$ as $\upq_{y_n^+,q_n^+}$ is preserved (see also Remark \ref{remark:assumptions_implies_preserved_quadrants}). In a similar way as above, for each $n\in\N$ there exists a $\beta_n$ such that $(X^{x,p}_{\beta},P^{x,p}_{\beta}) \in \upq_{y_n^+,q_n^+}$ for all $\beta > \beta_n$. This implies $\lim_{t \uparrow t_{x,p}}(X_t^{x,p}, P_t^{x,p}) = (1,\infty)$.
		
		Now let $\K = \R$. Let $c= H(x,p)$ and let $n_0\in\N$ be such that $y_n^- < x < y_n^+$ for $n\ge n_0$. Instead of using Lemma \ref{lemma:properties_obtained_for_CW_from_the_assumptions} in the above lines of proof for $\K =[-1,1]$, we can use Remark \ref{remark:assumptions_implies_preserved_quadrants} and Lemma \ref{lemma:delta_n} to obtain the existence of a $\beta_0$ and $\beta_n$ as above and follow the same lines. 
	\end{proof}

	\begin{proposition} \label{prop:trajectories_bounded_away_from_boundary}
		Let $H$ satisfy Assumption \ref{assumption:Hamiltonian_for_domain_extension}.
		For all $u\in (0,\partial_+)$ there exists a $v\in (u,\partial_+)$ such that for all $t\ge 0$ and all starting points and momenta $(x,p) \in (-v,v) \times \bR$, with $ X_t^{x,p} = b \in (-u,u)$ we have $X^{x,p}_s \in (-v,v)$ for all $s \in [0,t]$.
	\end{proposition}
	\begin{proof}
		Let $n$ be such that $b\in (y_n^-,y_n^+)$. 
		Let $z_n$ be as in Lemma \ref{lemma:delta_n} and $v> z_n$. 
		Suppose $x \in (-v ,v)$ and $p \in \R$ are such that $X_t^{x,p}=b$. 
		As the end point $(X_t^{x,p}, P_t^{x,p})$ lies outside the quadrants $\upq_{y_n^+,q_n^+}$ and $\downq_{y_n^-,q_n^-}$ the whole trajectory does not enter either of the quadrants as those quadrants are preserved. 
		
		As the complements of the sets $\{x: x \le z_n\} \times [q_n^-,\infty)$ and $\{x: x \ge z_n\}\times (-\infty,q_n^+]$ are preserved, 
		 $X_s^{x,p}$ for $s\in [0,t]$ is prevented from entering $\{x : x\le v\} \cup \{x : x\ge v\}$. 
	\end{proof}

\subsection{Additional model dependent conditions and properties} \label{section:cond_ham_boundary_case_specific}

In addition to the properties of the Hamiltonian flow that were established above, there are some peculiarities due to the boundaries of both settings that need to be treated separately. In the setting of the $\pm 1$-space-model we need these auxiliary technical results in Section \ref{section:relating_optimizers_to_hamflow} below to show that optimizers that start at the boundary can be related to solutions of the Hamilton equations. For the $\R$-space-model, we introduce a condition that ensures that the rate function has compact sublevel sets, something that for the $\pm 1$-space-model follows from Assumption \ref{assumption:Hamiltonian_for_domain_extension} above.

\subsubsection{The \texorpdfstring{$\pm 1$-space-model}{}}

In this section, we consider $\bK = [-1,1]$.
We start by showing that $\cL$ is twice continuously differentiable on an appropriate domain. 
We then proceed by extending the regularity of the Hamiltonian flow up to the boundary of $[-1,1]\times \bR$, after which we give an additional assumption that is needed to verify Proposition \ref{proposition:optimizers_equiv_hamflow_CW} \ref{item:optimisers_in_interior_interm_times} below.

	\begin{lemma}
		\label{lemma:lagrangian_is_C2_on_interior}
 Let $\K=[-1,1]$ and let $H$ satisfy Assumption \ref{assumption:Hamiltonian_for_domain_extension}.  Then $\cL(-1,v) = \infty$ for all $v\in (-\infty,v_-)$ and $\cL(1,v) = \infty$ for all $v\in (v_+,\infty)$,  and
		$\cL$ is $C^2$ on $(-1,1) \times \R \cup (\{-1\} \times (v_-,\infty)) \cup (\{1\} \times (-\infty,v_+))$.
	\end{lemma}
	
	\begin{proof}
That $\cL(-1,v) = \infty$ for all $v\in (-\infty,v_-)$ and $\cL(1,v) = \infty$ for all $v\in (v_+,\infty)$ follows from \eqref{eqn:limit_of_quotient_of_H_at_boundary_points}. 

		First we prove that $v \mapsto \cL(-1,v)$ is $C^2$ on $(v_-,\infty)$, similarly one proves that $v\mapsto \cL(1,v)$ is $C^2$ on $(-\infty,v_+)$. 
		Write $L(v) = \cL(-1,v)$. 
		In an analogous way as in \cite[Theorem A.2.3]{CaSi04} one proves that for all $v\in (v_-,\infty)$ there exists a unique $p$ such that $L(v) = pv - H(-1,p)$, that $\lim_{v\rightarrow \infty} \frac{L(v)}{v} = \infty$, that $H(-1,\cdot)$ is the Legendre transform of $L$ (see the Fenchel-Moreau Theorem \cite[Chapter 5]{RaSe15}) and that $p \in D^- L(v)$ if and only if $v\in D^- H(-1,p)$ (see \cite[Theorem 5.22]{RaSe15}). With this one proves with the same argument as in the proof of \cite[Theorem A.2.4]{CaSi04} that $L$ is $C^1$ on $(v_-,\infty)$. Similarly, the argument of the proof of \cite[Theorem A.2.5]{CaSi04} carries over so that we have that $L$ is $C^2$ on $(v_-,\infty)$. 
		
		That $\cL$ is $C^2$ on $(-1,1) \times \R$ follows from \cite[Theorem A.2.7]{CaSi04}. 
		
		Following the lines of the proof of \cite[Theorem A.2.7]{CaSi04}, it is sufficient to show that the map $q: (-1,1) \times \R \cup (\{-1\} \times (v_-,\infty)) \cup (\{1\} \times (-\infty,v_+)) \rightarrow \R$, which assigns to an element $(x,v)$ the unique $p$ such that $\cL(x,v) = pv - H(x,p)$, is $C^1$. 
		As $q$ satisfies $v - \partial_p H(x,q(x,v)) =0$, and $\partial_p^2 \overline H$ (see Assumption \ref{assumption:Hamiltonian_for_domain_extension}\ref{item:assumption_extended_regularity}) is $>0$ in an neighbourhood of $(x,p)$ for all $(x,p) \in [-1,1] \times \R$, by the implicit function theorem it follows that $q$ is $C^1$. 
	\end{proof}

\begin{lemma}
	\label{lemma:limit_points_in_extended_space}
	 Let $\K = [-1,1]$ and  let $H$ satisfy Assumption \ref{assumption:Hamiltonian_for_domain_extension}.

	Suppose that $(\gamma,\eta): (0,t] \rightarrow [-1,1]\times \R$ is $C^1$ 
	 and satisfies the Hamilton equations \eqref{eqn:hamilton_equations}. Then $(\gamma(0),\eta(0)):=\lim_{s\downarrow 0} (\gamma(s),\eta(s))$ exists in $[-1,1] \times \bR$ and $(\gamma,\eta)$ is $C^1$ on $[0,t]$ and satisfies the Hamilton equations.
\end{lemma}
\begin{proof}

	Define the time-inverted trajectories $\gamma_*, \eta_*$ on $[0,t)$ by
	\begin{equation*}
	\gamma_*(s) := \gamma(t-s) , \qquad
	\eta_* (s) := \eta( t-s).
	\end{equation*}
	Then 
	\begin{align}
	\begin{bmatrix}
	\dot{ \gamma_*}(s) \\
	\dot{ \eta_*}(s)
	\end{bmatrix}
	=
	\begin{bmatrix}
	-\partial_p H(\gamma_*(s),\eta_*(s)) \\
	\partial_x H(\gamma_*(s),\eta_*(s))
	\end{bmatrix}.
	\label{eqn:minus_hamilton_equation}
	\end{align}
	Because $(\gamma_*,\eta_*)$ solves the time-inverted Hamilton equations,  \eqref{eqn:minus_hamilton_equation}, Assumption \ref{assumption:Hamiltonian_for_domain_extension}(e) implies that if $n\in\N$ is  such that $(\gamma_*(0), \eta_*(0))$ is not in the set
	\begin{align}
	\label{eqn:union_of_two_quadrants}
	\upq_{y_n^+,q_n^+}
	\cup \,
	\downq_{y_n^-,q_n^-}
	\end{align}
	then $(\gamma_*(s),\eta_*(s))$ is not in this set for all $s\in [0,t)$.
		As $H(\gamma_*(s),\eta_*(s))= H(\gamma_*(0),\eta_*(0)=:c$ for all $s$, and the complement of \eqref{eqn:union_of_two_quadrants} intersected with $H^{-1}(\{c\})$
		is compact by Lemma \ref{lemma:properties_obtained_for_CW_from_the_assumptions}, 
	by \cite[Theorem 2.4.3]{Pe01} we find that the maximal interval of existence $(\gamma_*,\eta_*)$ satisfying \eqref{eqn:minus_hamilton_equation} is larger than $[0,t)$, which implies that both limits $\lim_{s\downarrow 0} (\gamma(s),\eta(s))= \lim_{s\uparrow t} (\gamma_*(s),\eta_*(s))$ and $ \lim_{s\downarrow 0} (\dot \gamma(s), \dot \eta(s))=- \lim_{s \uparrow t}  (\dot \gamma_*(s), \dot \eta_*(s))$ exist. Additionally, we  conclude that the trajectory $(\gamma,\eta)$ solves the Hamilton equations on the interval $[0,t]$.
\end{proof}


\begin{assumption} \label{assumption:trajectories_in_interior}
Let	 $\K = [-1,1]$ and assume that $H$ satisfies Assumption \ref{assumption:Hamiltonian_for_domain_extension} and that there is a function $S : [-1,1] \rightarrow [0,\infty]$ that is twice continuously differentiable on $(-1,1)$ such that 
	\begin{enumerate}
		\item $S$ is a Lyapunov function for $\dot{x} = \partial_p H(x,0)$, i.e. if $x(t)$ solves $\dot{x} = \partial_p H(x,0)$ then $t \mapsto S(x(t))$ is decreasing.
		\item $H(x,S'(x) - p) = H(x,p)$ for all $(x,p) \in (-1,1) \times \bR$.
		\item The map $x \mapsto \cL(x,0)$ is decreasing on a  neighbourhood $U_{-1}$ of $-1$ in $[-1,1]$ and increasing on an  neighbourhood $U_1$ of $1$ in $[-1,1]$. 
		\item We have
		\begin{equation*}
		\lim_{x \downarrow -1} \frac{\cL(x,0) - \cL(-1,0)}{S(x) - S(-1)} = \infty, \quad
		\lim_{x \uparrow 1} \frac{\cL(x,0) - \cL(1,0)}{S(x) - S(1)} = \infty.
		\end{equation*}
	\end{enumerate}
\end{assumption}

\subsubsection{The \texorpdfstring{$\R$-space-model}{}}

In this section, we consider $\bK = \bR$. We assume the existence of a Lyapunov function $\Upsilon$ that is needed to treat the non-compactness of $\bR$ in  the proof of compactness of the sublevel sets of the rate function (see Lemma \ref{lemma:compact_sublevel_sets}).

\begin{assumption} \label{assumption:compact_level_sets_R}
	There is a continuously differentiable function $\Upsilon$ with compact sublevel sets, i.e. for all $c \in \bR$ the set $\{x \in \bR : \Upsilon(x) \leq c\}$ is compact, with the additional property that $\sup_{x \in \R} H(x,\Upsilon'(x)) < \infty$. 
\end{assumption}

\subsection{The relation between optimizers and Hamiltonian trajectories and a study of the time-evolved rate function} \label{section:relating_optimizers_to_hamflow}

From the discussion above, we understand the behaviour of the Hamiltonian flow. In this section, we show that optimal trajectories exist (Lemma \ref{lemma:compact_sublevel_sets}), and, combined with their dual trajectories, satisfy the Hamilton equations (Propositions \ref{proposition:optimizers_equiv_hamflow_CW} and \ref{proposition:optimizers_equiv_hamflow_R}).
Then we show that the range of optimal trajectories can be controlled if $I_0 \in C^{1,\partial}(\bK)$ (Proposition \ref{proposition:optimizers_bounded_away_from_boundary}).

\begin{lemma}
\label{lemma:compact_sublevel_sets}
	Let $H$ satisfy Assumption \ref{assumption:Hamiltonian_for_domain_extension}. If $\bK = \bR$, let $H$ in addition satisfy Assumption \ref{assumption:compact_level_sets_R}
	
	Let $I_0$ have compact sublevel sets, then the rate function $\cI$ in \eqref{eqn:rate_function_paths} has compact sublevel sets.  
	
	 In addition, for every $t > 0$, $a \in \bK$ and $b \in \bK^\circ$ there is an optimal trajectory $\gamma$ for $S_t(a,b)$ and $I_t(b)$. 
\end{lemma}

	\begin{proof}
	Compactness of sublevel sets of $\cI$ is generally proven as a part of a large deviation principle. For a proof of this result in the $\pm 1$-space-model setting, see  \cite[Theorem 2]{Kr16b}. A similar proof for the $\R$-space-model can be carried out by using $\Upsilon$, see the proof of  \cite[Theorem A.14]{CoKr17}.
	
	 Pick $a \in \bK$, $b \in \bK^\circ$ and $t > 0$. Pick a $C^1$ curve $\gamma$ connecting $a$ to $b$ in time $t$ such that in addition $(\gamma(s),\dot{\gamma}(s))$ takes its values in the region where $\cL$ is $C^2$, cf. Lemma \ref{lemma:lagrangian_is_C2_on_interior} for the case where $\bK = [-1,1]$. This implies that $\gamma$ has finite cost. Thus, by the contraction principle and the compactness of the sublevel sets of $\cI$, where we take as a starting rate function $I_0(a) = 0$ and $I_0(x) = \infty$ if $x \neq a$, there is an optimizer for $S_t(a,b)$.

	
	Again by the contraction principle, but now with a general starting rate function $I_0$, we find that there is an optimizer for $I_t(b)$.
\end{proof}

\begin{proposition} \label{proposition:optimizers_equiv_hamflow_R}
	Let $H$ satisfy Assumption \ref{assumption:Hamiltonian_for_domain_extension}  in the setting that $\bK = \bR$, let $I_0 \in C^{1,\partial}(\bR)$, and $t > 0$.  

		Let $a,b \in \bR$ and suppose that $\gamma$ is an optimal trajectory for $S_t(a,b)$. Then $\gamma$ is $C^2$ on $[0,t]$. The dual trajectory $\eta$ defined on $[0,t]$ is $C^1$ and $(\gamma,\eta)$ satisfies the Hamilton equations \eqref{eqn:hamilton_equations} on the interval $[0,t]$. 
		
 		If $\gamma$ is an optimal trajectory for $I_t(b)$, then $\eta(0) = I_0'(\gamma(0))$. 
\end{proposition}

\begin{proof}
This follows by \cite[Theorem 6.2.8 and 6.3.3]{CaSi04}. 
\end{proof}

\begin{proposition} \label{proposition:optimizers_equiv_hamflow_CW}
	Let $H$ satisfy Assumption \ref{assumption:trajectories_in_interior} in the setting that $\bK= [-1,1]$, $I_0 \in C^{1,\partial}[-1,1]$ and let $t>0$  and $a,b \in [-1,1]$.  
	\begin{enumerate}
		\item
		\label{item:optimisers_in_interior_interm_times}
		Any optimal trajectory $\gamma$ for $S_t(a,b)$ satisfies $\gamma(s) \in (-1,1)$ for all $s \in (0,t)$.
		\item
		\label{item:optimizers_satisfy_hamilton_equations}
		Let $\gamma$ be an optimal trajectory for $S_t(a,b)$. 
		
		If $b \in (-1,1)$, then $\gamma$ is $C^2$ on $[0,t]$. The dual trajectory $\eta$ is defined on $[0,t]$ is $C^1$ and $(\gamma,\eta)$ satisfies the Hamilton equations \eqref{eqn:hamilton_equations} on the interval $[0,t]$. 

		If $b \in \{-1,1\}$ then $\gamma$ is $C^2$ on $[0,t)$. The dual trajectory $\eta$ is defined on $[0,t)$ is $C^1$ and $(\gamma,\eta)$ satisfies the Hamilton equations \eqref{eqn:hamilton_equations} on $[0,t)$ and $\lim_{s \uparrow t} (\gamma(s),\eta(s)) \in \{(-1,-\infty),(1,\infty)\}$.
		\item
		\label{item:optimisers_in_interior_time_0}
Any optimal trajectory $\gamma$ for $I_t(b)$	satisfies  $\gamma(0) \in (-1,1)$. Let $\eta$ be the dual trajectory as above, then we have $\eta(0) = I_0'(\gamma(0))$. 
	\end{enumerate}
\end{proposition}

\begin{proof}
	\ref{item:optimisers_in_interior_interm_times} follows from  \cite[Proposition 4.9]{Kr16d} and Assumption \ref{assumption:trajectories_in_interior}.
	
	For \ref{item:optimizers_satisfy_hamilton_equations}, first we consider $b\in (-1,1)$. 
	Let $\gamma$ be an optimizer for $S_t(a,b)$. By \ref{item:optimisers_in_interior_interm_times}, we have $\gamma(s) \in (-1,1)$ for $s \in (0,t)$. Thus, for any $\delta \in (0,t)$, the trajectory $\gamma$ restricted to $[\delta,t ]$ is contained in $(-1,1)$ and is the optimal trajectory over all absolutely continuous trajectories $\xi$ on $[\delta,t]$ with $\xi(t) = \gamma(t)$ and $\xi(\delta) = \gamma(\delta)$ of the functional
	\begin{align*}
	\xi \mapsto \int_{\delta}^{t}  \cL(\xi(s),\dot \xi(s)) \dd s.
	\end{align*}
	By Assumption \ref{assumption:Hamiltonian_for_domain_extension} (a), (b) and \cite[Theorem 6.2.8]{CaSi04} (note that $\cL$ is $C^2$ on $(-1,1)\times \R$ by Lemma \ref{lemma:lagrangian_is_C2_on_interior}),
	we find that the restriction of $\gamma$ to $[\delta,t]$ is $C^2$ and solves the Euler-Lagrange equation classically. This can be done for all $\delta > 0$, so $\gamma$ is $C^2$ on $(0,t]$ and solves the Euler-Lagrange equation
	\begin{equation}
	\frac{\dd}{\dd s} \partial_v \cL(\gamma(s),\dot \gamma(s)) = \partial_x \cL(\gamma(s),\dot \gamma(s)), \qquad s \in (0,t].
	\label{eqn:euler_lagrange_eq}
	\end{equation}
Let $\eta : (0,t] \rightarrow \bR$ be the dual path as in \eqref{eqn:dual_path_to_gamma} (note that $\cL$ is $C^2$ by Lemma \ref{lemma:lagrangian_is_C2_on_interior}). 

By \cite[Theorem 6.3.3]{CaSi04} $(\gamma,\eta)$ solves the Hamilton equations for all times $s \in (0,t]$.	 Lemma \ref{lemma:limit_points_in_extended_space}
 extends this to the closed interval $[0,t]$. 
	
	 If $b \in \{-1,1\}$, we can follow the same lines as above, now restricting to trajectories on $[\delta, t - \delta]$ instead of $[\delta,t]$ with $\delta \in (0,\frac{t}{2})$ for example. 
	 For the limit $\lim_{s\uparrow t} (\gamma(s),\eta(s))$ we refer to  Remark \ref{remark:limit_Hamiltonian_path_boundaries}.

	\smallskip

	For \ref{item:optimisers_in_interior_time_0},
	suppose that there exists an optimal trajectory $\gamma$ such that $\gamma(0) = -1$ and $\gamma(t) = b$. The proof for $\gamma(0) = 1$ is similar. By \ref{item:optimisers_in_interior_interm_times} and \ref{item:optimizers_satisfy_hamilton_equations} there exists a $\kappa>0$ such that $\gamma(s) < 1- \kappa$ for all $s\in [0,\frac{t}{2}]$. 

	Let $\rho \in C^1[0,t]$ be such that $\rho(0) > 0$, $\rho=0$ on $[\frac{t}{2},t]$ and such that $\rho$ takes its values in $[0,\kappa]$. Then $\gamma + \epsilon \rho$ attains its values in $[-1,1]$ for all $\epsilon \in [0,1]$.

		We derive a contradiction by showing the existence of a $\epsilon\in (0,1)$ such that $\gamma + \epsilon \rho$ has a lower cost than $\gamma$, i.e., $\cJ(\gamma + \epsilon \rho) < \cJ(\gamma)$ (for $\cI$ see \eqref{eqn:rate_function_paths}). We do this by showing that the difference quotient 
	\begin{equation} \label{eqn:optimisers_in_interior_proof}
	\frac{\cI(\gamma + \varepsilon \rho) - \cI(\gamma)}{\varepsilon} = \frac{I_0(\gamma(0) + \varepsilon\rho(0)) - I_0(\gamma(0))}{\varepsilon} + \frac{J(\gamma + \varepsilon\rho) - J(\gamma)}{\varepsilon},
	\end{equation}
	converges to $-\infty$ as $\epsilon \downarrow 0$, 	
	where $J$ is the path-space cost
	\begin{equation*}
	J(\zeta): = \int_0^t \cL(\zeta(s),\dot{\zeta}(s)) \dd s.
	\end{equation*}
As $I_0 \in C^{1,\partial }[-1,1]$, the first term on the right-hand-side of \eqref{eqn:optimisers_in_interior_proof} converges to $-\infty$. Therefore it is sufficient to show that the second term on the right-hand-side is bounded from above for small $\epsilon$. 
	
%

For $\theta>0$ we write 
\begin{align*}
\psi_\theta(s) = (\gamma(s)+ \theta \rho(s),\dot{\gamma}(s)+ \theta \dot{\rho}(s)). 
\end{align*}
By the mean-value theorem there exists a $\theta_s \in (0,\epsilon)$ for all $s\in [0,t]$ such that 
	\begin{equation*}
	\frac{J(\gamma + \varepsilon\rho) - J(\gamma)}{\varepsilon}
	=
	\int_0^t \partial_x \cL(\psi_{\theta_s}(s)) \rho(s) + \partial_v \cL(\psi_{\theta_s}(s)) \dot{\rho}(s) \dd s.
	\end{equation*}
The integrand and therefore the integral is bounded for all $\epsilon \le \tilde \epsilon$  if 
\begin{align}
\label{eqn:compact_set_in_domain}
\{ \psi_\theta(s) : s\in [0,t], \theta \in [0,\tilde \epsilon]\}
\subseteq (-1,1) \times \R \cup (\{-1\} \times (v_-,\infty)) \cup (\{1\} \times (-\infty,v_+)),
\end{align}
as the latter set is the domain on which  $\partial_x \cL$ and $\partial_v \cL$ are continuous, see Lemma \ref{lemma:lagrangian_is_C2_on_interior}. 
As $(\gamma,\eta)$ satisfies the Hamilton equations, we have $\dot{\gamma}(0) = \partial_p H(\gamma(0),\eta(0)) > v_-$ (by Assumption \ref{assumption:Hamiltonian_for_domain_extension}\ref{item:assumptions_quotient_H}). 
Hence we can choose $\tilde \epsilon$ small enough such that $\dot \gamma(0) + \theta \dot \rho(0) >v_-$ for all $\theta \in [0, \tilde \epsilon]$, which implies that 	\eqref{eqn:compact_set_in_domain} holds.

	Now that it is established that $\gamma$ starts in the interior, we can apply  \cite[Theorem 6.3.3]{CaSi04} to obtain that $\eta(0) = I_0'(\gamma(0))$.
\end{proof}

\begin{proposition} \label{proposition:optimizers_bounded_away_from_boundary}
Let $H$ satisfy Assumption \ref{assumption:Hamiltonian_for_domain_extension} and if $\K = [-1,1]$ additionally assume Assumption \ref{assumption:trajectories_in_interior}. Let $I_0 \in C^{1,\partial}(\K)$ and let $t>0$. 
\begin{enumerate}
		\item
		\label{item:optimal_curves_bounded_away_from_boundary}
For all $w \in (0,\partial_+)$ there exists a $v\in (w,\partial_+)$ such that 
for all $b\in [-w,w]$ and any optimal trajectory $\gamma$ for $I_t(b)$ we have
	$\gamma(s) \in [-v,v]$ for all $s\in [0,t]$.
		\item
		\label{item:optimal_curves_end_far_from_zero}
		For all $m>0$ there exists a $v\in (0,\partial_+)$ such that
		for all $b\in (\partial_-, -v) \cup (v,\partial_+)$ and any  optimal trajectory $\gamma$ for $I_t(b)$ with dual trajectory $\eta$
		\begin{align*}
		\eta(t) \ge m & \quad \mbox{ if } \quad b\in (v,\partial_+), \\
		\eta(t) \le -m & \quad \mbox{ if } \quad b\in  (\partial_-, -v).
		\end{align*}
\end{enumerate}
\end{proposition}
\begin{proof}
	
	\ref{item:optimal_curves_bounded_away_from_boundary} follows from Proposition \ref{prop:trajectories_bounded_away_from_boundary} 
 together with  Proposition \ref{proposition:optimizers_equiv_hamflow_R} (for $\K = \R$) or together with Proposition \ref{proposition:optimizers_equiv_hamflow_CW}\ref{item:optimizers_satisfy_hamilton_equations} (for $\K = [-1,1]$).

	\ref{item:optimal_curves_end_far_from_zero} Fix $n$ such that $q_n^+ \geq m$ and $q_n^- \leq - m$. As $I_0 \in C^{1,\partial}(\K)$, there is a $v\in (0,\partial_+)$ such that 	
	\begin{equation} \label{eqn:rate_function_and_starting_momentum}
	\sup_{x\in (\partial_-,-v)} I'_0(x) \leq q_n^-, \qquad \inf_{x\in (v,\partial_+)} I_0'(x) \geq q_n^+,
	\end{equation}
Let $z_n$ be as in Lemma \ref{lemma:delta_n}. We may and do assume that $v> z_n$. 
	
	Let $b> v$ and $\gamma$ the optimal trajectory for $I_b(b)$ and $\eta$ its dual trajectory. We show that $\eta(t) \ge y_n^+$ (in a similar way one proves that $b< -v$ implies $\eta(t) \le y_n^-$). 
	
	Suppose that $\gamma(0) > v$. Then $I_0'(\gamma(0)) > q_n^+$ by \eqref{eqn:rate_function_and_starting_momentum}. 
As $\gamma(0) > 1- \delta > 1- \delta_n > y_n^+$ (see \eqref{eqn:bound_delta_n_and_q_n}), $(\gamma,\eta)$ starts and, therefore, stays in the quadrant $\upq_{y_n^+,q_n^+}$, which implies $\eta(t) \ge q_n^+$. 
	
	Suppose that $\gamma(0) \leq v$. 	
	Because the complement of the region $\{x: x\ge v\} \times (-\infty,q_n^+]$ is preserved (by Lemma \ref{lemma:delta_n}), the Hamiltonian trajectory $(\gamma,\eta)$ cannot enter this region. Because $\gamma(t) = b \ge v > z_n \ge y_n^+$, this implies that $\eta(t) \geq q_n^+$.	
\end{proof}

\section{Properties of the time-evolved rate function}\label{section:proofs_regularity_rate_function}

To rigorously study the time-dependent rate function $u(t,x) := I_t(x)$ and $x \mapsto I_t(x)$  for both the $\pm 1$-space-model and the $\R$-space-model, we establish local semi-concavity and the boundary behaviour. The local semi-concavity implies we can use sub-gradients as in the classical theory and conclude that the push-forward of the gradient of the starting rate function  under the Hamiltonian flow  contains the reachable gradients of $I_t$.

\begin{lemma} \label{lemma:classically_semi_concave}
Assume Assumption \ref{assumption:general_ones_on_R_and_CW}. Then $I_t$ is locally semi-concave on $\K^\circ$ for all $t \geq 0$. Moreover, $(t,x) \mapsto I_t(x) = u(t,x)$ is  locally semi-concave on $(0,\infty) \times \K^\circ$. 
\end{lemma}
	\begin{proof}
Let $w\in (0,\partial_+)$ and $v\in (w,\partial_+)$ be as in Proposition \ref{proposition:optimizers_bounded_away_from_boundary}.
Let $x,y \in \K^\circ$ be such that $x<y$, $[x,y] \subseteq[-w,w]$ and $q:=v+y-x<\partial_+$. 
We show that $I_t$ is semi-concave on $[x,y]$ (this is sufficient as for all $z\in \K^\circ$ there exist $w,x,y$ as above with $z\in [x,y]$). 

Let $\lambda \in [0,1]$ and $\xi$ be an optimal trajectory for $I_t(\lambda x + (1-\lambda)y)$. 
By the choice of $v$, we have that $\xi$ attains its values in $[-v,v]$.
In addition, consider the trajectories $\xi_x, \xi_y$ defined by
		\begin{equation*}
		\xi_x (s) = \xi(s) + \frac{s}{t} 
			(1-\lambda) 
		\left(x-y\right), \qquad 	\xi_y (s) = \xi(s) + \frac{s}{t}
		\lambda 
		\left(y-x\right).
		\end{equation*}
		Note that $\xi_x(t) = x$ and $\xi_y(t) = y$, $\lambda \xi_x + (1-\lambda) \xi_y = \xi$ and $\xi_x(0)=\xi(0)=\xi_y(0)$. 
		The trajectories $\xi_x,\xi_y$ and $\xi$ take their values in $[-q,q]$. As $\cL$ is $C^2$ on $\bK^\circ \times \R$ by Lemma \ref{lemma:lagrangian_is_C2_on_interior}, we find that $\cL$ is semi-concave on $[-q,q] \times \R$. Thus the first statement follows similarly as in the proof of \cite[Theorem 6.4.3]{CaSi04}:
\begin{align*}
& \lambda I_t(x) + (1-\lambda) I_t(y) - I_t\left(\lambda x + (1-\lambda)y\right) \\
& \le   \int_0^t \lambda \cL(\xi_x(s), \dot \xi_x(s) )  
+ (1- \lambda)  \cL(\xi_y(s), \dot \xi_y(s) )  
+ \cL(\xi(s), \dot \xi(s) ) \dd s \\
& \le \int_0^t C \frac{\lambda(1-\lambda)}{2}   \frac{s^2+1}{t^2} ( (1-\lambda)^2 + \lambda^2)|y-x|^2 \dd s
 = 2C \left( \frac{t}{3}+ \frac{1}{t} \right) \frac{\lambda(1-\lambda)}{2} |x-y|^2.
\end{align*}			
		 The `moreover' statement follows along the same lines as the proof of \cite[Corollary 6.4.4]{CaSi04}.

	\end{proof}

\begin{proposition} \label{proposition:push_forward_graph_contains_reachable_grads}
Assume assumption \ref{assumption:general_ones_on_R_and_CW}. Let $t>0$. 
	
	\begin{enumerate}
		\item Fix $b \in \K^\circ$ and let $\gamma$ be an optimizer for $I_t(b)$ and let $\eta$ be the dual trajectory, then $\eta(t) \in D^+ I_t(b)$.
		\item For any $b \in \K^\circ$ the map that associates to $(p_t,p) \in D^* u(t,b)$ the Hamiltonian trajectory $(\gamma,\eta)$ with terminal condition $\gamma(t) = b$ and $\eta(t) = p$ provides a one-to-one correspondence between $D^* u(t,b)$ and optimizers of $I_t(b)$.
		\item We have
		\begin{equation*}
		\left\{(x,p) \, \middle| \, x \in \K^\circ, \, p \in D^* I_t(x)\right\} \subseteq \cG_t.
		\end{equation*}
		For any $b \in \K^\circ$ and any element $p \in D^* I_t(b)$, the Hamiltonian trajectory $(\gamma,\eta)$ with terminal condition $\gamma(t) = b$ and $\eta(t) = p$ yields an optimal trajectory $\gamma$ for $I_t(b)$.
	\end{enumerate}
\end{proposition}
\begin{proof}
	(a) and (b) can be proven as \cite[Theorems 6.4.8 and 6.4.9]{CaSi04} using that optimizers are bounded away from the boundary by Proposition  \ref{proposition:optimizers_bounded_away_from_boundary} \ref{item:optimal_curves_bounded_away_from_boundary}.
	
	The proof of (c) uses a variation on the ideas in the proof of \cite[Theorem 6.4.9]{CaSi04}. Let $x \in \K^\circ$ and $p \in D^* I_t(x)$. By definition, there are $(x_k,p_k) \in \K^\circ \times \bR$ such that $(x_k,p_k) \rightarrow (x,p)$, $I'_t$ is differentiable at $x_k$ and $p_k = I'_t(x_k)$.
	
	\smallskip
	
	Let $\gamma_k : [0,t] \rightarrow \K$ be an optimizing trajectory for $I_t(x_k)$ and let $\eta_k$ be the dual trajectory. By Propositions \ref{proposition:optimizers_equiv_hamflow_R} and \ref{proposition:optimizers_equiv_hamflow_CW} $(\gamma_k,\eta_k)$ satisfies Hamilton's equations and $\eta_k(0) = I_0'(\gamma_k(0))$. \cite[Theorem 6.4.8]{CaSi04} implies that $(\gamma_k(t),\eta_k(t)) = (x_k,p_k)$. 
 	By continuous dependence on the final conditions, we have $(\gamma_k(s),\eta_k(s)) \rightarrow (\gamma(s),\eta(s))$ for all $s\in [0,t]$ (see \cite[Theorem 2.3.2]{Pe01}),
	 where $(\gamma,\eta)$ solves the Hamilton equations with $(\gamma(t),\eta(t)) = (x,p)$. 
	 By continuity of  $\partial_p H$, we also obtain  $\dot \gamma_k(s) \rightarrow \dot \gamma(s)$ for all $s\in [0,t]$. 	
	 We obtain that indeed $(x,p) \in \cG_t$. By continuity of $I_0,I_t$ and $\cL$ we find additionally that
	\begin{align*}
	I_t(x) & = \lim_{k \rightarrow \infty} I_t(x_k) \\
	& = \lim_{k \rightarrow \infty} I_0(\gamma_k(0)) + \int_0^t \cL(\gamma_k(s),\dot{\gamma}_k(s)) \dd s \\
	& =  I_0(\gamma(0)) + \int_0^t \cL(\gamma(s),\dot{\gamma}(s)) \dd s,
	\end{align*}
	thus $\gamma$ is an optimizer for $I_t(x)$.
\end{proof}

\begin{lemma}
	\label{lemma:exploding_derivatives_at_boundary}
Assume Assumption \ref{assumption:general_ones_on_R_and_CW}. Then 
	\begin{align*}
	\sup_{a \in  (0,\partial_+)} \inf_{b \ge a} \inf D^+ I_t(b) = \infty,
	\qquad
	\inf_{ a \in  (\partial_-,0)} \sup_{b \le a}  \sup D^+ I_t(b) = - \infty.
	\end{align*}
\end{lemma}

\begin{proof}
	First note that $D^+I_t(b) =  \text{co}D^*I_t(b)$. By Proposition \ref{proposition:push_forward_graph_contains_reachable_grads} (c) $p \in D^* I_t(b)$ corresponds to a trajectory $(\gamma,\eta)$ that solves the Hamilton equations with terminal condition $(\gamma(t),\eta(t)) = (b,p)$, the rest follows by Proposition \ref{proposition:optimizers_bounded_away_from_boundary} \ref{item:optimal_curves_end_far_from_zero}.
\end{proof}

We end this section with the proof of the statement that $u$ is a viscosity solution of the Hamilton-Jacobi equation \eqref{eqn:hamilton_jacobi}.

\begin{proof}[Proof of Proposition \ref{proposition:solution_HJ_equation}]
	The proof follows as in the proof of \cite[Theorem 6.4.5]{CaSi04}, using Proposition  \ref{proposition:optimizers_bounded_away_from_boundary} \ref{item:optimal_curves_bounded_away_from_boundary} to make sure that optimal trajectories remain in compact sets bounded away from the boundary and Lemma \ref{lemma:classically_semi_concave} to make sure that $u$ is locally semi-concave on $\bK^\circ \times (0,\infty)$. 
\end{proof}


\section{Topological structure of \texorpdfstring{$\cG_t$}{}} \label{section:topological_properties_of_Gt}

In this section we study the structure of $\cG_t$. Inspired by the result of Proposition \ref{lemma:after_lifetime_in_boundary} we introduce an extension of the Hamiltonian flow to include the points $(\partial_-,-\infty)$ and $(\partial_+,\infty)$. In this way we make sense of Hamilton paths starting at $(x,p)$ for times $t>t_{x,p}$. In addition, we extend $\cG$ to include the points at infinity and thus we can use the properties of connected sets to study the structure of $\cG_t$ in Proposition \ref{proposition:hamiltonian_paths_and_graph_structures}.



\begin{proposition} \label{proposition:extension_of_flow_with_bdr}
	Assume assumption \ref{assumption:general_ones_on_R_and_CW}. Let 
	\begin{equation*}
	E := \left\{(t,x,p) \in [0,\infty) \times \K \times \bR : t < t_{x,p} \right\}.
	\end{equation*} 
We extend the space $\K \times \R$ with two points and write $\fS := \K \times \bR \cup \{(\partial_-,-\infty)\} \cup \{(\partial_+,\infty)\}$. We equip $\fS$ with the topology generated by the open subsets of $\K \times \R$ together with the sets $\upq_{a,b}^\circ \cup \{(\partial_+,\infty)\}$  and $\downq_{a,b}^\circ \cup \{(\partial_-,-\infty)\}$ for $a\in \K^\circ$ and $b\in \R$. 
	
	\begin{enumerate}
		\item \label{item:lower_semi_cont_of_max_time} The map $(x,p) \mapsto t_{x,p}$ is lower semi-continuous and $E\mapsto \K \times \R$, $(t,x,p)\mapsto (X_t^{x,p},P_t^{x,p})$ is continuous. 
		\item \label{item:extension_of_flow} 
		The map $\Psi : [0,\infty) \times \fS \rightarrow \fS$ defined by\footnote{Note that $\lim_{s \uparrow t_{x,p}} (X_s^{x,p},P_s^{x,p})$ is welldefined by Lemma \ref{lemma:after_lifetime_in_boundary}.} 
			\begin{equation*}
	\Psi(t,x,p) := \begin{cases}
	(X_t^{x,p},P_t^{x,p}) & \text{if } (t,x,p) \in E, \\
	\lim_{s \uparrow t_{x,p}} (X_s^{x,p},P_s^{x,p}) & \text{if } (x,p) \in \K \times \bR, \, t \geq t_{x,p}, \\
	(\partial_-,-\infty) & \text{if } (x,p) = (\partial_-,-\infty), \\
	(\partial_+,\infty) & \text{if } (x,p) = (\partial_+,\infty). \\
	\end{cases}
	\end{equation*}
	is continuous. 
	\end{enumerate} 
\end{proposition}

\begin{proof}
\ref{item:lower_semi_cont_of_max_time} follows from \cite[Theorem 2.5.1]{Pe01} or \cite[Theorem 2.8]{Te12} for example (as those theorems apply to open sets, in case $\K = [-1,1]$ consider $(-1-\epsilon,1+\epsilon)$ for $\epsilon>0$ and $\hat H$ as in Assumption \ref{assumption:Hamiltonian_for_domain_extension}(a) instead of $K$ and $H$).

\ref{item:extension_of_flow}
The continuity on $E$ follows from \ref{item:lower_semi_cont_of_max_time}. 
The continuity of $\Psi$ at $(t,\partial_-,-\infty)$ and $(t,\partial_+,\infty)$ follows by the preservation of quadrants, as mentioned in Remark  \ref{remark:assumptions_implies_preserved_quadrants} and Remark \ref{remark:limit_Hamiltonian_path_boundaries}.  
Therefore we are left to prove continuity of $\Psi$ at $(t,x,p)$ with $t\ge t_{x,p}$. 
We may restrict ourselves to sequential continuity because the topology on $\fS$ is metrizable by Urysohn's metrization theorem (e.g. \cite[Theorem 4.1.10]{Ru05}); the topology is second countable and normal (we leave it to the reader to check those properties). 

Assume $(t,x,p)$ is such that $t\ge t_{x,p}$ and $\Psi(t,x,p) = (\partial_+,\infty)$ (the case $\Psi(t,x,p) = (\partial_-,-\infty)$ is similar). 
Suppose that $(x_n,p_n,t_n) \rightarrow (x,p,t)$. 

It is sufficient to show that for all $k\in\N$ there exists an $N$ such that for all $n\ge N$  
\begin{align}
\label{eqn:hamilton_in_quadrant_for_large_n}
 (X_{t_n}^{x_n,p_n}, P_{t_n}^{x_n,p_n}) \in \upq_{y_k,q_k} \cup \{(\partial_+,\infty)\}.
\end{align}
Let $k\in \N$ and $s< t_{x,p}$ be such that 
\begin{align*}
(X_s^{x,p},P_s^{x,p}) \in \upq_{y_k,q_k}.
\end{align*}
By \ref{item:lower_semi_cont_of_max_time} we have $\liminfn t_{x_n,p_n} \ge t_{x,p}$. 
Hence there exists a $N_0$ such that $t_{x_n,p_n} > s$ for all $n\ge N_0$. As $s< t_{x_n,p_n}$ and $s<t_{x,p}$ we have by \ref{item:lower_semi_cont_of_max_time}
\begin{align}
\label{eqn:specificly_used_convergence_hamilton_paths}
\limn (X_s^{x_n,p_n}, P_s^{x_n,p_n})  = (X_s^{x,p},P_s^{x,p}) \in \upq_{y_k,q_k}. 
\end{align}
Therefore there exists an $N_1>N_0$ such that $(X_s^{x_n,p_n}, P_s^{x_n,p_n})  \in \upq_{y_k,q_k}$ for all $n\ge N_1$. 
Let $N >N_1$ be such that $t_n > s$ for all $n\ge N$.
As $\upq_{y_k,q_k}$ is preserved under the Hamiltonian flow and $t_n >s$ for $n\ge N$ we obtain that \eqref{eqn:hamilton_in_quadrant_for_large_n} holds for all $n\ge N$. 
\end{proof}

\begin{proposition}
\label{proposition:hamiltonian_paths_and_graph_structures}
Assume Assumption \ref{assumption:general_ones_on_R_and_CW}. 
Define $E_{t} := \{ x\in  \K :  t_{x,I_0'(x)} > t \}$ and  let $\Phi_{t} : E_{t} \rightarrow \K \times \R$ be defined by $\Phi_{t}(x) :=\left(X_t^{x,I_0'(x)},P_t^{x,I_0'(x)}\right)$. Note that $\cG_t = \Phi_{t}(E_{t})$.

\begin{enumerate}

\item 
\label{item:E_t_is_open}
$E_t$ is open for all $t$, $\Phi_t$ is continuous and injective on $E_t$. 

\item
\label{item:max_interval_in_E_t}
Let $t>0$ and suppose that $(a,b) \subseteq E_t$, $a, b \notin E_t$. Then $\Phi_t(a,b)$ is path-connected and both $\lim_{z \downarrow a} \Phi_t(z)$ and $\lim_{z \uparrow b} \Phi_t(z)$ are in $\{(\partial_-,-\infty),(\partial_+,\infty)\}$. 

\item 
\label{item:consecutive_intervals}
If $(a,b), (c,d)$ are consecutive maximally (as in \ref{item:max_interval_in_E_t}) connected components in $E_t$, then $\lim_{z \uparrow b} \Phi_t(z) = \lim_{z \downarrow c} \Phi_t(z)$.

\item 
\label{item:Phi_at_boundary}
Let 	\begin{equation*}
	z_- := \inf \left\{ z \in \bK^\circ \, \middle| \, z \in E_t\right\}, \qquad z_+ := \sup \left\{ z \in \bK^\circ \, \middle| \, z \in E_t\right\}.
	\end{equation*}
Then  $\lim_{z \downarrow z_-} \Phi_t(z) = (\partial_-,-\infty)$ and $\lim_{z \uparrow z_+} \Phi_t(z) = (\partial_+,\infty)$ for all $t>0$ .

\item
\label{item:interval_connecting_boundaries}
There exists a maximal (as in \ref{item:max_interval_in_E_t}) connected component $(a,b)$ in $E_t$ that connects $(\partial_-,-\infty)$ to $(\partial_+,\infty)$, i.e., 
$
		\lim_{z \downarrow a} \Phi_t(z) =(\partial_-,-\infty) 
		$ and $
		 \lim_{z \uparrow b} \Phi_t(z) = (\partial_+,\infty).
$
\item 
\label{item:gamma_t}  
Let $\gamma_t : E_t \rightarrow \K$ be defined by $\gamma_t(x) := X_t^{x,I_0'(x)}$. 
 $\gamma_t$ is strictly increasing on $E_t$ if and only if  $\cG_t$ is a graph. 
\end{enumerate}
\end{proposition}
\begin{proof}
Note that $I'_0$ is continuous and that $\Phi_t(x) = \Psi(t, x, I'_0(x))$ for $x\in E_t$, with $\Psi$ as in Proposition \ref{proposition:hamiltonian_paths_and_graph_structures}. 
By this one deduces \ref{item:E_t_is_open}, \ref{item:max_interval_in_E_t} and \ref{item:consecutive_intervals}. 
\ref{item:Phi_at_boundary} follows from preservation of quadrants and the fact that $I_0 \in C^{1,\partial}(\K)$. \ref{item:interval_connecting_boundaries} then follows as $\K \rightarrow \fS$ given by $ x \mapsto \Psi(t, x, I'_0(x))$ is continuous and connects $(\partial_-,-\infty)$ with $(\partial_+,\infty)$ by \ref{item:Phi_at_boundary}. 
\ref{item:gamma_t} follows from the previous. 
\end{proof}

\begin{remark}
\label{remark:overhangs_by_multiple_connected_components_E_t}
By Proposition \ref{proposition:hamiltonian_paths_and_graph_structures}\ref{item:interval_connecting_boundaries}, we see that if $E_t$ has at least two disconnected components, then $\cG_t$ has an overhang at an $x_0$.  
But in this situation, generally one cannot find $x_1,x_2$ at which $\cG_t$ has no overhang, with $x_1<x_0<x_2$,   so that we cannot apply Proposition \ref{proposition:graph_of_derivative_pushforward} to conclude that $I_t$ is not differentiable. 
\end{remark}

\section{Dynamics that preserve order} \label{section:abstract_convexity_preserving}


In this section, we give the proof of Theorem \ref{theorem:convexity_preserving_implies_diffb_preserv} and give sufficient conditions on $H$ for it to preserve order (in Proposition \ref{proposition:ordering_is_preserved},  see Definition \ref{definitions:H_preservation_convexity_I_convexity}\ref{item:preserves_order}). 
First we prove in Lemma \ref{lemma:ordering_is_preserved_basic} that if the order of the momentum is preserved along some time interval, that the order of the position is also preserved.  

%
%
%


\begin{lemma} \label{lemma:ordering_is_preserved_basic}
	Assume Assumption \ref{assumption:general_ones_on_R_and_CW}. Consider two pairs $(x_1,p_1),(x_2,p_2) \in \bK\times \bR$ and a time $t^* \in [0,\infty]$ with $t^* \leq t_{x_1,p_1} \wedge t_{x_2,p_2}$ such that $x_1 < x_2$ and $P_t^{x_1,p_1} < P_t^{x_2,p_2}$ for all $0 < t < t^*$.
	Then $X_t^{x_1,p_1} < X_t^{x_2,p_2}$ for all $t < t^*$.
\end{lemma}

\begin{proof}
	Assume the contrary, that there is some $t < t^*$ such that $X_t^{x_1,p_1} \geq X_t^{x_2,p_2}$.	Then there exists a $0 < t_0 < t^*$ such that
	\begin{equation*}
	X^{x_1,p_1}_t< X^{x_2,p_2}_t \text{ for all } t< t_0 \text{ and }
	X^{x_1,p_1}_{t_0}= X^{x_2,p_2}_{t_0}=a.
	\end{equation*}
	Therefore
	$
	\partial_p H(a,P^{x_1,p_1}_{t_0}) = \dot{X}^{x_1,p_1}_{t_0} \geq  \dot{X}^{x_2,p_2}_{t_0}= \partial_p H(a,P^{x_2,p_2}_{t_0})
	$ 
	from which we conclude, by strict monotonicity of $\partial_p H$ in the $p$ variable that $P^{x_1,p_1}_{t_0} \geq P^{x_2,p_2}_{t_0}$ contradicting our assumption.
\end{proof}

\begin{proposition} \label{proposition:ordering_is_preserved}
Assume Assumption \ref{assumption:general_ones_on_R_and_CW}. 
Suppose that  $A \subseteq \K \times \bR$ is preserved under $H$ and that either (i) or (ii) holds: 
		\begin{enumerate}[label=\emph{(\roman*)}]
			\item 
			\label{item:partial_x_strictly_decr_in_x}
			for all $p$, we have $x \mapsto \partial_x H(x,p)$
			is strictly decreasing on  $\{x  \, : \,  (x,p) \in A\}$.
			\item 
			\label{item:partial_x_decreasing_in_x_p}
			the map $(x,p) \mapsto \partial_x H(x,p)$ is decreasing in both coordinates on the set $A$.
		\end{enumerate}
Then $H$ preserves order on $A$. 
\end{proposition}
	\begin{proof}
		Assume the contrary, that there is some $s <t_{x_1,p_1} \wedge t_{x_2,p_2}$
		 such that $X_{s}^{x_1,p_1} \geq X_{s}^{x_2,p_2}$ or $P_{s}^{x_1,p_1} \geq P_{s}^{x_2,p_2}$. We assume that $t_0$ is the smallest of such times $s$. Then we have exactly one of the following two situations as the Hamiltonian trajectories do not meet:
		\begin{enumerate}
			\item $P^{x_1,p_1}_t< P^{x_2,p_2}_t $ for all $t \leq t_0$ and
			\begin{equation*}
			X^{x_1,p_1}_t< X^{x_2,p_2}_t \text{ for all } t< t_0 \text{ and }
			X^{x_1,p_1}_{t_0}= X^{x_2,p_2}_{t_0}=a.
			\end{equation*}
			\item $X^{x_1,p_1}_t< X^{x_2,p_2}_t $ for all $t \leq t_0$ and
			\begin{equation}\label{eqn:P_1_less_then_P_2_all_t_small}
			P^{x_1,p_1}_t< P^{x_2,p_2}_t \text{ for all } t< t_0 \text{ and }
			P^{x_1,p_1}_{t_0}= P^{x_2,p_2}_{t_0}=q.
			\end{equation}
		\end{enumerate}
		If (a) holds then we have a contradiction by Lemma \ref{lemma:ordering_is_preserved_basic}. If (b) holds we first proceed by using variant \ref{item:partial_x_strictly_decr_in_x}.  
		As in the proof of the lemma above, we find
		$
		- \partial_x H(X^{x_1,p_1}_{t_0}, p)=\dot{P}^{x_1,p_1}_{t_0}\geq  \dot{P}^{x_2,p_2}_{t_0}= -\partial_x H(X^{x_2,p_2}_{t_0}, p)
		$ 
		which by the assumed strict monotonicity of	$\partial_x H$ on $A$ implies that $X^{x_1,p_1}_{t_0} > X^{x_2,p_2}_{t_0}$.	This is in contradiction with (b). 
		
		Next, we use variant \ref{item:partial_x_decreasing_in_x_p}. We find
		\begin{equation*}
		p_2 - p_1 = P_{0}^{x_2,p_2} - P_{0}^{x_1,p_1} = \int_{0}^{t_0} \partial_x H(X_{s}^{x_2,p_2},P_{s}^{x_2,p_2}) - \partial_x H(X_{s}^{x_1,p_1},P_{s}^{x_1,p_1})  \dd s.
		\end{equation*}
		By assumption, the left-hand-side is strictly greater than $0$.	For $0 \leq s \leq t_0$, however, we have $X_{s}^{x_1,p_1} < X_{s}^{x_2,p_2}$ and $P_{s}^{x_1,p_1} < P_{s}^{x_2,p_2}$, which by \ref{item:partial_x_decreasing_in_x_p} is non-positive. 
	\end{proof}

\begin{proof}[Proof of Theorem \ref{theorem:convexity_preserving_implies_diffb_preserv}]

For both \ref{item:conv_implies_diffb} and \ref{item:conv_at_infinity_implies_short_time} we use the following for $t>0$. 
By Proposition \ref{proposition:hamiltonian_paths_and_graph_structures}\ref{item:gamma_t} it follows that if $\gamma_t$ is strictly increasing, then $\cG_t$ is a graph and thus $I_t \in C^{1,\partial}(\K)$ by Theorem \ref{theorem:equivalences_differentiability}. 

Theorem \ref{theorem:convexity_preserving_implies_diffb_preserv}\ref{item:conv_implies_diffb} follows then immediately from Proposition \ref{proposition:ordering_is_preserved} as convexity of a differentiable function implies that its derivative is increasing. 

Theorem \ref{theorem:convexity_preserving_implies_diffb_preserv}\ref{item:conv_at_infinity_implies_short_time}. 
Let $a,b>0$, $a\in \K^ \circ$ be such that $H$ preserves order on $\downq_{-a,-b,}$ and on $\upq_{a,b}$, $I_0'([-a,a]) \subseteq [-b,b]$ and $I_0$ is convex on $(\partial_-,-a]$ and on $[a,\partial_+)$. 

That $\gamma_t$ is strictly increasing outside $(-a,a)$ follows by the assumed strict convexity of $I_0$ and that $H$ preserves order by using Proposition \ref{proposition:ordering_is_preserved}. 
We will show that $\gamma_t$ is differentiable and that $\gamma_t'(x)>0$ for all $x\in [-a,a]$ and all $t$ small enough. 

As $[-a,a]\times [-b,b]$ is compact, by Proposition \ref{proposition:extension_of_flow_with_bdr}\ref{item:lower_semi_cont_of_max_time} 
there exists a $t^*>0$ such that $U:=(0,t^*) \times (-a,a) \times (-b,b) \subseteq E$. 
By \cite[Theorem 2.5.1]{Pe01} the map $\Psi$ is $C^1$ on $U$. 
As $x\mapsto (x,I'_0(x))$ is $C^1$ by assumption, the map $(t,x) \mapsto \Psi(t,x,I_0'(x))$ is $C^1$ on $(0,t^*)\times [-a,a]$. 
By definition $\gamma_t(x)$ is the first coordinate of $\Psi(t,x,I_0'(x))$ and therefore $(t,x)\mapsto \gamma_t(x)$ is also $C^1$ on $(0,t^*)\times [-a,a]$. 

Note that $t\mapsto \inf_{x\in [-a,a]} \gamma_t'(x)$ is continuous and equal to $1$ for $t=0$. Therefore there exists a $0<t_0\le t^*$ such that $\gamma_t'(x)>0$ for all $0<t<t_0$ and $x\in [-a,a]$. 
\end{proof}

\section{Occurrence of non-differentiability} \label{section:creation_of_overhangs_proofs}

In this section, we introduce the methods necessary for the proofs of Theorem \ref{theorem:loss_of_differentiability}.

\subsection{Linearization around stationary points} \label{section:linearization_around_stationary_points}

We start by studying linearizations of the Hamiltonian flows around stationary points. In contrast to homeomorphisms between flows, see e.g. \cite{Pe01}, $C^1$ diffeomorphisms are difficult to construct. We will refer to \cite{GuHaRa03} for one such construction. 

\begin{theorem} \cite[Theorem 3]{GuHaRa03} \label{theorem:existence_linearization}
Let $H$ satisfy Assumption \ref{assumption:general_ones_on_R_and_CW} and be $C^\infty$. Assume that $x_0$ is a stationary point and that $m \neq 0$ ($m$ as in \eqref{eqn:definition_m_and_c}). Then $H$ admits a $C^1$ linearization at $(x_0,0)$. 
\end{theorem}

\begin{proof}[Proof of Theorem \ref{theorem:loss_of_differentiability} \ref{item:loss_linearization}]
We show that $\gamma_t'(x_0) <0$, which establishes an overhang at $x_0$ by Proposition \ref{proposition:hamiltonian_paths_and_graph_structures}\ref{item:gamma_t}. 

We first establish this for the case that $m \neq 0$. The linearized system \eqref{eqn:linearised_system} with $(\xi_x(0),\zeta_x(0)) := (x,x I''_0(x_0))$ is solved by
	\begin{equation*}
	\begin{bmatrix}
	\xi_x(t) \\ \zeta_x(t) 
	\end{bmatrix}
	= \exp \left(t \begin{bmatrix}
	m & c \\ 0 & -m
	\end{bmatrix}\right) 
	\begin{bmatrix}
	x \\ x I''_0(x_0)
	\end{bmatrix}
	= 
	x \begin{bmatrix}
	e^{mt}  + \frac{c}{2m} (e^{mt} - e^{-mt})   I''_0(x_0) \\ e^{-mt}  I''_0(x_0)
	\end{bmatrix}.
	\end{equation*}
$e^{mt} + \frac{c}{2m} (e^{mt} - e^{-mt} ) I''_0(x_0) <0 $ if $t>t_0$, whence  $\partial_x \xi_x(t) |_{x=x_0} < 0$.  
Denote by $\Theta$ the first component of $\Psi^{-1}$,  the inverse of the linearization  (see Definition \ref{definition:linearization}). As $D \Psi ( x_0,0) = \1$, the identity matrix, by the inverse function theorem we have $D \Psi^{-1} (x_0,0) = \1$ and thus $\partial_{\xi} \Theta(x_0,0) = 1$ and $\partial_{\zeta} \Theta(x_0,0) =0$. Therefore 
\begin{align*}
\gamma_t'(x_0) 
& = \partial_\xi \Theta( \xi_{x_0}(t), \zeta_{x_0}(t)) 
\partial_x
\xi_x(t) |_{x=x_0} \\
& \qquad + \partial_\zeta \Theta( \xi_{x_0}(t), \zeta_{x_0}(t))  
\partial_x \zeta_x(t) |_{x=x_0} < 0,
\end{align*}
	
	\smallskip
	
	In case $m= 0$, the linearized system \eqref{eqn:linearised_system} with $(\xi_x(0),\zeta_x(0)) := (x,x I''_0(x_0))$ is solved by
	\begin{equation*}
	\begin{bmatrix}
	\xi_x(t) \\ \zeta_x(t) 
	\end{bmatrix}
	= \exp \left(t \begin{bmatrix}
	0 & c \\ 0 & 0
	\end{bmatrix}\right) 
	\begin{bmatrix}
	x \\ x I''_0(x_0)
	\end{bmatrix}
	= 
	x \begin{bmatrix}
	1 + tc I''_0(x_0) \\  I''_0(x_0)
	\end{bmatrix}.
	\end{equation*}
As 	$1+tc I_0''(x_0) <0 $ for $t>t_0$, again we obtain $\partial_x \xi_x(t) < 0$ so that the result follows as for the case $m \neq 0$.

	\begin{calculations}
		\begin{align*}
		\begin{bmatrix}
		m & c \\
		0 & -m 
		\end{bmatrix}
		\begin{bmatrix}
		1 & 1 \\
		0 & -\frac{2m}{c} 
		\end{bmatrix}
		= \begin{bmatrix}
		1 & 1 \\
		0 & -\frac{2m}{c} 
		\end{bmatrix}
		\begin{bmatrix}
		m & 0 \\
		0 & -m 
		\end{bmatrix}
		\end{align*}
		and
		\begin{align*}
		\begin{bmatrix}
		1 & 1 \\
		0 & -\frac{2m}{c} \\
		\end{bmatrix}^{-1}
		= \begin{bmatrix}
		1 & \frac{c}{2m} \\
		0 & -\frac{c}{2m}
		\end{bmatrix}
		\end{align*}
		so that
		\begin{align*}
		\begin{bmatrix}
		m & c \\
		0 & -m 
		\end{bmatrix}
		= \begin{bmatrix}
		1 & 1 \\
		0 & -\frac{2m}{c} 
		\end{bmatrix}
		\begin{bmatrix}
		m & 0 \\
		0 & -m 
		\end{bmatrix}
		\begin{bmatrix}
		1 & \frac{c}{2m} \\
		0 & -\frac{c}{2m}
		\end{bmatrix}
		\end{align*}
		and 
		\begin{align*}
		\exp\left( t 
		\begin{bmatrix}
		m & c \\
		0 & -m 
		\end{bmatrix} \right)
		& 
		= \begin{bmatrix}
		1 & 1 \\
		0 & -\frac{2m}{c} 
		\end{bmatrix}
		\begin{bmatrix}
		e^{tm} & 0 \\
		0 & e^{-tm}
		\end{bmatrix}
		\begin{bmatrix}
		1 & \frac{c}{2m} \\
		0 & -\frac{c}{2m}
		\end{bmatrix} 
		\\
		&= \begin{bmatrix}
		1 & 1 \\
		0 & -\frac{2m}{c} 
		\end{bmatrix}
		\begin{bmatrix}
		e^{tm} & \frac{c}{2m} e^{tm} \\
		0 & - \frac{c}{2m} e^{-tm}
		\end{bmatrix}
		= 
		\begin{bmatrix}
		e^{tm} & \frac{c}{2m} ( e^{tm} - e^{-tm}) \\
		0 & e^{-tm} 
		\end{bmatrix}
		\end{align*}
	\end{calculations}
\end{proof}

\begin{remark}
Note that we actually proved that $\{ y \in \R: (x_0,y) \in \cG_t\} \ge 3$ has at least three elements. 
\end{remark}

\subsection{Rotating areas in the Hamiltonian flow}

In this section, we study `rotating  loops' in the Hamiltonian flow, introduced in Definition \ref{def:loop}, and their connection to the emergence of points of non-differentiability. We give a short overview:

\begin{itemize}
	\item Under the Condition of Theorem \ref{theorem:loss_of_differentiability} \ref{item:loss_rotating}, we find a rotating loop in the Hamiltonian flow that intersects $\cG$:  solutions of the flow on this loop come back to the same point in some finite time. (Lemma \ref{lemma:regular_rotating_region}.)
	\item Consider some fixed rotating loop. A homotopy argument, considering the punctured space $\bK \times \bR$ without the interior of this loop, can be used to show that any intersection of the loop with the graph $\cG$ of the gradient of a rate function $I_0$ implies that for large $t$ there is an element of $\cG_t$ `above' the loop. Similarly, one can show that for large $t$ there is an element of $\cG_t$ `below' the loop. This will lead to the creation of an overhang. (Theorem \ref{theorem:above_and_below_lines}.)
	\item First, we characterise rotating loops by their insides. (Lemma \ref{lemma:partial_x_ne_0_then_less_or_larger})
\end{itemize}

\begin{definition}
\label{def:loop}
We call a set $\fL \subset \K \times \R$ a \emph{loop} if there exist $a,b\in \K$, $a<b$ and functions $g,h\in C[a,b]$ with $g(x) < h(x)$ for $x\in (a,b)$, $g(a)=h(a)$ and $g(b)=h(b)$ such that $\fL = g[a,b] \cup h[a,b]$. 

Let $(x,p) \in \K \times \R$ be a non-stationary point, $E=H(x,p)$. 
If there exists a $t_0>0$ such that $(X_{t_0}^{x,p},P_{t_0}^{x,p})= (x,p)$, then $\fL = \{(X_t^{x,p},P_t^{x,p}) : t\in [0,t_0]\}$ is called a \emph{rotating loop}. 
$E$ is called the \emph{energy} of $\fL$. 
Let $a$ be the largest value and $b$ be the smallest value for which $\fL \subseteq [a,b] \times \R$. 
The set $A= \{(x,p) \in [a,b]\times \R : H(x,p) <E\}$ is called the 
\emph{inside of $\fL$}. 
\end{definition}
Note that if  $H$ satisfies Assumption \ref{assumption:general_ones_on_R_and_CW}, then a rotating loop is indeed a loop.

\begin{figure}[h]
	\begin{center}
\begin{tabular}{cc}
		\includegraphics[width=0.3\textwidth]{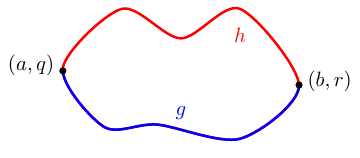}
		& 
		\includegraphics[width=0.3\textwidth]{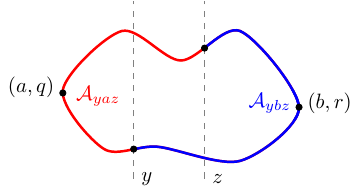}  
		 \\
		(a) & (b) 
\end{tabular}
\caption{(a) Functions $g$ and $h$, (b) The arcs $\cA_{y,a,z}$ and $\cA_{y,b,z}$.}
	\label{fig:A}
	\end{center}
\end{figure}

\begin{lemma}
\label{lemma:partial_x_ne_0_then_less_or_larger}
Let $H$ satisfy Assumption \ref{assumption:general_ones_on_R_and_CW}.

Let $E\in \R$. Suppose that $A$ is nonempty, relatively compact  and a connected component of $H^{-1}(-\infty,E)$ and that $\partial A$ is a connected component of $H^{-1}(\{E\})$. Then $\partial A$ is a loop and with $a,b,g,h$ as in Definition \ref{def:loop}, the functions $g$ and $h$ are $C^1$ on $(a,b)$. 
We write $q:= g(a)=h(a)$ and $r:= g(b) = h(b)$. 

Then $\partial_p H(a,q) = \partial_p H(b,r) =0$ and 
\begin{align}
\label{eqn:p_derivative_H}
&\partial_p H(x,h(x))>0, \quad  \partial_p H(x,g(x))<0 \qquad \mbox{ for } x\in (a,b), \\
\label{eqn:partial_xH_aq}
& \partial_x H(a,q) \ne 0\quad \Longrightarrow \quad \partial_x H(a,q) < 0, \\
\label{eqn:partial_xH_br}
&  \partial_x H(b,r) \ne 0\quad \Longrightarrow \quad \partial_x H(b,r) > 0. 
\end{align}
 $\fL$ is a rotating loop if\footnote{Actually ``if and only if''; as we don't use this we leave out the proof.} the inside of $\fL$ is a set $A$ as above with
\begin{align}
\label{eqn:partial_x_endpoints}
\partial_x H(a,q) \ne 0 \quad \mbox{ and } \quad \partial_x H(b,r) \ne 0.  
\end{align}
\end{lemma}
\begin{proof}
 
Let $a$ and $b$ be the smallest and largest element in the set $\{ x: \exists p, (x,p) \in \overline A \}$, respectively. By strict convexity of $p  \mapsto H(x,p)$, we  have $\partial A = \{ (x,p) \in \overline A: H(x,p) =E\}$. This strict convexity, together with the implicit function theorem, implies the existence of $g,h$ as in the first part of Definition \ref{def:loop}, which are $C^1$ on $(a,b)$ and it implies \eqref{eqn:p_derivative_H}.

	We prove \eqref{eqn:partial_xH_aq}, the proof of \eqref{eqn:partial_xH_br} being similar.
Suppose that $\partial_x H(a,q) \ne 0$. 
By the implicit function theorem, $\partial A$ can be described by a $C^1$ function around $(a,q)$, i.e., there exists a $\delta>0$ and a $C^1$ function $j: (q-\delta, q +\delta) \rightarrow [-1,1]$ such that $\{ (j(p),p) : p \in (q-\delta, q +\delta)\} = \{ (x,p) \in \partial A : p \in (q-\delta, q +\delta) \}$ and 
\begin{align}
j'(p) = - \frac{ \partial_p H(j(p),p)}{\partial_x H(j(p),p)} \qquad (p \in (q-\delta,q+\delta)). 
\label{eqn:derivative_j}
\end{align}
Note that $j'(p) =0$ if and only if $p=q$. 

Therefore for $p \in (q,q+\delta)$ we have $j'(q) >0$ and thus by \eqref{eqn:derivative_j} and \eqref{eqn:p_derivative_H}  $\partial_x H(j(p),p)<0$. Then $\partial_x H(a,q)<0$ follows by taking a limit. 

	 We will now show that $\partial_x H(a,q) \ne 0$ implies that one has a rotation around $a$ in the following sense. 
	Let $y,z \in (a,b)$ and $\cA_{y,a,z}$ be the arc from $y$ to $z$ via $a$ (see Figure \ref{fig:A}(b)), i.e., $\cA_{y,a,z} :=  \{ (x,g(x)) : x\in (a,y)\} \cup \{(a,q)\} \cup \{(x,h(x)) : x\in (a,z)\}$. 
We show that every point in the arc passes thought the point $(z,h(z))$ in a finite and bounded time: 
	We show that there exists a $t_a$ such that for all $(x,p) \in \cA_{y,a,z}$ there exists a $t_0< t_a$ such that 
		\begin{align*}
		(X_t^{x,p},P_t^{x,p}) \in \cA_{y,a,z} \mbox{ for all } t<t_0, \qquad 
		(X_{t_0}^{x,p},P_{t_0}^{x,p}) = (z,h(z)). 
		\end{align*}
	Let $q_1,q_2 \in (q-\delta,q+\delta)$ be such that $q_1<q<q_2$. 
	We may assume that $x_1:= j(q_1) < y$ and $x_2:= j(q_2)<z$. 
	There exists an $\epsilon>0$ such that 
	\begin{align*}
	\partial_p H(x,g(x)) < \epsilon \qquad (x \in [x_1,y]), \\
	- \partial_x H(j(p),p)>\epsilon \qquad (p \in [q_1,q_2]), \\
	\partial_p H(x,h(x)) >\epsilon \qquad (x \in [x_2,z]). 
	\end{align*}
	Therefore there exist $t_1,t_2,t_3$ such that 
	for all $x \in [x_1,y]$ 
	there exists a $t_0 \le t_1$ such that $X_t^{x,g(x)}> x_1$ for $t<t_0$, $X_{t_0}^{x,g(x)}=x_1$; 
	for all $x\in [x_2,z]$  there exists a $t_0 \le t_2$ such that $X_t^{x,h(x)}< z$ for $t<t_0$, $X_{t_0}^{x,h(x)}=z$; 
	and  for all $p \in [q_1,q_2]$ there exists a $t_0\le t_3$ such that $P_t^{j(p),p}< q_2$ for $t<t_0$, $P_{t_0}^{j(p),p}=q_2$. 
	Now by taking $t_a = t_1+t_2+t_3$, we leave it to the reader to check that the claim follows.

Similarly, one shows that $\partial_x H(b,r) \ne 0$ implies rotation around $b$ (along arcs $\cA_{y,b,z}$, see Figure \ref{fig:A}(b)). 
Therefore  it follows that $\fL = \partial A$ is a rotating loop if $\partial_x H(a,q) \ne 0$ and $\partial_x H(b,r) \ne 0$. 
\end{proof}

\begin{theorem}
\label{theorem:above_and_below_lines}
Let $H$ satisfy Assumption \ref{assumption:general_ones_on_R_and_CW}. 
Suppose that $\fL$ is a loop, $\cG \cap \fL \ne \emptyset$, and that $(x_1,p_1)$ and $(x_2,p_2)$ with $x_1 \le x_2$ are those points in $\cG \cap \fL$  such that $\{(x,p) \in \cG : x<x_1\} \cap \fL = \emptyset $ and $\{(x,p) \in \cG : x>x_2\} \cap \fL = \emptyset$. 
Let $x_0 \in [x_1,x_2]$. 
\begin{enumerate}
\item
\label{item:above_line_crossed}
 If  $t_1\ge 0$ is  such that  $ \left(X_{t_1}^{x_1,I_0'(x_1)}, P_{t_1}^{x_1,I_0'(x_1)}\right) = (x_0,h(x_0))$, then for all $t \ge t_1$ 
$\cG_t$ intersects the half-line $\{x_0\} \times [h(x_0),\infty)$. 
\item 
\label{item:below_line_crossed}
 If  $t_2\ge 0$ is  
 such that  $ \left(X_{t_2}^{x_2,I_0'(x_2)}, P_{t_2}^{x_2,I_0'(x_2)}\right) = (x_0,g(x_0))$, then for all $t \ge t_2$ 
$\cG_t$ intersects the half-line $\{x_0\} \times (-\infty,g(x_0)]$. 
\end{enumerate}

Consequently, if $\fL$ is a rotating loop, then there exists a $t_0$ such that $\cG_t$ contains an overhang for all $t\ge t_0$. 

\end{theorem}

\begin{proof}
We prove \ref{item:above_line_crossed} only, as the proof of \ref{item:below_line_crossed} is similar. 
Suppose that $t_1$ is as in \ref{item:above_line_crossed}. 
Write L for the half-line $\{x_0\} \times [h(x_0), \infty)$ and
write $\Theta: [0,\infty) \times [\partial_-, \partial_+] \rightarrow \fS$ for the function given by $\Theta(t,x) = \Psi(t,x,I_0'(x))$ (with the convention that $I_0'$ is defined in $\partial_-$ and $\partial_+$ as $-\infty$ and $\infty$, respectively) where $\Psi$ is as in Proposition \ref{proposition:extension_of_flow_with_bdr}. 
Note that $\Theta$ is continuous by continuity of $\Psi$ and because $I_0\in C^{1,\partial}(\K)$. 
We prove that  for all $t \ge t_1$  there exists an $x\in (\partial_-,x_1]$ such that $\Theta(t,x) \in L$.

The idea is that the curve $\Theta(t,[\partial_-,x_1])$ is pulled through the line $L$ by the rotation of $\Phi_t(x_1)$ over $\fL$ , and as it is connected to $(\partial_-,-\infty)$ it will be connected via $L$ for all larger times. 
We use an argument using homotopy theory 
to prove this. 

Let $\sigma	 : [0,1] \rightarrow [\partial_-,x_1]$ be a continuous function with $\sigma(0) = \partial_-$ and $\sigma(1) = x_1$. 
Define the map $\Gamma : [0,\infty) \times [0,3]$ as follows, 
$\Gamma(t,[0,1])$ is the  set $\Theta(0, [\partial_-,x_1])$; $\Gamma(t,[1,2])$ is the set $\Theta ([0,t],x_1)$, i.e., the Hamiltonian path starting at $(x_1,I_0'(x_1))$ up to time $t$; $\Gamma(t,[2,3))$ is  the set $\Theta(t,[\partial_-,x_1])$ (see also Figure \ref{fig:push_forward}); more precisely,
for $t\in [0,\infty)$ and $s\in [0,1]$ 
\begin{align*}
\Gamma(t,s)& := \Theta(0,\sigma(s)), \\
\Gamma(t,1+s)& := \Theta(ts,x_1), \\
\Gamma(t,2+ s)& := \Theta(t,\sigma(1-s)). 
\end{align*}

\begin{figure}[h]
	\begin{center}
		\includegraphics[width=0.6\textwidth]{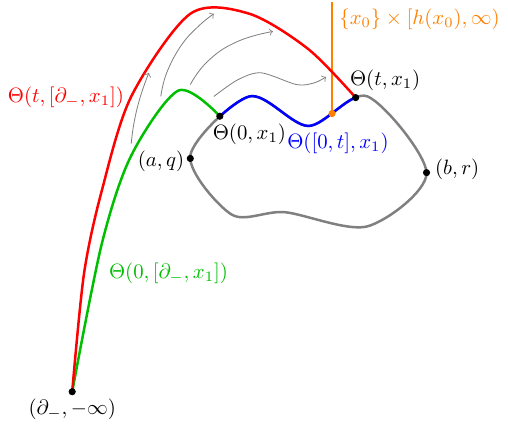}  
\caption{ {\color{red}$\Theta(0, [\partial_-,x_1])$}, {\color{dgreen}$\Theta(t,[\partial_-,x_1])$} and {\color{blue}$\Theta([0,t],x_1)$.}}
	\label{fig:push_forward}
	\end{center}
\end{figure}
$\Gamma$ is continuous and therefore a path homotopy between the paths $\gamma_t: [0,3] \rightarrow \fS\setminus A $ for $t\in [0,\infty)$ where $\gamma_t(s) = \Gamma(t,s)$  and $A$ is the inside of $\fL$. 
Hence $\gamma_t$ is path-homotopic to $\gamma_0$   in the space $\fS \setminus A$ and thus to the single point $(\partial_-,\infty)$. 
That a closed path is homotopic to another one in a topological space, basically means that one can continuously transform one path in that space to the other one, being homotopic to a point means homotopic to a path that stays at that point. 
In \cite[Section 5]{Ru05} one finds the necessary background for homotopy theory. 

We will use the following fact: 
\emph{
A closed path $\gamma :[0,3] \rightarrow \fS \setminus \{(x_0,p_0)\}$ is homotopic to the point $(\partial_-,-\infty)$ if and only if there exist $s_1,s_2 \in [0,3]$, $s_1<s_2$ such that $\gamma(s_1), \gamma(s_2)$ are either both in   $\{x_0\} \times (p_0,\infty)$  or both in  $\{x_0\} \times (-\infty,p_0)$  and $s_1$, $s_2$ are the first and last time such that the path crosses the line $\{x_0\} \times \R$, respectively, i.e., }
\begin{align*}
\gamma([0,s_1)) \cup \gamma((s_2,3]) \subseteq(-\infty,x_0) \times \R \cup \{(\partial_-,-\infty)\}. 
\end{align*}
This fact can be proven as follows; for simplicity with $(0,0)$ instead of $(x_0,p_0)$ and $(-1,0)$ instead of $(\partial_-,-\infty)$. 
Every closed path is homotopic to $s\mapsto (-\cos(2\pi k s), \sin(2\pi k s))$ for some $k\in \Z$ (see, e.g. \cite[Example 5.2.7]{Ru05}). It is straightforward to check that such $s_1$ and $s_2$ do not exist in case $k \ne 0$. 

Let $p_0 \in (g(x_0),h(x_0))$. 
Let $t \ge t_1$. 
As $x_1 < x_0$, $s_1 \ge 1$. 
By the choice of $t_1$ we have $s_1\in [1,2]$ and $\gamma_t(s_1) \in \{x_0\} \times (p_0,\infty)$. 
As the $\partial_p H$ at $(x_0,h(x_0))$ is strictly positive (see \eqref{eqn:p_derivative_H}), the $s_2$ as above cannot be in $[1,2)$. 
Thus $s_2\in [2,3]$, which proves that there exists an   $x\in [\partial_-,x_1]$  such that $\Theta(t,x) \in \{x_0\} \times [h(x_0),\infty)$. 
\end{proof}

The following lemma establishes the existence of a rotating loop $\fL$ with $\cG \cap \fL \ne \emptyset$ under the assumptions made in Theorem \ref{theorem:loss_of_differentiability}\ref{item:loss_rotating}, so that with Theorem \ref{theorem:above_and_below_lines}  the statement of Theorem \ref{theorem:loss_of_differentiability}\ref{item:loss_rotating} is proven.

	\begin{lemma} \label{lemma:regular_rotating_region}
		Assume Assumption \ref{assumption:general_ones_on_R_and_CW} and suppose that $H$ is $C^3$. Suppose that $m_1, m_2 \in \bK^\circ$ are two points such that $m_1 < m_2$ and 
		\begin{enumerate}[label=\textnormal{(\roman*)}]
			\item $\partial_p H(m_1,0) = 0 = \partial_p H(m_2,0)$ and $\partial_p H(x,0) \neq 0$ for all $x \in (m_1,m_2)$,
			\item $\partial_x \partial_p H(m_1,0) \neq 0$ and $\partial_x \partial_p H(m_2,0) \neq 0$.
		\end{enumerate}
		If $ \left\{(x,p) \in (m_1,m_2) \times \bR \, \middle| \, H(x,p) < 0 \right\} \cap \cG \ne \emptyset$, then 
 there exists a rotating loop $\fL$ such that $\cG \cap \fL \ne \emptyset$. 
	\end{lemma}

\begin{proof}

By Lemma \ref{lemma:partial_x_ne_0_then_less_or_larger} it is sufficient to show that there exists an $E<0$ such that the set
		\begin{equation}
		\label{eqn:level_set_suitable_for_A}
A=	 \left\{(x,p) \in (m_1,m_2) \times \bR \, \middle| \, H(x,p) < E \right\},
		\end{equation}
		is as in Lemma \ref{lemma:partial_x_ne_0_then_less_or_larger} such that $\cG \cap A \ne \emptyset$ and \eqref{eqn:partial_x_endpoints} holds.

To find such an $E$ so that $A$ is connected, we consider the function that gives the minimum of $H$ at $x$, $\bfE(x) = \inf_{p\in \R} H(x,p)$. 
We use that a set $A$ as defined in \eqref{eqn:level_set_suitable_for_A} is connected as soon as $(x_1,p_1),(x_2,p_2) \in A$ imply that $\bfE(x) <E$ for all $x\in [x_1,x_2]$. 
Let $p(x) =\argmin_{q\in \bR} H(x,q)$, so that $q = p(x)$ if and only if $\partial_p H(x,q)=0$ and $\bfE(x) = H(x,p(x))$. 
By Assumption (i) $\bfE(m_1) = 0 = \bfE(m_2)$ and $\bfE <0 $ on $(m_1,m_2)$. 
As $H$ is $C^3$, $p$ is $C^2$ by the implicit function theorem and 
\begin{equation}
\label{eqn:derivative_of_p}
p'(x) = - \frac{ \partial_x \partial_p H(x,p(x))}{ \partial_p^2 H(x,p(x))}. 
\end{equation}
Note that $\bfE$ is also $C^2$. 
We show that $\bfE''(m_*) <0$ for $m_*\in \{m_1,m_2\}$. 
We have 
\begin{equation} \label{eqn:derivative_of_E}
\bfE'(x) = \partial_x [ H(x,p(x))]
= \partial_x H(x,p(x)) + \partial_p H(x,p(x)) p'(x) 
= \partial_x H(x,p(x)),
\end{equation}
as $\partial_p H(x,p(x))=0$. Moreover, 
\begin{align*}
\bfE''(x) = \partial_x^2 H(x,p(x))+ \partial_x \partial_p H(x,p(x)) p'(x). 
\end{align*}
As for $m_* \in \{m_1,m_2\}$, $p(m_*)=0$ we have $\partial_x^2 H(m_*,p(m_*))=0$ (as $H(x,0)$ for all $x$), 
Hence by \eqref{eqn:derivative_of_p} and Assumption (ii) 
\begin{align*}
\bfE''(m_*) = - \frac{ [ \partial_x \partial_p  H(m_*,0)]^2}{ \partial_p^2 H(m_*,0)} <0.
\end{align*}
Let $\delta>0$ be such that $\bfE''<0$ on $[m_1, m_1 + \delta]$ and $[m_2-\delta, m_2]$ and thus 
\begin{align}
\label{eqn:bfE_derivative_negative}
\bfE' <0 \mbox{ on } 
[m_1, m_1 + \delta]
\mbox{ and } 
\bfE' >0 \mbox{ on } 
[m_2-\delta, m_2].
\end{align}
Let $\epsilon$ be such that $-\epsilon$ is the maximum of $\bfE$ on $[m_1 + \delta, m_2 - \delta]$. 
Then for all $E\in (0,\epsilon)$ the set \eqref{eqn:level_set_suitable_for_A} is nonempty, connected and its closure is a compact set and a connected component of $H^{-1}(-\infty,E]$. Moreover,  \eqref{eqn:partial_x_endpoints}  is satisfied by \eqref{eqn:derivative_of_E} and \eqref{eqn:bfE_derivative_negative}.
By choosing $E$ small enough, one obtains $ \cG \cap A \ne \emptyset$. 
\end{proof}

\begin{remark}
Suppose that $A$ is a set as in Lemma \ref{lemma:partial_x_ne_0_then_less_or_larger} with $\cG \cap A \ne \emptyset$ and $\partial_x H(a,q)\ne 0$ and $\partial_x H(b,r) = 0$ (or the other way around). 
Then one can also conclude the creation of overhangs for large times as follows. 
There exists a point $(x_0,p_0) \in \cG \cap A$ with a lower energy, i.e., $ H(x_0,p_0) < E$. Therefore $X_t^{x_0,p_0}$ stays bounded away from $b$, 
while by Lemma \ref{lemma:partial_x_ne_0_then_less_or_larger} (see the proof where we show the rotation of the arcs $\cA_{y,a,z}$) $X_t^{x_i,I_0'(x_i)}$ for large $t$ will be arbitrarily close to $b$, where $(x_1,I_0'(x_1)$ and $(x_2,I_0'(x_2)$ are distinct elements of $\cG \cap \partial A$. 
\end{remark}

\section{Proof of Theorem \ref{theorem:application_explicit} and Theorem \ref{theorem:loss_recovery_differentiability} } \label{section:verification_explicit_example}

We end our paper with the application of the developed methods for two main classes of examples. To be precise, we prove Theorem \ref{theorem:application_explicit}. The proof will only be carried out for the $\pm 1$-space-model as the proof of the $\R$-space-model is up to a change of notation the same.

\begin{proof}[Proof Theorem \ref{theorem:application_explicit} for the $\pm 1$-space-model]
	\ref{item:all_times_I_t_equals_I_0} is immediate as $H(x,I_0'(x)) = 0$, and hence $\cG_t = \cG_0$.

\begin{calculations}
\begin{align*}
 I_0(x) & = \frac{1-x}{2} \log (1-x) + \frac{1+x}{2} \log (1+x) - \frac{1}{2}\alpha x^2 + C, 
\\
I_0'(x) &= -\frac12 \log (1-x) +\frac12 \log (1+x) + \alpha x =  \arctanh (x) - \alpha x. 
\\
H(x,I_0'(x)) &= 
\frac{1-x}{2} e^{\beta x} \left[ \frac{1+x}{1-x} e^{-2\alpha x} -1 \right] + \frac{1+x}{2} e^{-\beta x} \left[ \frac{1-x}{1+x} e^{2\alpha x}  -1 \right] \\
&= 
- \frac{1-x}{2} e^{\beta x} 
+ \frac{1+x}{2} e^{-\alpha x} +
-  \frac{1+x}{2} e^{-\beta x} 
+ \frac{1-x}{2} e^{\alpha x}
\end{align*}
\end{calculations}	
	
	\smallskip

	For the proof of \ref{item:all_a,b_alpha_beta_allways_short_time_gibbs} we apply Theorem \ref{theorem:convexity_preserving_implies_diffb_preserv} \ref{item:conv_at_infinity_implies_short_time}: $I_0$ is strictly convex at infinity (see e.g. \eqref{eqn:second_derivative_I_0}) and $H$ preserves order at infinity  by Proposition \ref{proposition:convex_at_infty_explict_functions}\ref{item:order_preserving_for_CW}.

	\smallskip
	
	For the proof of \ref{item:high_temp_starting} we apply Theorem \ref{theorem:convexity_preserving_implies_diffb_preserv} \ref{item:conv_implies_diffb}. For $\alpha \leq 1$, $I_0'$ is strictly increasing (see also \eqref{eqn:derivative_I_0} and \eqref{eqn:second_derivative_I_0}). For $\beta \in [0,1]$ 	by Proposition \ref{proposition:convex_at_infty_explict_functions}\ref{item:order_preserving_for_CW} $H$  preserves order on $ \downq_{0,0}^\circ \cup \{(0,0)\} \cup \upq_{0,0}^\circ$, of which the graph of $I_0'$, $\cG$, is a subset.

	\smallskip
	
	For the proof of \ref{item:heating_up_low_temp} we apply Theorem \ref{theorem:loss_of_differentiability}\ref{item:loss_linearization}. We consider the stationary point $x_0 = 0$. First, note that $m = \partial_x \partial_p H(0,0) = 2(\beta - 1) \neq 0$ and $c = \partial_p^2 H(0,0) = 4$. As the Hamiltonian is $C^\infty$, there is a $C^1$ linearization of the Hamiltonian flow at $(0,0)$ by Theorem \ref{theorem:existence_linearization}. Explicit calculation yields $I_0'(0) = 0$ and $I_0''(0) = (1-\alpha)$. The condition $I_0''(0) < - \frac{2m}{c} \wedge 0$ translates into $1 \vee \beta < \alpha$   and 
\eqref{eqn:time_for_vertical_in_linearized_sytem} into $t_1$ as in (ii). 

\begin{calculations}
Indeed \towil{aanpassen}
\begin{align*}
& - \frac{1}{2m} \log \left(1+\left(\frac{c}{2m}I_0''(0)\right)^{-1} \right) \\
& =  
- \frac{1}{8 (\beta - 1) \cosh(h) } \log \left(1+\left(\frac{8\cosh(h) }{8 (\beta - 1) \cosh(h)}(1- \alpha) \right)^{-1} \right) \\
& = 
- \frac{1}{8 (1- \beta ) \cosh(h) } \log \left(\frac{1-\alpha}{1-\alpha}+ \frac{\beta -1}{1-\alpha} \right) \\
& = 
- \frac{1}{8 (1- \beta ) \cosh(h) } \log \left(\frac{\beta -\alpha}{1-\alpha} \right) 
\end{align*}

\end{calculations}

%

	\smallskip

	For  \ref{item:cooling_down_low_temp}, first we make the following observations. 
	By Remark \ref{remark:derivatives_of_H_and_I_0} we see that $H(x,p) =0$ if and only if $p=0$ or $p = f_\beta(x)$, where 
	\begin{align}
	\label{eqn:f_beta}
	f_\beta(x) = \arctanh(x) - \beta x. 
	\end{align}
	As $\beta>1$ and $\arctanh'(x) = \frac{1}{1-x^2}$, the function $f_\beta$ intersects the $x$-axis at $3$ points, $m_-,0,m_+$.
	As $1<\alpha<\beta$, then function $I_0'$, being $f_\alpha$ (see \eqref{eqn:derivative_I_0}, $f_\alpha$ is as $f_\beta$ in \eqref{eqn:f_beta}), intersects the $x$-axis at $3$ points too, $y_-,0,y_+$ and $m_-<y_-<0<y_+<m_+$ (see Figure \ref{fig:arctanh}). 
		Moreover, the graph of $I_0'$ has a nonempty intersection with $B_-$ and $B_+$, where
\begin{align*}
B_- & = \{(x,p) \in (m_-,0) \times \R : H(x,p) <0 \}, \\ 
B_+ & = \{(x,p) \in (0,m_+) \times \R : H(x,p) <0 \}. 
\end{align*}	

First we will show that $\cG_t$ has no overhang at $\gamma_t(y_-)$, at $0$ and at $\gamma_t(y_+)$. 
Let $x<y_-$, then $I_0'(x) < 0$  and $H(x,I_0'(x)) \ne 0$, so that $P_t^{x,I_0'(x)}<0$ for all $t \ge 0$. With Lemma \ref{lemma:ordering_is_preserved_basic} this implies that $\gamma_t(x) < \gamma_t(y_-)$ for all $t$. 
On the other hand, $\Phi_t(x) \in B_1$ for all $x \in (y_-,0)$ and all $t$, whence Lemma \ref{lemma:ordering_is_preserved_basic} implies $\gamma_t(y_-) <\gamma_t(x)$. So that (by symmetry) we obtain for $x_1,x_2,x_3,x_4$ with $x_1 < y_- < x_2 <0 < x_3 < y_+ < x_4$ and all $t \ge 0$ (as long as $\gamma_t(x_i)$ exists)
\begin{align*}
\gamma_t(x_1) < \gamma_t(y_-) < \gamma_t(x_2) <0 < \gamma_t(x_3) < \gamma_t(y_+) < \gamma_t(x_4) . 
\end{align*}
So indeed, $\cG_t$ has no overhangs at $\gamma_t(y_-)$, $0$ and $\gamma_t(y_+)$. 

By Proposition \ref{proposition:graph_of_derivative_pushforward} we obtain that $I_t$ is non-differentiable at two points $x_- \in (m_-,0)$ and $x_+\in (0,m_+)$, as soon as  $\cG_t$ has an overhang in $B_-$ and $B_+$. 
The existence of a $t_2$ such that for $t \ge t_2$ this is the case follows from Theorem \ref{theorem:loss_of_differentiability}\ref{item:loss_rotating}, as soon as the conditions (i) and (ii) are satisfied for both $(m_1,m_2) = (m_-,0)$ and $(m_1,m_2)=(0,m_+)$. We show that this is indeed the case.

%
%
%

\begin{figure}
	\begin{center}
		\includegraphics[width=0.2\textwidth]{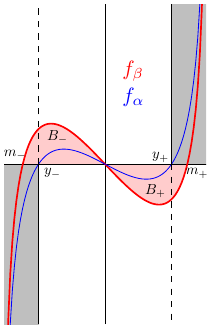}  
\caption{$f_\beta$, $f_\alpha$, their zero's and regions $B_{\pm}$.}
\label{fig:arctanh}
	\end{center}
\end{figure}
	
	By \eqref{eqn:partial_p_H_in_hyperbolic_functions} we see that $\partial_p H(x,0)=0$ if and only if $\arctanh(x) = \beta x$, whence it is clear that (i) of Theorem \ref{theorem:loss_of_differentiability}\ref{item:loss_rotating} is satisfied. 
	By \eqref{eqn:partial_p_partial_x_H_in_hyperbolic_functions} we see that $\partial_p \partial_x H(x,0) =0$ if and only if 
	\begin{align}
	\label{eqn:formula_for_partial_px_H_x0_equals_0}
	1- \beta + \beta x \tanh(\beta x) =0. 
	\end{align}
$f_\beta$ has a local maximum or minimum at $x$ if $f_\beta'(x)=0$, which is the case if and only if $1 - \beta + \beta x^2 =0$. 
At those points $f_\beta$  is not equal to zero. 
By definition of $m_\pm$ we have $1- \beta + \beta m_\pm \tanh(\beta m_\pm)= 1- \beta + \beta m_{\pm}^2$, from which we conclude that \eqref{eqn:formula_for_partial_px_H_x0_equals_0} does not hold for $x= m_{\pm}$. Similarly, we have that \eqref{eqn:formula_for_partial_px_H_x0_equals_0}  does not hold at $x=0$. This proves condition (ii). 
\end{proof}

\begin{proof}[Proof of Theorem \ref{theorem:loss_recovery_differentiability}]
First we make the following observation. 

\underline{Observation} 
Note that 
the Hamiltonian dynamics for the momentum is autonomous: $\dot{P} = \sinh(2P)$. 
Therefore $t_{x_1,p} = t_{x_2,p} <\infty$ for all $p \ne 0$ and $x_1,x_2\in [-1,1]$, and, if  $0<|p_1|<|p_2|$, then $t_{x_1,p_1} < t_{x_2,p_2}$. 
 By Proposition \ref{proposition:convex_at_infty_explict_functions} 
$H$ preserves order $[-1,1] \times \bR$.

\ref{item:unique_zero_derivative_then_later_diff}
Let $x_0$ be the unique solution to $I_0' = 0$. 
Then there exists a $\delta>0$ and a neighbourhood $U$ of $x_0$, such that $I_0'$ is strictly increasing on  $U$ and such that $I_0'(x) < \delta$ implies $x\in U$. 
Let $t^* = t_{0,\delta}$. As we saw in our observation this implies that $t_{z,I_0'(z)} <t^*$ for $z \notin U$ and thus $E_t $ as in Proposition \ref{proposition:hamiltonian_paths_and_graph_structures} is a subset of $U$ for $t \ge t^*$. 
As $H$ preserves order, this implies that $\gamma_t$ is strictly increasing on $E_t$, i.e., $\cG_t$ is a graph. 


\ref{item:loss_recovery_differentiability}
Fix $\alpha>1$. We write 
\begin{align*}
g_\theta (x) = \arctanh(x) - \alpha x - \theta. 
\end{align*}
There exists a $z>0$, such that $-z$ is a local maximum and $z$ is a local minimum of $g_\theta$ for all $\theta$. 
Let $\kappa = g_0(-z) =  - g_0(z)$ (note that the graph of $g_\kappa$ looks like the graph of $I_0'$, the dashed blue graph 
\includegraphics[height=8pt]{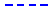}  
 in Figure \ref{fig:I_0_derivative_different_alphas}). 
For the rest of the proof we let $\theta$ be fixed and such that 
$
-\theta > \kappa >0. 
$ 
(By symmetry we could have also treated $\theta > \kappa $.)

	
Let $t_1 = t_{-z,g_\theta(-z)}$. 
We show in STEP 1 that there exists an overhang at a time before $t_1$, 
so that we can take $t_0$ to be be the infimum of all times at which there is an overhang. 
By the choice of $\theta$, $g_\theta = I_0'$ has a unique zero, which is the unique global minimiser of $I_0$. By \ref{item:unique_zero_derivative_then_later_diff} there exists a $t^*$ such that $I_t$ is $C^1$ and so that there is no overhang for $t \ge t^*$. We let $t_2$ be the supremum of all times at which there is an overhang. 
In STEP 2 we show that $I_t$ is non-differentiable for $t\in (t_0,t_1)$. 

\underline{STEP 1} 
By our observation, $t_{z,I_0'(z)}<t_1$ and thus $M:= \lim_{t\uparrow t_1} P_t^{z,g_\theta(z)}<\infty$. By Lemma \ref{lemma:delta_n} there exists a $u \in (0,1)$ such that $\gamma_t(z) \le u$ for all $t\le t_1$.
\begin{calculations}
(take $u= z_n$ for $n$ such that $q_n^+ \ge M$.)
\end{calculations}
As $P_t^{-z,g_\theta(-z)} \rightarrow \infty$ as $t \uparrow t_1$, by Lemma \ref{lemma:delta_n} again (or Lemma \ref{lemma:after_lifetime_in_boundary}) we have $\lim_{t\uparrow t_1} \gamma_t(-z) = 1$. 
Hence there exists a $t<t_1$ such that $\gamma_t(-z) >u \ge \gamma_t(z)$, i.e., there is an overhang. 

\underline{STEP 2}
Let $y<-z$ be the unique point such that $I_0'(y)=0$, 
and let $w> z$ be such that $g_\theta(w) = g_\theta(-z)$. 
We show that for all $t< t_1$, there is no overhang in $(-1,\gamma_t(y))$ and in $(\gamma_t(w),1)$, so that Proposition \ref{proposition:graph_of_derivative_pushforward} implies that $I_t$ is non-differentiable for $t\in (t_0,t_1)$. 
But this follows 
as $g_\theta$ is strictly increasing on $(-1,-z)$ and on $(z,1)$, and $g_\theta <0$ on $(-1,y)$ and $g_\theta > g_\theta(-z)$ on $(w,1)$ 
by the fact that $H$ preserves order. 

\end{proof}



\begin{remark}
\label{remark:t_1_strict_less_t_2}
Note that in our proof $t_1$ is strictly less than $t_2$; 
Indeed, for $t >t_1$ and $t < t_{z,g_\theta(z)}$ the set $\Phi_t(E_t \cap [-1,-z])$ connects $(-1,-\infty)$ with $(1,\infty)$ and the set $\Phi_t(E_t \cap (-z,1])$ is non-empty as  it contains $\Phi_t(z)$. 
\end{remark}

\appendix

\section{The verification of Assumption \ref{assumption:Hamiltonian_for_domain_extension} for the main examples} \label{appendix:assumptions_for_models}

In this appendix, we show that Hamiltonians that satisfy Assumption \ref{assumption:general_ones_on_R_and_CW} (a) in fact satisfy Assumption \ref{assumption:general_ones_on_R_and_CW} (b).

\subsection{Verification for the Curie-Weiss example}

\begin{lemma} \label{lemma:assumptions_jump_imply_assumptions_general1}
		Let $H$ satisfy Assumption  \ref{assumption:general_ones_on_R_and_CW}(a) for the $\pm 1$-space-model,  then $H$ satisfies Assumption \ref{assumption:Hamiltonian_for_domain_extension}.
\end{lemma}

For convenience in the proof, we recall part of \cite[Corollary A.2.7]{CaSi04}.

\begin{lemma}
	\label{lemma:dual_convexity}
	$H$ is as in \eqref{eqn:Hamiltonian_CWmodel} then for all $x\in (-1,1)$ and $v\in \R$ there exists a unique $p$ for which $v= \partial_p H(x,p)$, 
$	\cL(x,v) = p\partial_p H(x,p) - H(x,p)$, 
$    \partial_x	\cL(x, v) = -\partial_x H(x,p)$, 
and $	\partial_v \cL(x,v) = p . $ 
\end{lemma}

\begin{proof}[Proof of Lemma \ref{lemma:assumptions_jump_imply_assumptions_general1}]
	
	\ref{item:assumption_extended_regularity} follows from Assumption \ref{assumption:jump_rates}(b). \ref{item:assumptions_quotient_H}  follows directly from Assumption \ref{assumption:jump_rates}(a).  For \ref{item:assumptions_drift_boundary}, note that $p(x) := \argmin_p H(x,p)$ satisfies $\partial_p H(x,p) = 0$. By Assumption \ref{assumption:jump_rates}(a), we find $\partial_p H(-1,p) > 0$ and $\partial_p H(1,p) < 0$ for all $p$, implying by the continuity of $\partial_p H$ that \ref{item:assumptions_drift_boundary} holds.
	
	In addition, \ref{item:assumptions_convergent_quadrants} follows by an explicit computation using Assumption \ref{assumption:jump_rates}(c). 
	
	We are left to verify \ref{item:assumptions_theta}. Pick some some compact set $K \subseteq (-1,1)$. 

	We consider the function $\theta_K = \theta$ with $\theta(r) = \frac{1}{4} \max\{r \log r, 1 \}$. Property (i) is immediate. 
	
	We proceed with the proof of (ii). 
	Let $M\ge 0$. Note that 
	\begin{align*}
	\sup \{ \theta(r+m) : r \in [0,M \vee 2], m\in [0,M]\} <\infty. 
	\end{align*}
	Hence it is sufficient for (ii) to show that $\theta(r+m)/ \theta(r)$ is bounded from above for $r\ge M\vee 2$ and $m\in [0,M]$. 
	For such $m$ and $r$ we have $m\le M \le r$ and thus 
	\begin{align*}
	\frac{\theta(r+m)}{\theta(r)} = \frac{r+m}{r} \frac{\log(r+m)}{\log r} \le 2 \frac{\log(r+m)}{\log r}.
	\end{align*} 
	The latter ratio is indeed bounded. 
	
	\smallskip
	
	For (iii), by Lemma \ref{lemma:dual_convexity}, it suffices to show the existence of a constant $c$ such that 
	\begin{align}
	\label{eqn:lower_bound_sufficient_for_iii}
	p \partial_p H(x,p) - H(x,p) \ge \theta(|\partial_p H(x,p)|) - c \qquad \mbox{  for all }x\in K, p\in \R. 
	\end{align}
	
	For (iv), by Lemma \ref{lemma:dual_convexity}, it suffices to show the existence of a constant $c$ such that  
	\begin{align}
	|\partial_x H(x,p)| + |p| \le c \theta(|\partial_p H(x,p)|) \qquad \mbox{  for all }x\in K, p\in \R. 
	\end{align}
	
	We will consider the following computations and estimations
	\begin{align*}
	\partial_p H(x,p) &= v_+(x) 2e^{2p} - v_-(x) 2e^{-2p}, \\
	\partial_x H(x,p) &= v_+'(x) \left[e^{2p} - 1\right] + v_-'(x)\left[e^{-2p} -1\right],
	\end{align*}
	By Assumption \ref{assumption:jump_rates}(a) there exists $0<a<b$  such that 
	\begin{align*}
	v_-(x), v_+(x) \in [a,b] \qquad \mbox{ and } \qquad
	|v_-'(x)|, |v_+'(x)| \le b \qquad \mbox{ for all } x\in K,
	\end{align*}
	Then 
	\begin{align}
	| \partial_x H(x,p)| \le 
	\label{eqn:bound_H_x}
	\begin{cases}
	b (e^{2p} -1) +b = b e^{2p} &  \mbox{ if } p \ge 0, \\ 
	b+ b (e^{-2p} -1) = b e^{-2p} & \mbox{ if } p \le 0. 
	\end{cases}
	\end{align}

	Set $\psi(p) = 2p e^{2p} - e^{2p} + 1$. Then $\psi \geq 0$, and 
	\begin{align}
	\notag p \partial_p H(x,p) - H(x,p) & = v_+(x)\psi(p) + v_-(x) \psi(-p) \\
	\label{eqn:ineq_pH_p_min_H} & \geq \begin{cases}
	\psi(p)  v_+(x)  & \text{if } p \geq 0, \\
	\psi(-p)  v_-(x)  & \text{if } p \leq 0.
	\end{cases}
	\end{align}
	Let $p_0 > 0$ be such that 
	\begin{align*}
	a e^{2p} \ge 2, \qquad 
	a e^{2p} \ge b e^{-2p} + \frac{a}{2} e^{2p} \qquad \mbox{ for } p \ge p_0. 
	\end{align*}
	So that for all $x\in K$ (e.g., $v_+(x) e^{2p} \ge b e^{2p} \ge a e^{-2p} +\frac{b}{2} e^{2p} \ge v_-(x) e^{-2p}+ \frac{b}{2} e^{2p} \ge v_-(x) e^{-2p}$ for $p \ge p_0$)
	\begin{align}
	| \partial_p H(x,p)| 
	\label{eqn:ineq_H_p_le}
	& \le 
	\begin{cases}
	2 e^{2p} v_+(x) & \mbox{ if } p \ge p_0 \\
	2 e^{-2p} v_-(x) & \mbox{ if } p \le  - p_0
	\end{cases}, \\
	| \partial_p H(x,p)| 
	\label{eqn:ineq_H_p_ge}
	& \ge a e^{2|p|} \qquad \mbox{ if } |p| \ge p_0. 
	\end{align}
	Note that \eqref{eqn:ineq_H_p_ge} also implies $| \partial_p H(x,p)| \ge 2$ for $|p|\ge p_0$ and thus that $\theta(r) = r\log r$ for $r = | \partial_p H(x,p)| $. 
	
	Using \eqref{eqn:ineq_pH_p_min_H} and \eqref{eqn:ineq_H_p_le} we obtain the following lower bounds
	\begin{align*}
	& | p \partial_p H(x,p) - H(x,p)| - \theta ( | \partial_p H(x,p)| ) \\
	&  \ge 
	\begin{cases}
	\psi(p) v_+(x) - \tfrac12 e^{2p}v_+(x) \log ( 2e^{2p}v_+(x) ) \\
	\qquad \qquad \qquad  \ge a\Big( \psi(p) - p e^{2p} - \tfrac12 e^{2p} \log (2 b) \Big) \qquad \mbox{ if } p \ge p_0, \\
	\psi(-p)  v_-(x) - \tfrac12 e^{-2p}v_-(x) \log ( 2e^{-2p} v_-(x) ) \\
	\qquad \qquad \qquad  \ge a\Big( \psi(-p) + p e^{-p} - \tfrac12 e^{-2p} \log (2b) \Big) \qquad \mbox{ if } p \le -p_0. 
	\end{cases}
	\end{align*}
	As $p \mapsto \psi(p) - p e^{2p} - \tfrac12 e^{2p} \log (2 b)$ is bounded from below for $p\in [0,\infty)$, this implies that there exists a $c>0$ such that the inequality in\eqref{eqn:lower_bound_sufficient_for_iii} holds for $p$ with $|p|\ge p_0$. 
	
	As \eqref{eqn:ineq_pH_p_min_H} implies that $p \partial_p H(x,p) - H(x,p) \ge 0$ and as $\theta ( |\partial_p H(x,p)|)$ is bounded from above for $x\in K$ and $|p|\le p_0$, we can choose $c$ such that \eqref{eqn:lower_bound_sufficient_for_iii} holds for all $p \in \R$.

	By \eqref{eqn:bound_H_x}  we have 
	$| \partial_x H(x,p)| + |p|
	\le b e^{2|p|} + |p|$ for all $x\in K$ and $p\in\R$.  
	
	To conclude (iv), by \eqref{eqn:ineq_H_p_ge} it is sufficient (and not difficult) to see that there exists a $c$ such that
	$be^{2|p|} +|p| \le c \frac12 e^{2|p|} (2|p| + \log (2a)) $ in case $ |p|\ge p_0$ 
	and $be^{2|p|} +|p| \le c$ in case $|p|\le p_0$ (as $\theta \ge 1$). 
\end{proof}

\begin{lemma} \label{lemma:assumptions_jump_imply_assumptions_general2}
Let $H$ satisfy Assumption  \ref{assumption:general_ones_on_R_and_CW}(a) for the $\pm 1$-space-model,  then $H$ satisfies Assumption \ref{assumption:trajectories_in_interior}. 
\end{lemma}

\begin{proof}
	Assumption \ref{assumption:Hamiltonian_for_domain_extension} has been verified in Lemma \ref{lemma:assumptions_jump_imply_assumptions_general1}. 
	We consider $S$ with $S'(x) = \frac{1}{2} \log \frac{v_-(x)}{v_+(x)}$. Note that $\lim_{x \downarrow -1} S'(x) = - \infty$ and $\lim_{x \uparrow 1} S'(x) = \infty$ because of Assumption \ref{assumption:jump_rates} (a). The integration constant of $S$ is chosen by choosing the infimum of $S$ equal to $0$. As $v_+,v_-$ are twice-continuously differentiable and positive on the interior, also $S$ is twice continuously differentiable on the interior.
	
	We leave the calculations for Assumptions \ref{assumption:trajectories_in_interior}(a) and (b) for the reader (for (a) one computes that $\partial_p H(x,0) S'(x) \le 0$ for all $x\in (-1,1)$). 
	
	By (b), $\partial_p H(x,p) = - \partial_p H(x, S'(x)-p)$, thus $\partial_p H(x, \frac12 S'(x))=0$ and so
	\begin{equation*}
	\cL(x,0) = H\left(x, \frac{1}{2} S'(x)\right). 
	\end{equation*}
	
	With this we have 
	\begin{equation*}
	\partial_y\cL(y,0) = \left(\frac{v_+'(y)}{\sqrt{v_+(y)}} - \frac{v_-'(y)}{\sqrt{v_-(y)}} \right)\left(\sqrt{v_+(y)}- \sqrt{v_-(y)}  \right).
	\end{equation*}
	By Assumption \ref{assumption:jump_rates} (a) and (c) it then follows that $\partial_y \cL(y,0)$ converges to $-\infty$ at $-1$ and to $\infty$ at $1$. This shows that (c) is satisfied. 
	
	For (d) we consider the $-1$ boundary, the other case follows similarly. By Cauchy's mean-value theorem, there exists $y \in (-1,x)$ such that
	\begin{equation*}
	\frac{\cL(x,0) - \cL(-1,0)}{S(x) - S(-1)} = \frac{\partial_y \cL(y,0)}{\partial_y S(y)}.
	\end{equation*}
	Hence (using Cauchy's mean-value theorem) 
	\begin{align*}
	\lim_{y\rightarrow -1} \frac{\partial_y \cL(y,0)}{\partial_y S(y)}
	& = v_-'(-1) \sqrt{v_+(-1)} \lim_{y\rightarrow -1} 
	\frac{ - v_-(y)^{-\frac12} }{\tfrac{1}{2} \log v_-(y) - \tfrac{1}{2} \log v_+(y)} \\
	& = v_-'(-1) \sqrt{v_+(-1)} \lim_{y\rightarrow -1} 
	\frac{ v_-'(y) v_-(y)^{-\frac32} }{ v_-' (y) v_-(y)^{-1} -  v_+'(y) v_+(x)^{-1}} = \infty. 
	\end{align*}	
\end{proof}

\subsection{Verification for the Brownian example}

\begin{lemma}
	Let $H$ satisfy Assumption  \ref{assumption:general_ones_on_R_and_CW}(a) for the $\R$-space-model, then $H$ satisfies Assumption \ref{assumption:Hamiltonian_for_domain_extension}.
\end{lemma}

\begin{proof}
	The proof follows the same lines as the proof of Lemma \ref{lemma:assumptions_jump_imply_assumptions_general1} using $\theta(r) = c \max\{|r|^2,1\}$. The calculations in this setting are significantly easier.
\end{proof}

\begin{lemma}
	Let $H$ satisfy Assumption  \ref{assumption:general_ones_on_R_and_CW}(a) for the $\R$-space-model, then $H$ satisfies Assumption \ref{assumption:compact_level_sets_R}.
\end{lemma}

\begin{proof}
	By \cite[Lemma 3.4]{CoKr17}, the one-sided Lipschitz property of $W'$ implies that $\Upsilon(x) = \log(1+\frac{1}{2}x^2)$ is appropriate.
\end{proof}



\bibliography{KraaijBib}{}
\bibliographystyle{alpha}


\end{document}